\documentclass[a4paper,11pt,oneside]{article}
\usepackage[utf8]{inputenc}
\usepackage[T1]{fontenc}
\usepackage[english]{babel}
\usepackage{amsthm}
\usepackage{amssymb}
\usepackage{amsmath}
\usepackage{mathrsfs}
\usepackage{graphicx}
\usepackage{xcolor}
\usepackage{eurosym}
\usepackage{pifont}
\usepackage{fancybox}
\usepackage{multicol}
\usepackage{enumitem}
\usepackage{url}
\usepackage{bbm}
\usepackage[all]{xy}
\usepackage{authblk}
\reversemarginpar
\numberwithin{equation}{section}
\usepackage[top=2.5cm, bottom=3cm, left=1.5cm , right=1.5cm]{geometry}
\usepackage[colorlinks=true,linkcolor=black,citecolor=black,urlcolor=black,hyperfootnotes=false]{hyperref}

\theoremstyle{definition} \newtheorem{Def}{Definition}[section]
\theoremstyle{plain}\newtheorem{Thm}[Def]{Theorem}
\theoremstyle{plain}\newtheorem{Prop}[Def]{Proposition}
\theoremstyle{definition}\newtheorem{Rem}[Def]{Remark}
\theoremstyle{definition}\newtheorem{Ex}[Def]{Example}
\theoremstyle{plain}\newtheorem{Lem}[Def]{Lemma}
\theoremstyle{plain}\newtheorem{Cor}[Def]{Corollary}
\theoremstyle{remark} \newtheorem*{Claim*}{Claim}

\newcommand{\R}{\mathbb{R}}
\newcommand{\N}{\mathbb{N}}
\newcommand{\Z}{\mathbb{Z}}
\newcommand{\C}{\mathbb{C }}
\newcommand{\T}{\mathbb{T}^3}

\newcommand{\tor}[1]{\mathbb{T}^{#1}}

\newcommand{\Prob}{\mathbb{P}}
\newcommand{\E}[1]{\mathbb{E}\left[#1\right]}
\newcommand{\EX}[2]{\mathbb{E}_{#1}\left[#2\right]}

\newcommand{\V}[1]{\mathrm{Var}{\left[#1\right]}}

\newcommand{\eqLaw}{ \overset{\mathcal{L}}{=} }
\newcommand{\Law}{\xrightarrow[]{\mathcal{L}} }

\newcommand{\norm}[1]{\left\lVert#1\right\rVert}
\newcommand{\scal}[2]{\langle #1,#2\rangle}

\newcommand{\eps}{\varepsilon}
\newcommand{\ind}[1]{\mathbb{I}
\left\{#1\right\}}

\newcommand{\notcon}[3]{ {#1} \not\equiv {#2} \pmod {#3}  }

\newcommand{\Nn}{\mathcal{N}_n}
\newcommand{\Hyp}[5]{{}_{#1}F_{#3}\left(#2;#4;#5\right)}

\newcommand{\etr}[1]{\mathrm{etr}\left(#1\right)}

\DeclareMathOperator{\proj}{proj}

\DeclareMathOperator{\Leb}{Leb}
\DeclareMathOperator{\Id}{Id}
\DeclareMathOperator{\bX}{\mathbf{X}}

\DeclareMathOperator{\bT}{\mathbf{T}}

\DeclareMathOperator{\Tr}{tr}

\allowdisplaybreaks

\begin{document}

\title{\normalsize{\textbf{\uppercase{Matrix Hermite polynomials, Random determinants \\ and the geometry of Gaussian fields}}}}
\date{\today}
\author[*]{\textsc{Notarnicola} Massimo}
\affil[*]{\textit{Unit\'e de Recherche en Math\'ematiques, Universit\'e du Luxembourg}}
\maketitle

\abstract{We study generalized Hermite polynomials with rectangular matrix arguments arising in multivariate statistical analysis and the theory of zonal polynomials. We show that these are well-suited for expressing the Wiener-It\^{o} chaos expansion of functionals of the spectral measure associated with Gaussian matrices. In particular, we 
obtain the Wiener chaos expansion of Gaussian determinants of the form 
$\det(XX^T)^{1/2}$ and prove that, in the setting where the rows of $X$ are i.i.d.~centred Gaussian vectors with a given covariance matrix, its projection coefficients admit a geometric interpretation in terms of intrinsic volumes of ellipsoids, thus extending the content of \cite{Kab14} to arbitrary chaotic projection coefficients. Our proofs are based on a crucial relation between generalized Hermite polynomials and generalized Laguerre polynomials. In a second part, we introduce the matrix analog of the classical Mehler's formula for the Ornstein-Uhlenbeck semigroup and prove that matrix-variate Hermite polynomials are eigenfunctions of these operators. As a byproduct, we derive an orthogonality relation for Hermite polynomials evaluated at correlated Gaussian matrices. We  apply our results to vectors of independent arithmetic random waves on the three-torus, proving in particular a CLT in the high-energy regime for a generalized notion of total variation on the full torus. }

\vspace{1cm}


\sloppy
\paragraph{Notation.} For integers $\ell,n\geq 1$, we write $[n]:=\{1,\ldots,n\}$ and $\R^{\ell \times n}$ to indicate the $\ell n$-dimensional vector space of $\ell\times n$ matrices with entries in $\R$ with $\Id_n$ denoting the identity matrix of dimension $n$. We write $\mathcal{P}_{n}(\R)$ for the space of positive-definite matrices of dimension $n$. For $X \in \R^{\ell \times n}$, we denote by $\mathrm{Vec}(X)$ its vectorisation, that is the vector in $\R^{\ell n}$ obtained from $X$ by juxtaposing its columns and $\etr{X}:= e^{\Tr(X)}$, where $\Tr(X)$ is the trace of $X$.
We write  $\phi^{(\ell,n)}$ for the probability density function of $X \in \R^{\ell \times n}$ with i.i.d. standard normal entries, given by 
\begin{gather*}
\phi^{(\ell,n)}(X) = (2\pi)^{-n\ell/2} \etr{-2^{-1}XX^T }.
\end{gather*}
In this case, we write
$X \sim \mathscr{N}_{\ell \times n}(0 , \Id_{\ell}\otimes \Id_n)$ and refer to it as the standard normal matrix distribution. Here, $\otimes$ denotes the usual Kronecker product of  matrices. When $\ell=n=1$, we write $\phi^{(1,1)}=:\phi$ for the standard Gaussian density on $\R$. 

For numerical sequences $\{a_n\},\{b_n\}$, we write $a_n=O(b_n)$ or $a_n \ll b_n$ to indicate that there exists an absolute constant $C>0$ such that $|a_n| \leq C|b_n|$ and $a_n=o(b_n)$ to indicate that  $a_n/b_n\to0$ as $n \to \infty$.
Throughout this paper, we assume that every random object is defined on a common probability space $(\Omega,\mathcal{F},\Prob)$ and write $\E{\cdot}$ and $\V{\cdot}$ for the mathematical expectation and variance with respect to $\Prob$, respectively.

\newpage
\tableofcontents
\section{Introduction}
In   applications to stochastic geometry dealing with the asymptotic analysis of local geometric quantities associated with Gaussian random fields on manifolds, one is often confronted   with expressions involving quantities of the type $F(X)$, where $X$ is a rectangular centred Gaussian matrix and $F$ is a certain \textit{spectral} function, that is, $F$ only depends on the  spectral measure associated with the matrix $XX^T$. For instance, if $Z=\{Z(x):x\in \mathcal{M}\}$ is a $\ell$-dimensional stationary Gaussian field on a manifold $\mathcal{M}$ of dimension $n$ (with $1\leq \ell \leq n$), the \textit{nodal volume} of $Z$ over a region $R \subset \mathcal{M}$ has typically the form
\begin{eqnarray*}
\int_{R} \delta_0(Z(x)) F(J_Z(x)) \mathrm{vol}_{\mathcal{M}}(dx),
\end{eqnarray*}
where $\delta_0$ indicates the Dirac mass at zero, $J_Z(x)$ stands for the Jacobian matrix of $Z$ computed at $x$ and $F(X) = \sqrt{\det{(XX^T)}}$, see for instance \cite{Wig10,MPRW16,MRW17,NPR17,Not20} for some distinguished examples.
While objects of this type are amenable to analysis by Wiener-It\^{o} chaos expansions (which involves in particular the decomposition of $F(J_Z(x))$ into Hermite polynomials having the entries of $J_Z(x)$ as arguments, see the references above), it is to be expected that such a technique will generate combinatorially untractable expressions for large values of the dimensions $\ell$ and $n$ (see for instance \cite{DEL19} and  \cite{Not20}). The aim of this paper is to tackle directly such a difficulty by initiating a systematic study of chaotic expansions for spectral random variables $F(X)$ as above by using \textit{matrix-variate Hermite polynomials}, that is, a collection of orthogonal polynomials with matrix entries which are indexed by partitions of integers, obtained by orthogonalizing matrix monomials of the type $\Tr([XX^T]^s)$ with respect to the law of a Gaussian matrix. We will see that matrix-variate Hermite polynomials inherit the rich combinatorial structure and actually can be defined in terms of \textit{zonal polynomials}  introduced in \cite{J61}, thus allowing one to deduce explicit formulae in any dimension. We now  describe the principal achievements of the present work.
\begin{enumerate}[label = \textbf{(\alph*)}]
\item In \cite{Thangavelu} (see also the related work \cite{Kochneff}), the author studies Hermite expansions of functions of the form $F(x)=f_0(\norm{x})P(x)$ on $\R^n$, where $f_0$ is a function depending only on the norm $\norm{x}$ and $P$ is a harmonic polynomial. In particular, in such a work, the author provides explicit formulae for the projection coefficients associated with the Wiener-It\^{o} chaos expansion of functionals $F$ as above in terms of \textit{Laguerre polynomials} on the real line. In Theorem \ref{ThmWC}, we extend this framework by studying matrix-Hermite expansions of radial functionals of the type $F(X) = f_0(XX^T)$ on matrix spaces. Our results involve \textit{generalized Laguerre polynomials with matrix argument}, thus yielding a natural counterpart to the work by Thangavelu \cite{Thangavelu} in higher dimensions.
\item In \cite{Kab14} (see Theorem 1.1 therein), Kabluchko and Zaporozhets establish a formula for the expected value of Gaussian determinants of the form 
$F(X) = \sqrt{\det(XX^T)}$ in terms of mixed volumes and intrinsic volumes of ellipsoids associated with the covariance matrices of the underlying Gaussian vectors, yielding in particular an expression for the projection of $F$ onto the Gaussian Wiener chaos of order zero associated with $X$ (which corresponds to the expectation). In Theorem \ref{CoeffSigma} and Theorem \ref{InterVol} of the present paper, we substantially extend their framework by considering \textit{arbitrary} projection coefficients of the form $\E{F(X) H_{\kappa}^{(\ell,n)}(X)}$ (where $X$ is a centred Gaussian matrix of dimension $\ell \times n$) associated with such random determinants. Our results can be formulated using integrations on the so-called \textit{Stiefel manifold} (see Theorem \ref{CoeffSigma}), which can subsequently  be interpreted in terms of mixed and intrinsic volumes (see Theorem \ref{InterVol}).
\item In Section \ref{SecOrtho}, we introduce a 
collection of  operators on matrix spaces via a \textit{Mehler-type} formula, whose definition is amenable to that of the classical Ornstein-Uhlenbeck semigroup on the Euclidean space $\R^n$. In Theorem \ref{Action} we provide a characterization of   matrix-Hermite polynomials as the eigenfunctions of these operators, thus yielding a direct analog of the action of the Ornstein-Uhlenbeck semigroup on
classical Hermite polynomials on the real line. 
We subsequently use Theorem  \ref{Action} in order to deduce an intrinsic orthogonality relation between two matrix-Hermite polynomials evaluated in correlated Gaussian matrices. Such a result extends the classical orthogonality relation for matrix-Hermite polynomials as well as the 
case of Hermite polynomials on the real line.
Conjecturally, the objects and techniques introduced in Section \ref{SecOrtho} generate a basis for a special Malliavin Calculus on matrix spaces via the introduction of further operators, such as  Malliavin derivatives, adjoints and generators of the Ornstein-Uhlenbeck semigroup (see e.g. \cite{N06,NP12}). Such a program of study however   largely falls outside the scope of the present paper and is left open for further research.
\item In Section \ref{SecApp}, we apply   our results to the study of the \textit{generalized total variation} of multi-dimensional Gaussian random fields, defined as the integral of the square root of the Gramian determinant of its normalized Jacobian matrix. More specifically, we study the high-energy behaviour of the generalized total variation of multiple independent \textit{Arithmetic Random Waves} on the three-torus. In particular, in Theorem \ref{ARW} we establish its expected mean, an asymptotic law for its variance and a Central Limit Theorem for the suitably normalized total variation. Our arguments rely on the expansion in matrix-Hermite polynomials of the total variation, allowing us to prove that its probabilistic fluctuations are entirely characterized by its projection on the second Wiener chaos. Throughout this application, we also make use of variance expansions of radial functionals by means of its projection coefficients (see Proposition \ref{VarEx}). Our findings are to be compared with Theorem 1 of \cite{PR16}, where the authors prove a CLT for the \textit{Leray measure} of Arithmetic Random Waves on the two-torus. 
\end{enumerate}

The organization of the paper is as follows: In Section \ref{Prel}, we present preliminary notions that will be used in our proofs, notably on zonal polynomials and generalized Laguerre polynomials (Section \ref{SecZP}), 
 polar matrix factorizations (Section \ref{SecPD}) and
tools from integral geometry such as mixed volumes, intrinsic volumes of convex bodies and general facts about ellipsoids (Section \ref{SecGeom}). Our main contributions are presented in  Section \ref{SecMainResults}. Finally, the entire Section \ref{SecProofs} is devoted to the proofs of our results.   

\medskip
\noindent\textit{Acknowledgment.}  The author thanks Prof. Giovanni Peccati for his guidance throughout this work and acknowledges support from the 
Luxembourg National Research Fund PRIDE15/10949314/GSM.

\section{Preliminaries}\label{Prel}
\subsection{Zonal polynomials and generalized Laguerre polynomials}\label{SecZP}
\underline{\textit{Zonal polynomials.}}
Zonal polynomials with matrix argument were introduced in \cite{J61}, using group representation theory, as certain homogeneous symmetric functions of the eigenvalues (also called the \textit{latent roots}) of the matrix. 
We give a brief overview of zonal polynomials and their properties; the reader is referred for instance to the books \cite{BFZP} and \cite{Ch2012} for a thorough introduction to zonal polynomials.
Let us now fix   integers $\ell\geq 1$  and $k\geq 0$. We write $\kappa \vdash k$ to denote a \textit{partition} $\kappa$ of $k$ into no more than $\ell$ integer parts (note that such a notation  does not involve the integer  
$\ell$, whose role should be understood from the context), that is 
\begin{gather*}
\kappa=(k_1,\ldots,k_{\ell}) , \quad k_1\geq k_2\geq \ldots \geq k_{\ell}> 0, \quad k_1+\ldots+k_{\ell}=k .
\end{gather*}
For instance, if $\ell=1$, then $\kappa=(k)$ is the only partition of an integer $k$; if $\ell\geq 2$, then $\kappa = (2)$ and $\kappa=(1,1)$ are the only partitions of $k=2$. Sometimes it is useful to represent the partition $\kappa \vdash k$ as $\kappa=(1^{\nu_1} 2^{\nu_2} \ldots k^{\nu_k})$ to indicate that the integer $j$ occurs with multiplicity $\nu_j$; in particular $\nu_1+2\nu_2+\ldots+k\nu_k=k$. With this notation, we have for instance $(1,1)=(1^2) \vdash 2$ and $ (1,2,3,3)=(1^{1} 2^1 3^2) \vdash 9$.

\medskip
Let  $S \in \R^{\ell \times \ell}$ be a symmetric  matrix with eigenvalues $s_1,\ldots,s_{\ell}$. For an integer $k\geq 1$, we denote by $\mathrm{Pol}_{k}(S)$ the space of homogeneous polynomials of degree $k$ in the $\ell(\ell+1)/2$ variables of $S$. For an invertible matrix $L \in \R^{\ell \times \ell}$, the transformation $S \to LSL^T$ induces a representation $\pi$ of $\mathrm{GL}_{\ell}(\R)$ into the vector space $\mathrm{GL}(\mathrm{Pol}_{k}(S))$ of isomorphisms from $\mathrm{Pol}_{k}(S)$ to itself (\cite[Eq.(A.2.1)]{Ch2012}): 
\begin{gather*}
\pi : \mathrm{GL}_{\ell}(\R) \to 
\mathrm{GL}(\mathrm{Pol}_{k}(S))  \ ; \quad 
L \to \pi(L) \ ,
\end{gather*}
given by 
\begin{gather*}
\pi(L)(P):= P(L^{-1}S(L^{-1})^T)  .
\end{gather*}
It can be shown that $\mathrm{Pol}_k(S)$ can be decomposed as direct sum (\cite[p.297]{Ch2012})
\begin{gather}\label{Eq:Pol}
\mathrm{Pol}_k(S) = \bigoplus_{\kappa \vdash k} V_{\kappa}(S)  ,
\end{gather}
where $\{V_{\kappa}(S): \kappa \vdash k\}$ are irreducible and $\pi$-invariant subspaces.  Since $\Tr(S)^k$ is a homogeneous symmetric polynomial of degree $k$ in the eigenvalues of $S$, it can accordingly be decomposed in the spaces $V_{\kappa}(S)$ as follows (\cite[Eq.(4.3.38)]{BFZP}), 
\begin{gather}\label{defZon}
\Tr(S)^k = (s_1+\ldots+s_{\ell})^k = \sum_{\kappa \vdash k} C_{\kappa}(S),
\end{gather}
where $C_{\kappa}(S)$ denotes the \textit{zonal polynomial} associated with the partition $\kappa$ of $k$, that is, $C_{\kappa}(S)$ is the projection of $\Tr(S)^k $ onto the space $V_{\kappa}(S)$.
Applying \eqref{defZon} with $\ell=1$ gives $C_{(k)}(s)=s^k$, so that zonal polynomials can be interpreted as a  generalization of classical monomials. In particular, evaluating at $s=1$ yields $C_{(k)}(1)=1$. 
Zonal polynomials satisfy a generalized binomial formula (\cite[Eq.(4.5.1)]{BFZP}),  
\begin{gather}\label{Bin}
\frac{C_{\kappa}(S+\Id_{\ell})}{C_{\kappa}(\Id_{\ell})} = \sum_{s=0}^{k} \sum_{\sigma \vdash s}
{\kappa \choose \sigma} \frac{C_{\sigma}(S)}{C_{\sigma}(\Id_{\ell}) }  ,  \quad \kappa \vdash k.
\end{gather}
This relation in particular defines the generalized binomial coefficients ${\kappa \choose \sigma}$. Taking $S=a\Id_{\ell}$ for $a \in \R$ in \eqref{Bin} yields
\begin{gather*}
\frac{C_{\kappa}((a+1)\Id_{\ell})}{C_{\kappa}(\Id_{\ell})} = 
\sum_{s=0}^{k} 
\sum_{\sigma \vdash s}
{\kappa \choose \sigma} \frac{C_{\sigma}(a \Id_{\ell})}{C_{\sigma}(\Id_{\ell}) } ,
\end{gather*} 
so that, using the homogeneity property of zonal polynomials gives 
\begin{gather*}
(a+1)^{k} = 
\sum_{s=0}^{k}  
\sum_{\sigma \vdash s} a^s
{\kappa \choose \sigma}   .
\end{gather*} 
In particular, using the usual binomial formula for real numbers on the left-hand side, one deduces a relation linking classical and generalized binomial coefficients (\cite[Eq.(4.5.2)]{BFZP}):
\begin{gather*}
{k \choose s} 
=  \sum_{\sigma \vdash s}
{\kappa \choose \sigma} .
\end{gather*} 
A table with generalized binomial coefficients up to order $5$ can be found in Table 4.4.1 of \cite{BFZP}. 
For $X\in \R^{\ell\times n}$, zonal polynomials associated with partition $\kappa \vdash k$ and matrix argument $XX^T$ can be decomposed as (\cite[Theorem 4.3.6]{BFZP})
\begin{gather}\label{ZonHom}
C_{\kappa}(XX^T) = \sum_{(1^{\nu_1} 2^{\nu_2} \ldots k^{\nu_k}) \vdash k   } z_{\kappa  \nu}^{(k)} t_1(X)^{\nu_1}\ldots t_k(X)^{\nu_k},
\end{gather}
where 
\begin{gather}\label{ts}
t_{s}(X) := \Tr{([XX^T]^s)}, \quad s\geq 1
\end{gather}
 and $z_{\kappa \nu}^{(k)}$ are numerical constants. 
Writing $t_j:=t_j(X)$, the zonal polynomials associated with partitions up to order $3$ are given by (see e.g. \cite[Table 4.3.1]{BFZP})
\begin{gather*}
C_{(1)}(XX^T) = t_1 \\
C_{(2)}(XX^T) = \frac{1}{3}(t_1^2+2t_2) , \quad C_{(1,1)}(XX^T) = \frac{2}{3}(t_1^2-t_2)\\
C_{(3)}(XX^T) = \frac{1}{15}(t_1^3+6t_1t_2+8t_3) ,\quad 
C_{(2,1)}(XX^T) = \frac{3}{5}(t_1^3+t_1t_2-2t_3) \\
C_{(1,1,1)}(XX^T) = \frac{1}{3}(t_1^3-3t_1t_2+2t_3).
\end{gather*}
In particular, since for every $j \in [k]$, $t_j(X)^{\nu_j}$ is a homogeneous polynomial of degree $2j\nu_j$ in the entries of $X$, it follows from \eqref{ZonHom} that $C_{\kappa}(XX^T)$ is a homogeneous polynomial of degree $2k$ in the entries of $X$,  
\begin{gather}\label{HOM}
C_{\kappa}(XX^T) = \sum_{|\alpha|=2k} z_{\alpha}^{\kappa} \prod_{i=1}^{\ell} \prod_{j=1}^n X_{ij}^{\alpha_{ij}},
\end{gather}
where $\alpha=(\alpha_{ij}) \in \N^{\ell \times n}$ is a multi-index such that $|\alpha|=\sum_{i=1}^{\ell}\sum_{j=1}^n \alpha_{ij}=2k$ and $z_{\alpha}^{\kappa}$ is a numerical constant depending on $\alpha$ and $\kappa$.
Zonal polynomials evaluated at the identity matrix $\Id_{\ell}$ can be computed to be (\cite[Eq.(A.2.7)]{Ch2012})
\begin{gather*}
C_{\kappa}(\Id_{\ell}) 
= 2^{2k}k! \bigg(\frac{\ell}{2}\bigg)_{\kappa}
\frac{\prod_{i<j}^p (2k_i-2k_j-i+j)}{\prod_{j=1}^p (2k_j+p-j)!},
\end{gather*}
where $p=p(\kappa)$ is the number of non-zero parts in $\kappa$, and for $a \in \C$ , $(a)_{\kappa} $ stands for the generalized Pochammer symbol (\cite[Eq.(A.2.4)]{Ch2012})
\begin{gather}\label{PochGen}
(a)_{\kappa} := \prod_{j=1}^{\ell} \bigg(a-\frac{j-1}{2}\bigg)_{k_j} , \quad 
(a)_{n} = a(a+1)\cdots(a+n-1)
\end{gather}
defined in terms of classical  Pochammer symbols $(a)_n$.
The product of two zonal polynomials associated with partitions $\tau \vdash t$ and $\sigma \vdash s$ respectively, is given by (\cite[Eq.(4.3.65)]{BFZP})
\begin{gather}\label{akt}
C_{\tau}(S) C_{\sigma}(S) = \sum_{\kappa \vdash t+s} a_{\tau,\sigma}^{\kappa} C_{\kappa}(S),
\end{gather}
for some uniquely determined coefficients $a_{\tau,\sigma}^{\kappa}$. A table for these coefficients is  found in Table 4.3.2(a) of \cite{BFZP}.
Moreover, for positive-definite matrices $S$ and $T$, zonal polynomials have the property (\cite[Eq.(4.3.18)]{BFZP})
\begin{gather}\label{Inv}
C_{\kappa}(S^{1/2}TS^{1/2})=C_{\kappa}(ST)=C_{\kappa}(TS)=C_{\kappa}(T^{1/2}ST^{1/2}).
\end{gather}

\noindent\underline{\textit{Generalized Laguerre polynomials.}}
For a symmetric matrix $S \in \R^{\ell \times \ell}$, the generalized Laguerre polynomial  of order $\gamma>-1$ associated with a partition $\kappa$ of $k$ and   matrix variable $S$ is defined as (\cite[Eq.(4.6.5)]{BFZP})
\begin{gather}\label{LagZon}
L_{\kappa}^{(\gamma)}(S) = \bigg(\gamma+\frac{\ell+1}{2}\bigg)_{\kappa} C_{\kappa}(\Id_{\ell})
\sum_{s=0}^{k} \sum_{\sigma \vdash s}
{\kappa \choose \sigma} \frac{(-1)^s}{(\gamma+\frac{\ell+1}{2})_{\sigma}}\frac{C_{\sigma}(S)}{C_{\sigma}(\Id_{\ell})}.
\end{gather}
The first Laguerre polynomials  associated with partitions up to order three are listed in (4.6.8) of \cite{BFZP}.
The generalized Laguerre polynomials define a class of orthogonal polynomials on $\mathcal{P}_{\ell}(\R)$ with respect to the weight function $\etr{-R}\det(R)^{\gamma}$, that is, for every integers $k,l\geq0$ and every partitions $\kappa \vdash k, \sigma \vdash l$, one has (\cite[Theorem 4.6.4]{BFZP})
\begin{gather}\label{LagOrt}
\int_{\mathcal{P}_{\ell}(\R)} L_{\kappa}^{(\gamma)}(R) L_{\sigma}^{(\gamma)}(R) \etr{-R}\det(R)^{\gamma} \nu(dR) \notag\\  
= \ind{\kappa=\sigma} \times k! C_{\kappa}(\Id_{\ell})
\Gamma_{\ell}\left(\gamma+\frac{\ell+1}{2}\right) \left(\gamma+\frac{\ell+1}{2}\right)_{\kappa}, \label{LagOrt}
\end{gather} 
where $\nu(dR)$ denotes the Lebesgue measure on $\mathcal{P}_{\ell}(\R)$. Here,
for $a \in \R$, $\Gamma_{\ell}(a)$ denotes the multivariate Gamma function defined by
\begin{gather*}
\Gamma_{\ell}(a):= \pi^{\ell(\ell-1)/4} \prod_{i=1}^{\ell} \Gamma(a -2^{-1}(i-1)) \ , \quad \ell \geq 1,
\end{gather*}
where $\Gamma(\cdot)$ is the usual Gamma function.
A useful formula that we will use at several occasions is the following (see e.g. \cite[Theorem 4.4.1]{BFZP}) 
\begin{gather}\label{intzonal}
\int_{\mathcal{P}_{\ell}(\R)}
\etr{-AR}\det(R)^{t-\frac{\ell+1}{2}}C_{\kappa}(RB) \nu(dR) 
= (t)_{\kappa} \Gamma_{\ell}(t) \det(A)^{-t}C_{\kappa}(BA^{-1}),
\end{gather}
where $A\in \C^{\ell \times \ell}$ is a  complex symmetric matrix with positive real part, $B\in \C^{\ell \times \ell}$ is a  complex symmetric matrix and $t$ is such that $\Re(t)>(\ell-1)/2$.

\subsection{Polar decomposition for matrices}\label{SecPD}
Let $1\leq \ell \leq n$ be integers. For $X=(X_{ij}) \in \R^{\ell \times n}$, we denote by $dX:=(dX_{ij})$ its associated differential matrix. We endow the spaces $\R^{\ell \times n}$ and $\mathcal{P}_{\ell}(\R)$ with the measures
\begin{eqnarray*}
(dX) := \prod_{i=1}^{\ell}\prod_{j=1}^{n} dX_{ij} \ , \quad 
\nu(dX) := \prod_{1\leq i \leq j \leq \ell} dX_{ij}
\end{eqnarray*}
respectively. 
Assuming that the rows of $X$ are linearly independent, the \textit{polar decomposition} of  $X$ is uniquely given by (see for instance  \cite{Downs})
\begin{gather}\label{Pol}
X =   R^{1/2} \cdot U \ , \quad R = XX^T \in \mathcal{P}_{\ell}(\R) \ , \quad 
\qquad U =  (XX^T)^{-1/2}X \in O(n,\ell),
\end{gather}  
where $R^{1/2}$ denotes the positive square root of $R$, that is the unique matrix $B$ such that $B^2=R$. We also define $R^{-1/2}:=(R^{1/2})^{-1}$.
The space $O(n,\ell)$ in \eqref{Pol} denotes the so-called \textit{Stiefel manifold} of matrices $Y \in \R^{\ell \times n}$ such that $YY^T= \Id_{\ell}$, that is, $Y$ has orthonormal rows.  An element of $O(n,\ell)$ is called an \textit{$\ell$-frame in $\R^n$}, see for instance \cite[p.8]{Ch2012}.The matrices $R$ and $U$ in \eqref{Pol} are seen to be the radial part and orientation of $X$, respectively and hence the decomposition $X=R^{1/2}U$
is a generalization of the standard polar factorization for vectors (obtained for $\ell=1$).

\medskip
\noindent\underline{\textit{Haar measure on the Stiefel manifold.} }
The family of Stiefel manifolds $O(n,\ell)$ contains as special cases the $n$-shpere $O(n,1) = \mathbb{S}^{n-1}$ and the orthogonal group $O(n,n) = O(n)$. The space $O(n,\ell)$ is the compact manifold of dimension $n\ell - \ell - \ell(\ell-1)/2$  realized as the homogeneous space $O(n)/O(n-\ell)$. The Stiefel manifold is endowed with a left and right-invariant Haar measure $\mu$, that is, for every $P \in O(n)$ and every $Q \in O(\ell)$, 
\begin{gather*}
\mu(UP) = \mu(U) = \mu(QU) ,
\end{gather*}
for every $U \in O(n,\ell)$. Remark that our notation of $\mu$ is independent of $\ell$ and $n$, and should be understood from the context. We refer the reader  for instance to \cite{Ch2012} or \cite{Muir} for details on the construction of such a measure. 
The total volume of $O(n,\ell)$ is given by (\cite[Eq.(1.4.8)]{Ch2012})
\begin{gather*}
v(n,\ell) := \mu(O(n,\ell)) = \int_{O(n,\ell)} \mu(dU) 
= \frac{2^{\ell}\pi^{n\ell/2}}{\Gamma_{\ell}(n/2)} . 
\end{gather*}
The normalised measure 
\begin{eqnarray}
\tilde{\mu}(dU) := \frac{1}{v(n,\ell)} \mu(dU)
\end{eqnarray}
hence defines a left and right invariant probability measure on $O(n,\ell)$. We call it the \textit{Haar probability measure} on $O(n,\ell)$.

\subsection{Intrinsic volumes, mixed volumes and ellipsoids}\label{SecGeom}

\noindent\underline{\textit{Intrinsic volumes and mixed volumes.}}
We present two important notions from integral geometry: intrinsic and mixed volumes. We mainly follow \cite{Schneider} for this part (see in particular Section 14.2 therein).
For an integer $n\geq1$, we  denote by $\mathbb{K}^n$ the set of convex bodies in $\R^n$. We write $\mathbb{B}_n$ for the unit ball in $\R^n$ and $\mathrm{vol}_n$ for the $n$-dimensional volume measure  in $\R^n$. For $K \in \mathbb{K}^n$ and $\eps >0$, we write 
\begin{gather*}
K^{\eps}:= K+\eps\mathbb{B}_n = \{ x\in \R^n: \mathrm{dist}(x,K) \leq \eps\}
\end{gather*}
for the parallel body of $K$ at distance $\eps$. \textit{Steiner's formula} (\cite[Eq.(14.5)]{Schneider}) asserts that its volume is a polynomial of degree $n$ in $\eps$,  
\begin{gather}\label{Steiner}
\mathrm{vol}_n(K^{\eps}) = \sum_{j=0}^n \eps^{n-j} \kappa_{n-j} V_j(K),
\end{gather}
where the coefficients $\{V_j(K), j=0,\ldots,n\}$ denote the \textit{intrinsic volumes} of $K$. We set $V_j(\emptyset):=0$. For instance, when $n=2$, $V_2(K)$ is the area, $V_1(K)$ is half the boundary length and $V_0(K)$ is the Euler characteristic of $K$. Moreover, for every $n\geq 1$, we have $V_n(K)=\mathrm{vol}_n(K)$, that is, the $n$-th intrinsic volume coincides with the $n$-dimensional volume measure. 
The intrinsic volumes of the unit ball $\mathbb{B}_n$ can be computed to be (\cite[Eq.(14.8)]{Schneider})
\begin{gather*}
V_j(\mathbb{B}_n) = {n \choose j} \frac{\kappa_{n}}{\kappa_{n-j}}, \quad \kappa_{n}= \frac{\pi^{n/2}}{\Gamma(1+n/2)}.
\end{gather*}
For an integer $1\leq j\leq n$, we denote by $G(n,j)$ the Grassmannian of $j$-dimensional linear subspaces of $\R^n$. It carries a unique invariant Haar probability measure $\nu_{n,j}$. One possible way to realize Grassmannians is as the quotient space $G(n,j) = O(n,j)/O(j)$, where two elements $U_1,U_2 \in O(n,j)$ are equivalent if and only if there exists an orthogonal matrix $Q \in O(j)$ such that $U_1= QU_2$, see for instance \cite[p.8-9]{Ch2012}. 
Intrinsic volumes admit a useful integral representation, known as \textit{Kubota's formula} (\cite[Eq.(6.11)]{Schneider}), 
\begin{gather}\label{Kubota}
V_j(K) = {n \choose j}\frac{\kappa_n}{\kappa_j \kappa_{n-j}}\int_{G(n,j)} \mathrm{vol}_j(K|\mathscr{L}) \nu_{n,j}(d\mathscr{L}),
\end{gather}
where $K|\mathscr{L}$ stands for the image of the orthogonal projection of $K$ onto $\mathscr{L} \in G(n,j)$, and integration is with respect to the Haar probability measure on $G(n,j)$.

Let $m \geq 1$ and consider $m$ convex bodies $K_1,\ldots,K_m \in \mathbb{K}^n$. Then, for real numbers $\lambda_1,\ldots, \lambda_m \geq0$, the $n$-dimensional volume of the Minkowski sum $\lambda_1K_1+\ldots \lambda_mK_m$ is a homogeneous polynomial of degree $n$ in the variables $\lambda_1,\ldots,\lambda_m$ (\cite[Eq.(14.7)]{Schneider}), 
\begin{gather*}
\mathrm{vol}_n(\lambda_1K_1+\ldots \lambda_mK_m)
= \sum_{i_1,\ldots,i_n=1}^m V(K_{i_1},\ldots,K_{i_n}) \lambda_{i_1}\cdots \lambda_{i_n}, 
\end{gather*} 
for uniquely determined symmetric coefficients $V(K_{i_1},\ldots,K_{i_n})$. These coefficients are called the \textit{mixed volumes} of the convex bodies $K_{i_1},\ldots,K_{i_n}$.  This formula is a generalization of  Steiner's formula in \eqref{Steiner}.
Whenever we have mixed volumes involving only two distinct convex bodies $K_1$ and $K_2$, we use the short-hand notation 
\begin{gather}\label{Eq:NotationVol}
V(\underbrace{K_1,\ldots,K_1}_{\ell \textrm{ times}}, 
\underbrace{K_2,\ldots,K_2}_{n-\ell \textrm{ times}}) 
=: V(K_1[\ell],K_2[n-\ell]) , \quad \ell \geq1.
\end{gather}
Intrinsic volumes of a convex body $K \in \mathbb{K}^n$ are related to mixed volume by the relation (\cite[Eq.(14.18)]{Schneider})
\begin{gather}\label{IntMix}
V_j(K) = \frac{{n \choose j}}{\kappa_{n-j}} V(K[j],\mathbb{B}_n[n-j]) , \quad j=1,\ldots, n
\end{gather}

\medskip
\noindent\underline{\textit{General facts about ellipsoids.}}
We will need some preliminaries about a particular type of convex bodies, namely ellipsoids, see e.g. \cite{Kab14}.
Let $\Sigma \in \R^{n\times n}$ be a non-singular symmetric matrix. 
We define the   ellipsoid $\mathcal{E}_{\Sigma}$ of $\R^n$ represented by the matrix $\Sigma$,\begin{gather*}
\mathcal{E}_{\Sigma}= \{x \in \R^n: x^T \Sigma^{-1}x \leq 1\},
\end{gather*}
  obtained as an affinity of the unit $n$-dimensional ball $\mathbb{B}_n$, that is $\mathcal{E}_{\Sigma}=\{ \Sigma^{1/2}y : y \in \mathbb{B}_n \} $. In particular, its $n$-dimensional volume is  given by
\begin{gather}\label{volell}
\mathrm{vol}_n(\mathcal{E}_{\Sigma})= \kappa_n \det(\Sigma)^{1/2}.
\end{gather}
For any non-degenerate linear transformation represented by a matrix $A \in \R^{n \times n}$, the ellipsoid $A\mathcal{E}_{\Sigma}= \{ Ax : x \in \mathcal{E}_{\Sigma}\}$
is represented by the matrix $A\Sigma A^T$, that is 
\begin{gather*}
A\mathcal{E}_{\Sigma}= \{  x\in \R^n : x^T (A\Sigma A^T)^{-1} x \leq 1 \} = \mathcal{E}_{A\Sigma A^T}.
\end{gather*} 
Let $\mathscr{L} \in G(n,\ell)$ be a $\ell$-dimensional linear subspace in $\R^n$ and denote by $L \in O(n,\ell)$ any matrix whose rows form an orthonormal basis of $\mathscr{L}$. Then, the image of the orthogonal projection of $\mathcal{E}_{\Sigma}$ onto $\mathscr{L}$, written $\mathcal{E}_{\Sigma}|\mathscr{L}$, is an ellipsoid in $\R^{\ell}$ that is represented by the matrix $L\Sigma L^T \in \R^{\ell\times \ell}$. In particlar, it follows from \eqref{volell} that its $\ell$-dimensional volume is
$\mathrm{vol}_{\ell}(\mathcal{E}_{\Sigma}|\mathscr{L})= \kappa_{\ell}\det(L\Sigma L^T)^{1/2}$.

\section{Main results}\label{SecMainResults}

\subsection{Wiener-chaos expansion of matrix-variate functions}\label{SecWC}

\noindent\underline{\textit{Hermite polynomials on the real line.}}
Let $m\geq 1$ be an integer and $\mathbf{X}=(X_1,\ldots,X_m)$ be a standard $m$-dimensional Gaussian vector. For  $\alpha=(\alpha_1,\ldots, \alpha_m) \in \N^m$, we write $\alpha!:= \alpha_1!\cdots \alpha_m!$ and $|\alpha|:= \alpha_1+\ldots + \alpha_m$ and define the \textit{multivariate Hermite polynomials} associated with the vector $(X_1,\ldots,X_m)$ as the tensor product of univariate Hermite polynomials, that is  
\begin{gather*}
H^{\otimes m}_{\alpha}(X_1,\ldots,X_m):= \prod_{l=1}^{m}   
H_{{\alpha}_l}(X_l),
\end{gather*}
where $H_{\alpha_l}$ denotes the Hermite polynomial of order $\alpha_l$ on the real line. It is well-known that the normalised Hermite polynomials 
$\{(k!)^{-1/2}H_k: k \geq 0\}$ form a complete orthonormal system of $L^2(\phi):=L^2(\R,\phi(z)dz)$ (see e.g. \cite{NP12}). 
This implies that the collection of normalised multivariate Hermite polynomials
\begin{gather}\label{Hm}
\mathbb{H}_{[m]}=\left\{  (\alpha!)^{-1/2}  H^{\otimes m}_{\alpha}  : \alpha \in \N^m \right\}
\end{gather}
form a complete orthonormal system of $L^2(\phi^{\otimes m})$, where $\phi^{\otimes m}$ stands for the standard $m$-dimensional Gaussian measure. In particular, every random variable 
$F \in L^2(\phi^{\otimes m})$ admits a unique decomposition 
\begin{gather}\label{WCVec}
F = \sum_{k \geq 0} \sum_{|\alpha|=k} \widehat{F}(\alpha) H^{\otimes m}_{\alpha} ,
\end{gather}
where 
\begin{gather}\label{CoeffVec}
\widehat{F}(\alpha) := (\alpha!)^{-1} \int_{\R^m} F(x_1,\ldots,x_m)
H^{\otimes m}_{\alpha}(x_1,\ldots,x_m)  \phi^{\otimes m}(dx_1,\ldots,dx_m) 
\end{gather}
denotes the Fourier-Hermite coefficients of $F$ associated with the multi-index $\alpha$. For $k\geq0$, we write 
\begin{gather}\label{Ck}
\mathbf{C}_k^{\mathbf{X}} = \overline{\mathrm{span}_{\R}}\{ H_{\alpha}^{\otimes m}(X_1,\ldots,X_m): |\alpha|=k\}
\end{gather}
for the closed linear subspace of  $L^2(\Prob)$ generated by multivariate Hermite polynomials of cumulative degree $k$. The space $\mathbf{C}_k^{\mathbf{X}}$ is the so-called \textit{$k$-th Wiener chaos} associated with the vector $\mathbf{X}=(X_1,\ldots,X_m)$. We have that $\mathbf{C}_0^{\mathbf{X}}=\R$. For $F \in L^2(\phi^{\otimes m})$, we denote by $\proj(F|\mathbf{C}_k^{\mathbf{X}})$ the projection of $F$ onto $\mathbf{C}_k^{\mathbf{X}}$, that is, \eqref{WCVec} can be rewritten as the $L^2(\Prob)$-converging series
\begin{gather*}
F = \sum_{k \geq 0} \proj(F | \mathbf{C}_k^{\mathbf{X}}). 
\end{gather*}
This decomposition is known as the \textit{Wiener-It\^{o} chaos expansion of $F$.}

\medskip
\noindent\underline{\textit{Matrix-variate Hermite polynomials.}}
Matrix-variate Hermite polynomials on the matrix space 
$\R^{\ell \times n}$ are introduced in \cite{Ch92} and admit an expansion in \textit{zonal polynomials}. More specifically, the matrix-variate Hermite polynomials associated with the partition $\kappa \vdash k$ of an integer $k\geq 0$, written $H_{\kappa}^{(\ell,n)}$, is given by (\cite[Eq. (4.11)]{Ch92}):\begin{gather}\label{HerZon}
H^{(\ell,n)}_{\kappa}(X) = 
k!C_{\kappa}(\Id_{\ell}) \sum_{s=0}^k 
\sum_{\sigma\vdash s} 
\sum_{\tau\vdash k-s} 
\frac{a^{\kappa}_{\tau,\sigma}}{(-2)^{k+s}(k-s)!}
\frac{C_{\sigma}(XX^T)}{s!(\frac{n}{2})_{\sigma}C_{\sigma}(\Id_{\ell})} , 
\end{gather}
where the coefficients $a^{\kappa}_{\tau,\sigma}$ are defined by the relation \eqref{akt} and $(\ell/2)_{\sigma}$ denotes the generalized Pochammer symbol, formally defined in \eqref{PochGen}. Zonal polynomials being generalizations of monomials, the expansion in \eqref{HerZon} is to be compared to the classical expansion of univariate Hermite polynomials in the basis of monomials (see e.g. \cite[p.19]{NP12})
\begin{gather*}
H_k(x) = \sum_{n=0}^{\lfloor k/2 \rfloor} \frac{k!(-1)^n}{n!(k-2n)!2^n}x^{k-2n}.
\end{gather*}
Alternatively, $H^{(\ell,n)}_{\kappa}$ are defined by Rodrigues formula (\cite[Eq.(4.9)]{Ch92})
\begin{gather}\label{Rod}
H^{(\ell,n)}_{\kappa}(X)\phi^{(\ell,n)}(X) =
4^{-k}\bigg(\frac{n}{2}\bigg)_{\kappa}^{-1}
C_{\kappa}(\partial X\partial X^T) \phi^{(\ell,n)}(X), 
\end{gather}
where, for $X = (X_{ij})\in \R^{\ell \times n}$, the differential matrix $\partial X$ is given by 
$\partial X = \left( \frac{\partial}{\partial{X_{ij}}}\right)$.
We note that \eqref{Rod} is a generalization of the classical well-known Rodrigues formula for univariate Hermite polynomials (see for instance \cite[Proposition 1.4.2]{NP12})
\begin{gather}\label{UniRod} 
H_k(x) \phi(x) = (-1)^k \frac{d^k}{dx^k}\phi(x) , \quad k \geq 0.
\end{gather}
Matrix-variate Hermite polynomials are linked to the generalized matrix-variate Laguerre polynomials by the relation (\cite[Eq.(5.16)]{Ch92} and \cite[Eq.(10)]{Hay})
\begin{gather}\label{RelHerLag}
\gamma_{\kappa}\cdot L_{\kappa}^{(\frac{n-\ell-1}{2})}(XX^T)
= H^{(\ell,n)}_{\kappa}(\sqrt{2}X), \quad \gamma_{\kappa}:=(-2)^{-k} \bigg(\frac{n}{2}\bigg)_{\kappa}^{-1}, \quad \kappa \vdash k.
\end{gather}
Moreover, matrix-variate Hermite polynomials are orthonormal on 
$\R^{\ell \times n}$ with respect to the matrix-normal density function $\phi^{(\ell,n)}$, that is (see e.g. \cite[Corollary 3]{Hay}) for every integers $k,l\geq0$ and every partitions $\kappa \vdash k,\sigma \vdash l$, 
\begin{gather}\label{HerOrtho}
\int_{\R^{\ell \times n}} H_{\kappa}^{(\ell,n)}(X)
H_{\sigma}^{(\ell,n)}(X) \phi^{(\ell,n)}(X) (dX)  
= \ind{\kappa=\sigma} \times
4^{-k}\bigg(\frac{n}{2}\bigg)_{\kappa}^{-1} k !C_{\kappa}(\Id_{\ell}) .
\end{gather}
Let now $X \sim \mathscr{N}_{\ell \times n}(0,\Id_{\ell}\otimes\Id_n)$ and write 
$s_1,\ldots,s_{\ell}$ for the eigenvalues of $XX^T$. The \textit{spectral measure} of $XX^T$ associated with the matrix $X$ is the measure
\begin{gather*}
\mu_X(ds):=\sum_{i =1}^{\ell} \delta_{s_i}(ds),
\end{gather*}
supported on the spectrum of $XX^T$, where $\delta_y$ is the Dirac mass at $y$. We write $L^2(\mu_X):=L^2(\Omega,\sigma(\mu_X),\Prob)$ to indicate the subspace of $L^2(\phi^{(\ell,n)}):=L^2(\R^{\ell \times n}, \phi^{(\ell,n)}(X)(dX))$ consisting of those random variables that are measurable with respect to the sigma algebra generated by $\mu_X$. By this, we mean the subspace of $L^2(\phi^{(\ell,n)})$ of random variables that are generated by elements of the type 
\begin{gather}\label{F}
\int_{\R } f(t) \mu_X(dt) =: \mu_X(f).
\end{gather}
Since matrix-variate Hermite polynomials in \eqref{HerZon} admit an expansion into zonal polynomials, they are themselves symmetric functionals of the eigenvalues $s_1 ,\ldots, s_{\ell} $. 
This fact together with the orthogonality relation \eqref{HerOrtho}, implies that the family of normalised matrix-variate Hermite polynomials
\begin{gather}\label{HerONB}
\mathbb{H}_{[\ell \times n]}:=\left\{    c(\kappa)^{-1/2} H_{\kappa}^{(\ell,n)} : \kappa \vdash k, k \geq 0 \right\}, \quad 
c(\kappa):=4^{-k}\bigg(\frac{n}{2}\bigg)_{\kappa}^{-1} k !C_{\kappa}(\Id_{\ell})
\end{gather}
forms an orthonormal system in 
$L^2(\mu_X)$.
In Appendix \ref{App:Complete}, we prove the following Proposition, stating that this system is also complete in $L^2(\mu_X)$.
 \begin{Prop}\label{Prop:Complete}
 The system $\mathbb{H}_{[\ell\times n]}$ is complete in $L^2(\mu_X)$.
 \end{Prop}
Therefore, every $F \in L^2(\mu_X)$ admits a unique decomposition in the basis \eqref{HerONB}, 
\begin{gather}\label{Decomp}
F = \sum_{k\geq 0} \sum_{ \kappa \vdash k} \widehat{F}(\kappa) H_{\kappa}^{(\ell,n)} , 
\end{gather}
where 
\begin{gather}\label{Coeff}
\widehat{F}(\kappa) := c(\kappa)^{-1} \int_{\R^{\ell\times n}} F(X) H_{\kappa}^{(\ell,n)}(X) \phi^{(\ell,n)}(X)(dX)
\end{gather}
is the Fourier-Hermite coefficient of $F$ associated with the partition $\kappa$ and $c(\kappa)$ is as in \eqref{HerONB}.  
To state our result, we introduce some further notation.
For an integer $s\geq 0$ and $X \sim \mathscr{N}_{\ell\times n}(0,\Id_{\ell}\otimes \Id_n)$, we recall the notation  
$
t_{s}(X) := \Tr{([XX^T]^s)} 
$ introduced in \eqref{ts} and define the spaces   
\begin{gather*}
\mathcal{U}_0^X := \R , \quad \mathcal{U}_k^X := \overline{\mathrm{span}_{\R}}\left\{ \prod_{j=1}^m t_{s_j}(X) : s_1+\ldots+s_m \leq k , m \geq 1\right\}, \ k\geq1,
\end{gather*}
where the closure is with respect to $L^2(\mu_X)$.
By construction, we have that $ \mathcal{U}_k^X \subset \mathcal{U}_{k+1}^X$. We let 
\begin{gather*}
\mathbf{U}_{k}^X:=  \mathcal{U}_k^X  \ominus  \mathcal{U}_{k-1}^X  :=  \mathcal{U}_k^X  \cap  (\mathcal{U}_{k-1}^X)^{\perp}, 
\end{gather*}
that is, $\mathbf{U}_k^X$ is the space of those random variables in $\mathcal{U}_k^X$ that are orthogonal in $L^2(\Prob)$ to elements of $\mathcal{U}_{k-1}^X$. 
Expanding matrix-Hermite polynomials into zonal polynomials by \eqref{HerZon} and subsequently zonal polynomials into monomials of the type $t_s(X)$ by \eqref{ZonHom} shows that Hermite polynomials admit an expansion into monomials $t_s(X)$. In particular, since Hermite polynomials are orthogonal in view of \eqref{HerOrtho}, it follows   that
\begin{gather*}
\mathbf{U}_k^X = \overline{\mathrm{span}_{\R}}\{ H_{\kappa}^{(\ell,n)}(X) : \kappa \vdash k \}.
\end{gather*}
The following result links matrix-variate Hermite polynomials with the classical Wiener-It\^{o}  decomposition in  
\eqref{WCVec}. In particular, we establish an explicit formula for projection coefficients associated with radial functionals of the form $F(X)=f_0(XX^T)\in L^2(\mu_X)$ where $X\sim \mathscr{N}_{\ell \times n}(0,\Id_{\ell}\otimes \Id_n)$ in terms of generalized Laguerre polynomials (see Section \ref{SecZP} for definitions). Such a formula is to be compared to \cite{Thangavelu,Kochneff}, where the authors study Hermite expansions of functions of the form $F(x)=f_0(\norm{X})P(x)$ on $\R^n$, where $P$ is a harmonic polynomial.

\begin{Thm}\label{ThmWC}
For integers $1\leq \ell \leq n$, let $X \sim \mathscr{N}_{\ell \times n}(0, \Id_{\ell} \otimes \Id_n)$ and write $\mathbf{X}=\mathrm{Vec}(X)$. 
Then, for every integer $k\geq0$ and every partition $\kappa\vdash k$, we have that $H_{\kappa}^{(\ell,n)}(X)$ is an element of $\mathbf{C}_{2k}^{\mathbf{X}}$ and for every $F \in L^2(\mu_X)$,
\begin{gather}\label{proj2k}
\proj(F|\mathbf{C}_{2k}^{\mathbf{X}}) = \proj(F|\mathbf{U}_k^X) = 
\sum_{\kappa\vdash k} \widehat{F}(\kappa) H_{\kappa}^{(\ell,n)}(X), 
\end{gather}
where $\widehat{F}(\kappa)$ is as in \eqref{Coeff}. In particular, we have that  $\proj(F|\mathbf{C}_{2k+1}^{\mathbf{X}})=0$. Moreover, if $F(X)=f_0(XX^T)$, then \begin{eqnarray}\label{RAD}
\widehat{F}(\kappa) = \frac{1}{2^{n\ell/2}\Gamma_{\ell}(\frac{n}{2})}\frac{(-2)^{k}}{k!C_{\kappa}(\Id_{\ell})}\times\int_{\mathcal{P}_{\ell}(\R)} f_0(R) L_{\kappa}^{(\frac{n-\ell-1}{2})}(2^{-1}R) \etr{-2^{-1}R} \det(R)^{\frac{n-\ell-1}{2}}\nu(dR), 
\end{eqnarray}
where $L_{\kappa}^{(\gamma)}$ denotes the generalized Laguerre polynomial of order $\gamma>-1$ associated with the partition $\kappa$, defined in \eqref{LagZon} and $\nu(dR)$ is the Lebesgue measure on $\mathcal{P}_{\ell}(\R)$.
\end{Thm}

Our proof of Theorem \ref{ThmWC} suggests that, combining the generalized Rodrigues formula \eqref{Rod} with the univariate Rodrigues formula \eqref{UniRod}, matrix-variate Hermite polynomials can be expressed in terms of multivariate Hermite polynomials. For instance, combining \eqref{Rod} with \eqref{UniRod} in the case $\ell=n=1$ (so that for every integer $k\geq0$, 
$\kappa=(k)$ is the only partition of $k$) and writing $\phi=\phi^{(1,1)}$ for the standard Gaussian density function
yields for every $k\geq 0$
\begin{gather*}
H_{(k)}^{(1,1)}(X) \phi(X)
= 4^{-k} \left(\frac{n}{2}\right)_{(k)}^{-1} C_{(k)}([\partial X]^2)\phi(X)\\
= 4^{-k} \left(\frac{n}{2}\right)_{k}^{-1}
 \left(\frac{\partial}{\partial X}\right)^{2k} \phi(X) 
 =4^{-k} \left(\frac{n}{2}\right)_{k}^{-1}
H_{2k}(X)\phi(X),
\end{gather*}
where we used  that $\left(\frac{n}{2}\right)_{(k)}=\left(\frac{n}{2}\right)_{k}, C_{(k)}(a) = a^k$ for $a\in \R$ and the Rodrigues formula for classical Hermite polynomials in \eqref{UniRod}. This shows in   particular that 
\begin{gather*}
H_{(k)}^{(1,1)}(X) = 4^{-k} \left(\frac{n}{2}\right)_{k}^{-1}
H_{2k}(X).
\end{gather*}
Proceeding similarly for arbitrary dimensions $\ell$ and $n$, we compute the first matrix-variate Hermite polynomials associated with partitions of order up to 2  to be 
\begin{eqnarray}
H_{(1)}^{(\ell,n)}(X) &=&  \frac{1}{2n}\sum_{i=1}^{\ell} \sum_{j=1}^{n} H_2(X_{ij}) \ , \notag \\ 
H_{(2)}^{(\ell,n)}(X) &=& \frac{1}{12n(n+2)}
\bigg( 3\sum_{i\in [\ell]}\sum_{j\in [n]} H_4(X_{ij})
+ 3\sum_{i_1\neq i_2\in [\ell]} \sum_{j\in [n]} H_2(X_{i_1j}) H_2(X_{i_2j}) \notag \\
&&+3\sum_{i\in [\ell]} \sum_{j_1\neq j_2\in [n]} H_2(X_{ij_1}) H_2(X_{ij_2})
+ \sum_{i_1\neq i_2\in [\ell]} \sum_{j_1\neq j_2\in [n]} H_2(X_{i_1j_1}) H_2(X_{i_2j_2}) \notag \\
&&+ 2\sum_{i_1\neq i_2\in [\ell]} \sum_{j_1\neq j_2\in [n]} H_1(X_{i_1j_1}) H_1(X_{i_2j_1})   
H_1(X_{i_1j_2}) H_1(X_{i_2j_2}) \bigg) \notag \ , \\
H_{(1,1)}^{(\ell,n)}(X) &=& \frac{1}{6n(n-1)}\bigg(
\sum_{i_1\neq i_2\in [\ell]} \sum_{j_1\neq j_2\in [n]}
H_2(X_{i_1j_1})H_2(X_{i_2j_2}) \notag \\
&&- \sum_{i_1\neq i_2\in [\ell]} 
\sum_{j_1\neq j_2\in[n]} H_1(X_{i_1j_1}) H_1(X_{i_2j_1})   
H_1(X_{i_1j_2}) H_1(X_{i_2j_2}) \bigg). \label{HerPol}
\end{eqnarray}

Combining the content of Theorem \ref{ThmWC} with the orthogonality relation \eqref{HerOrtho}, allows one to derive variance expansions of spectral variables $F(X)\in L^2(\mu_X)$ where $X \sim \mathscr{N}_{\ell \times n}(0,\Id_{\ell}\otimes \Id_n)$ as a converging series in terms of its Fourier-Hermite coefficients.
\begin{Prop}\label{VarEx}
For integers $1\leq \ell \leq n$, let $X \sim \mathscr{N}_{\ell \times n}(0,\Id_{\ell}\otimes \Id_n)$ and $F(X) \in L^2(\mu_X)$. Then, 
\begin{gather*}
\V{F(X)}=\sum_{k\geq 1}\sum_{\kappa \vdash k} 
\frac{4^{k}(\frac{n}{2})_{\kappa}}{k!C_{\kappa}(\Id_{\ell}) } 
\E{F(X) H_{\kappa}^{(\ell,n)}(X)}^2,
\end{gather*}
where the convergence of the series is part of the conclusion.
\end{Prop}

\subsection{Fourier-Hermite coefficients of Gaussian determinants as intrinsic volumes of ellipsoids}\label{SectGeom}
In this section, we consider rectangular Gaussian matrices $X$ and provide the Wiener chaos expansion of determinants of the form $\det(XX^T)^{1/2}$. In \cite{Kab14}, Kabluchko and Zaporozhets consider the case where $X\in \R^{\ell\times n}$ has centred independent rows with respective covariance matrices $\Sigma_1,\ldots,\Sigma_{\ell}$, 
and prove that (see in particular  \cite[Theorem 1.1]{Kab14})
\begin{gather}\label{Kab11}
\E{\det(XX^T)^{1/2}} = \frac{(n)_{\ell}}{(2\pi)^{\ell/2}\kappa_{n-\ell}}V(\mathcal{E}_{\Sigma_1},\ldots,\mathcal{E}_{\Sigma_{\ell}},\mathbb{B}_n,\ldots,\mathbb{B}_n),
\end{gather}
where $V(\mathcal{E}_{\Sigma_1},\ldots,\mathcal{E}_{\Sigma_{\ell}},\mathbb{B}_n,\ldots,\mathbb{B}_n)$ denotes the mixed volume of the ellipsoids $\mathcal{E}_{\Sigma_i}, i=1,\ldots,\ell$ associated with matrices $\Sigma_i$ and $\mathbb{B}_n$ denotes the unit ball in $\R^n$ with volume $\kappa_n= \pi^{n/2}/\Gamma(1+n/2)$. We also refer the reader to \cite[Theorem 3.2]{Vit:91}, where the author proves a similar formula linking the expected absolute determinant of a  matrix with i.i.d. copies of a random vector to the volume of the zonoid associated with the random distribution.

\medskip
In Theorem \ref{InterVol} below, we substantially extend the framework of Kabluchko and Zaporozhets to arbitrary projection coefficients associated with the Wiener chaos expansion of such Gaussian determinants in the case where the rows of $X$ are i.i.d~centred Gaussian vectors with the same covariance matrix $\Sigma$. 

\medskip
Let $\Sigma \in \R^{n \times n}$ be a symmetric positive-definite matrix and  $\{X^{(i)}=(X_1^{(i)},\ldots,X_n^{(i)}):i \in [\ell]\}$ a collection of independent Gaussian vectors with covariance matrix $\Sigma$. We write $X \in \R^{\ell \times n}$ for the matrix whose $i$-th row is $X^{(i)}$. It follows that $X$ has distribution $ \mathcal{N}_{\ell\times n}(0,\Id_{\ell} \otimes \Sigma)$ with density function 
\begin{gather}\label{denSigma}
\phi^{(\ell,n)}_{\Sigma}(X)
= (2\pi)^{-n\ell/2} \det(\Sigma)^{-\ell/2} \etr{- 2^{-1}X\Sigma^{-1}X^T}.
\end{gather} 
As a consequence, the matrix
$X\Sigma^{-1/2}$ has the $\mathcal{N}_{\ell\times n}(0,\Id_{\ell} \otimes \Id_n)$ distribution (see e.g. \cite[Theorem 2.3.10]{Gupta2018}).
Based on the matrix-variate Hermite polynomials $H_{\kappa}^{(\ell,n)}$ and their orthogonality relation with respect to $\phi^{(\ell,n)}$ in \eqref{HerOrtho}, we define 
\begin{gather}\label{HerSigma}
H^{(\ell,n)}_{\kappa}(X; \Sigma) := \det(\Sigma)^{\ell k} H_{\kappa}^{(\ell,n)}(X\Sigma^{-1/2}). 
\end{gather}
In particular, we note that $H_{\kappa}^{(\ell,n)}(\cdot, \Id_n)=H_{\kappa}^{(\ell,n)}(\cdot)$. The following proposition shows that $H^{(\ell,n)}_{\kappa}(\cdot, \Sigma)$ are orthogonal with respect to the density $\phi^{(\ell,n)}_{\Sigma}$ in \eqref{denSigma}.
\begin{Prop}\label{DefSigma}
For every integers $k,l\geq0$ and every partitions $\kappa \vdash k, \sigma \vdash l$, we have
\begin{gather*}
\int_{\R^{\ell \times n}} H^{(\ell,n)}_{\kappa}(X;\Sigma)
H^{(\ell,n)}_{\sigma}(X;\Sigma) \phi^{(\ell,n)}_{\Sigma}(X)(dX)  
= \ind{\kappa=\sigma} \times \det(\Sigma)^{2\ell k}4^{-k}\bigg(\frac{n}{2}\bigg)_{\kappa}^{-1} k!C_{\kappa}(\Id_{\ell}).
\end{gather*}
\end{Prop}
Therefore, the family of normalized polynomials  
\begin{gather}\label{cSigma}
\mathbb{H}_{\Sigma}:= \left\{ c(\kappa; \Sigma)^{-1/2} H_{\kappa}^{(\ell,n)}(\cdot;\Sigma): \kappa \vdash k \geq0\right\}, \quad 
c(\kappa;\Sigma):=\det(\Sigma)^{2\ell k}4^{-k}\bigg(\frac{n}{2}\bigg)_{\kappa}^{-1} k!C_{\kappa}(\Id_{\ell})
\end{gather}
forms a complete orthonormal system of $L^2(\mu_X)$, where $\mu_X$ denotes the spectral measure of $XX^T$ associated with $X$. Hence, for every 
$F \in L^2(\mu_X)$, one has the expansion 
\begin{gather*}
F(X) = \sum_{k\geq 0} \sum_{\kappa\vdash k} \widehat{F}(\kappa; \Sigma) H_{\kappa}^{(\ell,n)}(X; \Sigma),
\end{gather*}
where the projection coefficients are given by
\begin{eqnarray}\label{Fourier}
\widehat{F}(\kappa; \Sigma) 
&=& c(\kappa;\Sigma)^{-1}
\int_{\R^{\ell\times n}} F(X) H^{(\ell,n)}_{\kappa}(X;\Sigma) \phi_{\Sigma}^{(\ell,n)}(X) (dX) \notag \\
&=& c(\kappa;\Sigma)^{-1}
\EX{X}{F(X) H^{(\ell,n)}_{\kappa}(X;\Sigma)}, \quad X \sim \mathscr{N}_{\ell \times n}(0, \Id_{\ell}\otimes \Sigma).
\end{eqnarray}
The next result provides an explicit formula for the projection coefficients $\widehat{F}(\kappa; \Sigma)$ in the special case where $F(X) = \det(XX^T)^{1/2}$. 
\begin{Thm}\label{CoeffSigma}
For integers $1\leq \ell \leq n$ and $\Sigma \in \R^{n\times n}$ positive-definite symmetric, let $X \sim \mathscr{N}_{\ell \times n}(0,\Id_{\ell}\otimes \Sigma)$. Then, $F(X)=\det(XX^T)^{1/2}$ is an element of $L^2(\mu_X)$, and one has the decomposition 
\begin{gather*}
F(X)= \sum_{k\geq0} \sum_{\kappa\vdash k} \widehat{F}(\kappa;\Sigma) H_{\kappa}^{(\ell,n)}(X;\Sigma), 
\end{gather*}
where the Fourier-Hermite coefficients of $F$ are given by the formula
\begin{gather}
\widehat{F}(\kappa; \Sigma) 
= \frac{ (-2)^{k}  }{\det(\Sigma)^{\ell k} k!}\bigg(\frac{n}{2}\bigg)_{\kappa} 
\sum_{s=0}^{k} \sum_{\sigma\vdash s}
{\kappa \choose \sigma} (-1)^s\frac{(\frac{n+1}{2})_{\sigma}}{(\frac{n}{2})_{\sigma}}  \notag \\ 
\times
\det(\Sigma)^{-\ell/2}   2^{\ell/2}  \frac{\Gamma_{\ell}(\frac{n+1}{2})}{\Gamma_{\ell}(\frac{n}{2})}  \int_{O(n,\ell)}   \det(U\Sigma^{-1}U^T)^{-(n+1)/2} \tilde{\mu}(dU). \label{For1}
\end{gather}
\end{Thm}
Here,  ${\kappa \choose \sigma}$ denote the generalized binomial coefficients defined by \eqref{Bin}, and $\tilde{\mu}$ stands for the Haar probability measure on the Stiefel manifold $O(n,\ell)$ of  $\ell$-frames in $\R^n$. 

\medskip
As anticipated, our next result yields a geometric interpretation of the projection coefficients $\widehat{F}(\kappa;\Sigma)$ appearing in \eqref{For1} in terms of mixed volumes and intrinsic volumes of ellipsoids (see Section \ref{SecGeom} for preliminaries on these notions and in particular notation \eqref{Eq:NotationVol}).

\begin{Thm}\label{InterVol}
For integers $1\leq \ell \leq n$ and $\Sigma \in \R^{n\times n}$ positive-definite symmetric, let $X \sim \mathscr{N}_{\ell \times n}(0,\Id_{\ell}\otimes \Sigma)$. Then, for $F(X)=\det(XX^T)^{1/2}$, we have
\begin{eqnarray}
\widehat{F}(\kappa;\Sigma)
&=&
 \mathcal{M}(\kappa;\Sigma, \ell,n)\cdot
 V(\mathcal{E}_{\Sigma}[\ell],\mathbb{B}_n[n-\ell]) \label{For2} \\
&=& \mathcal{M}(\kappa;\Sigma,\ell,n)\cdot
\kappa_{n-\ell}{ n \choose \ell}^{-1} 
V_{\ell}(\mathcal{E}_{\Sigma}) , \label{For3}
\end{eqnarray}
where 
\begin{eqnarray*}
\mathcal{M}(\kappa;\Sigma, \ell,n):=
\frac{ (-2)^{k}  }{\det(\Sigma)^{\ell k} k!}\bigg(\frac{n}{2}\bigg)_{\kappa} 
\sum_{s=0}^{k} \sum_{\sigma \vdash s}
{\kappa \choose \sigma} (-1)^s\frac{(\frac{n+1}{2})_{\sigma}}{(\frac{n}{2})_{\sigma}}  
\frac{(n)_{\ell} }{(2\pi)^{\ell/2} \kappa_{n-\ell}} 
\end{eqnarray*}
and 
where $V(\cdot,\ldots,\cdot)$ and $V_{\ell}(\cdot)$ stand for the mixed and $\ell$-th intrinsic volumes, respectively (see also notation \eqref{Eq:NotationVol}), $\mathbb{B}_n$ denotes the unit ball in $\R^n$ and $\kappa_n=\pi^{n/2}/\Gamma(1+n/2)$ denotes its volume.
In particular, for $\kappa=(0)$, 
\begin{gather}\label{Kab}
\E{\det(XX^T)^{1/2}} = 
\frac{(n)_{\ell}}{(2\pi)^{n\ell/2}\kappa_{n-\ell}}V(\mathcal{E}_{\Sigma}[\ell],\mathbb{B}_n[n-\ell])
= \frac{(n)_{\ell} }{(2\pi)^{\ell/2} }{ n \choose \ell}^{-1} 
V_{\ell}(\mathcal{E}_{\Sigma}).
\end{gather}
\end{Thm}

\begin{Rem}
\begin{enumerate}[label=\textbf{(\alph*)}]
\item
We point out that \eqref{Kab} coincides with \eqref{Kab11} in the case where
$\Sigma_i=\Sigma$ for $i=1,\ldots,\ell$. In this sense, relations \eqref{For2} and \eqref{For3} therefore considerably generalize the content of Theorem 1.1 in \cite{Kab14} to arbitrary chaotic projection coefficients $\hat{F}(\kappa;\Sigma)$ associated with partitions $\kappa$ of order  $k \geq 1$.  
\item Our proof of Theorem \ref{InterVol} suggests the following new relation for intrinsic volumes of ellipsoids 
\begin{gather*}
V_{\ell}(\mathcal{E}_{\Sigma})=
{n\choose \ell}\frac{\kappa_n}{\kappa_{n-\ell}}\det(\Sigma)^{-\ell/2}     
\int_{O(n,\ell)}   \det(U\Sigma^{-1}U^T)^{-(n+1)/2}  \tilde{\mu}(dU)  , \quad 1\leq \ell \leq n,
 \label{Identity}
\end{gather*}
where $\tilde{\mu}$ indicates the Haar probability measure on the Stiefel manifold $O(n,\ell)$. 
\item
In Section \ref{General}, we sketch an attempt to further generalize the findings of Kabluchko  and Zaporozhets to the more general setting where the rows of $X$ are independent with respective covariance matrices $\Sigma_1,\ldots,\Sigma_{\ell}$. As we will explain, we are not successful to adapt our techniques employed in the proof of Theorems \ref{CoeffSigma} and \ref{InterVol} to this more general framework. Such a difficulty may be explained by the fact that the polynomials defined in \eqref{HerOmega} are not easily tractable for matrix calculus, as we have to deal with each row separately.
\end{enumerate}
\end{Rem}
\medskip
The following Corollary is obtained from Theorem \ref{CoeffSigma} applied with $\Sigma=\Id_{n}$, that is, when $X$ has independent rows with independent coordinates. In this case, we have $H_{\kappa}^{(\ell,n)}(X;\Id_n)=H_{\kappa}^{(\ell,n)}(X)$ and $\widehat{F}(\kappa;\Id_n)=\widehat{F}(\kappa)$ as in \eqref{Coeff}. 
\begin{Cor}\label{CoeffId}
For integers $1\leq \ell \leq n$, let $X \sim \mathscr{N}_{\ell \times n}(0,\Id_{\ell}\otimes \Id_n)$. Then, $F(X)=\det(XX^T)^{1/2}$ is an element of $L^2(\mu_X)$, and one has the decomposition 
\begin{gather*}
F(X)= \sum_{k\geq0} \sum_{\kappa\vdash k} \widehat{F}(\kappa) H_{\kappa}^{(\ell,n)}(X), 
\end{gather*}
where the Fourier-Hermite coefficients of $F$ are given by the formula
\begin{gather}\label{ForId}
\widehat{F}(\kappa) 
= 2^{\ell/2}  \frac{\Gamma_{\ell}(\frac{n+1}{2})}{\Gamma_{\ell}(\frac{n}{2})} \frac{ (-2)^{k}  }{ k!}\bigg(\frac{n}{2}\bigg)_{\kappa} 
\sum_{s=0}^{k} \sum_{\sigma\vdash s}
{\kappa \choose \sigma} (-1)^s\frac{(\frac{n+1}{2})_{\sigma}}{(\frac{n}{2})_{\sigma}}.
\end{gather}
In particular, 
\begin{gather}\label{coeff0}
\E{\det(XX^T)^{1/2}}= 2^{\ell/2}\frac{\Gamma_{\ell}(\frac{n+1}{2})}{\Gamma_{\ell}(\frac{n}{2})}.
\end{gather}
\end{Cor}

\begin{Rem}
\begin{enumerate}[label=\textbf{(\alph*)}]
\item 
Combining the contents of Corollary \ref{CoeffId} and Theorem \ref{ThmWC}, we see that  \eqref{ForId} provides the chaotic projection coefficients associated with the Wiener-chaos decomposition of $\det(XX^T)^{1/2}$. In Section \ref{SecApp}, we consider functionals of multi-dimensional Gaussian fields arising in stochastic geometry, that admit a certain integral representation  in terms of Jacobian determinants, and 
effectively use formula \eqref{ForId} to obtain a compact expression of their Wiener-It\^{o} chaos expansions. 
\item  Formula \eqref{coeff0} is to be compared with the definition of $\alpha(\ell,n)$    in \cite[Eq.(1.8)]{Not20}  and in particular with Remark 1.2 (a) therein for a link to the so-called \textit{flag coefficients}
\begin{eqnarray*}
 { n\brack \ell}:={n \choose \ell} \frac{\kappa_{n}}{\kappa_{n-\ell}\kappa_{\ell}},
\end{eqnarray*}
 also appearing in the Gaussian Kinematic formula (see for instance Chapter 13 in \cite{AT09}). In particular, one has that 
 \begin{eqnarray*}
 \E{\det(XX^T)^{1/2}}
 = \frac{\ell! \kappa_{\ell}}{(2\pi)^{\ell/2}} { n\brack \ell}.
 \end{eqnarray*}
\end{enumerate}
\end{Rem}

\subsection{Generalized Ornstein-Uhlenbeck semigroup}\label{SecOrtho}
\subsubsection{A Mehler-type representation}
In this section, we provide the 
equivalent counterpart on matrix spaces of the classical Ornstein-Uhlenbeck semigroup 
$\{P_t:t\geq 0\}$ on $\R$ defined via \textit{Mehler's formula} (see e.g. \cite[Theorem 2.8.2]{NP12})
\begin{gather*}\label{Mehler}
P_tf(x) = \E{f(e^{-t}x+\sqrt{1-e^{-2t}}X_0)} , \quad X_0 \sim \mathscr{N}(0,1), \quad x \in \R,\quad t\geq0.
\end{gather*}
For an integer $d\geq 1$ and $f:\R^d \to \R$, we write 
\begin{eqnarray}\label{Eq:Ptd}
P_t^{(d)} f(x) = \E{f(e^{-t}x+\sqrt{1-e^{-2t}}X_0)}, \quad X_0 \sim \mathscr{N}_d(0, \Id_d), \quad x \in \R^d, \quad t\geq 0
\end{eqnarray}
for the Ornstein-Uhlenbeck operator in dimension $d$, in such a way that $P_t=P_t^{(1)}$. 
We fix integers $1\leq \ell\leq n$, and define the space 
\begin{gather}\label{Piln}
\Pi(\ell,n)= \{ f: \R^{\ell \times n}\to \R: f(XH)=f(X) \textrm{ for every } H \in O(n)\},
\end{gather}
that is, an element of $\Pi(\ell,n)$ is a matrix-variate function that is right-invariant under orthogonal transformations. For a diagonal matrix $A=\mathrm{diag}(a_1,\ldots,a_n)\in \R^{n\times n}$ with $a_1,\ldots,a_n \geq 0$ and $f\in\Pi(\ell,n)$, we introduce the operator
\begin{gather}\label{OU}
\mathcal{O}_{t;A}^{(\ell,n)}f(X) = \E{ \int_{O(n)} f(XHe^{-tA}+X_0(\Id_n-e^{-2tA})^{1/2}) \tilde{\mu}(dH)\bigg|X} , \quad t\geq0
\end{gather}
where the expectation is taken with respect to $X_0 \sim \mathscr{N}_{\ell\times n}(0,\Id_{\ell}\otimes \Id_n)$, for a matrix $M \in \R^{n\times n}$,   
\begin{gather*}
e^{tM} = \sum_{p \geq 0} \frac{1}{p!}(tM)^p 
\end{gather*} 
denotes the matrix exponential of $M$, and $\tilde{\mu}$ indicates the probability Haar measure on the orthogonal group $O(n)$. 

The next result specifies the action of the operators
$\mathcal{O}_{t;A}^{(\ell,n)}$ for generic diagonal matrices $A$ with non-negative entries on the class of matrix-variate Hermite polynomials and  naturally complements the action of $P_t$ on Hermite polynomials on the real line  given by (see e.g.  \cite[Proposition 1.4.2]{NP12})
\begin{gather}\label{ActPt}
P_t H_k(x) = e^{-kt} H_k(x).
\end{gather}
\begin{Thm}\label{Action}
For every diagonal matrix $A=\mathrm{diag}(a_1,\ldots, a_n)\in \R^{n\times n}$ such that $a_1,\ldots,a_n\geq0$, 
 every integer $k\geq0$ and every partition 
 $\kappa \vdash k$, we have that
\begin{gather}\label{OURel}
\mathcal{O}_{t;A}^{(\ell,n)}H_{\kappa}^{(\ell,n)}(X) = \frac{C_{\kappa}(e^{-2tA})}{C_{\kappa}(\Id_n)}H_{\kappa}^{(\ell,n)}(X).
\end{gather}
In particular, the family $\{\mathcal{O}_{t;A}^{(\ell,n)}:t\geq 0\}$ is a semigroup on the class $\Pi(\ell,n)$ if and only if $a_1=\ldots=a_n=a$. More precisely, in this case, $\mathcal{O}_{t;A}^{(\ell,n)}$ coincides with  $P_{at}^{(\ell n)}$ on the class $\Pi(\ell,n)$. 
\end{Thm}
From \eqref{OURel} it becomes clear that the polynomials $H_{\kappa}^{(\ell,n)}$ are eigenfunctions of $\mathcal{O}_{t;A}^{(\ell,n)}$ with respective eigenvalue $ C_{\kappa}(e^{-2tA}) C_{\kappa}(\Id_n)^{-1}$. Moreover, if $F \in L^2(\mu_X)$ admits the expansion \eqref{Decomp}, 
then $\mathcal{O}_{t;A}^{(\ell,n)}F \in L^2(\mu_X)$ and 
\begin{eqnarray*}
\mathcal{O}_{t;A}^{(\ell,n)}F = \sum_{k\geq 0}\sum_{\kappa\vdash k} \widehat{F}(\kappa)\frac{C_{\kappa}(e^{-2tA})}{C_{\kappa}(\Id_n)} H_{\kappa}^{(\ell,n)},
\end{eqnarray*}
that is, the projection coefficients of $\mathcal{O}_{t;A}^{(\ell,n)}F$ are obtained from those of $F$ by multiplying by $C_{\kappa}(e^{-2tA}) C_{\kappa}(\Id_n)^{-1}$. Let us make some remarks about Theorem \ref{Action}. 

\begin{Rem}\label{RemOU}
\begin{enumerate}[label=\textbf{(\alph*)}]
\item 
Using the fact  that $H_{\kappa}^{(\ell,n)}$ is an element of the class $\Pi(\ell,n)$ (as can be seen for instance from \eqref{RelHerLag}),
we deduce from \eqref{OURel} applied with $A=\Id_n$ that
\begin{eqnarray}\label{Eq:Ptclassic}
P_t^{(\ell n)} H_{\kappa}^{(\ell,n)}(X) =\mathcal{O}_{t;\Id_n}^{(\ell,n)} H_{\kappa}^{(\ell,n)}(X)
= \frac{C_{\kappa}(e^{-2t\Id_n})}{C_{\kappa}(\Id_n)} H_{\kappa}^{(\ell,n)}(X)
= e^{-2kt} H_{\kappa}^{(\ell,n)}(X),
\end{eqnarray}
where we used that $C_{\kappa}(e^{-2t\Id_n})=e^{-2kt}C_{\kappa}(\Id_n)$ by homogeneity.
Recalling that   $H_{\kappa}^{(\ell,n)}(X)$ is an element of the $2k$-th Wiener chaos associated with $\mathrm{Vec}(X)$, it is clear that the classical Ornstein-Uhlenbeck semigroup $\{P_t^{(\ell n)}:t\geq0\}$ acts on the entries  $X_{ij}$ of $X$ via the relation  
$P_t^{(\ell n)} H_{\kappa}^{(\ell,n)}(X) = 
e^{-2kt} H_{\kappa}^{(\ell,n)}(X)$, which is consistent with \eqref{Eq:Ptclassic}.  
\item 
Let us assume that $A=\mathrm{diag}(a,\ldots, a), a\geq0$. Then the relation  in \eqref{OURel} reduces to
\begin{gather*}
\mathcal{O}_{t;A}^{(\ell,n)}H_{\kappa}^{(\ell,n)}(X) = e^{-2ta}H_{\kappa}^{(\ell,n)}(X),
\end{gather*}
in view of the relation $C_{\kappa}(e^{-2tA}) = e^{-2ta}C_{\kappa}(\Id_n)$. In particular, from this identity, one can directly verify the semigroup property verified by $\mathcal{O}_{t;A}^{(\ell,n)}$ on matrix-Hermite polynomials, as for every $s,t\geq 0$, 
\begin{eqnarray*}
\mathcal{O}_{t+s;A}^{(\ell,n)} H_{\kappa}^{(\ell,n)}(X) = e^{-2(t+s)a}H_{\kappa}^{(\ell,n)}(X)
= e^{-2ta} e^{-2sa}H_{\kappa}^{(\ell,n)}(X) 
= \mathcal{O}_{t;A}^{(\ell,n)} \mathcal{O}_{s;A}^{(\ell,n)} H_{\kappa}^{(\ell,n)}(X).
\end{eqnarray*}
Combining this relation  with 
\eqref{OURel} in particular suggests the identity
\begin{eqnarray*}
\frac{C_{\kappa}(e^{-2(t+s)A})}{C_{\kappa}(\Id_n)} = \frac{C_{\kappa}(e^{-2tA})C_{\kappa}(e^{-2sA})}{C_{\kappa}(\Id_n)^2},
\end{eqnarray*}
which fails to hold in the case where the diagonal entries of $A$ are not all equal. Indeed, for simplicity a direct computation in the case $\ell=n=2, \kappa=(1), a_1=1, a_2=2$ shows that the left and right-hand sides of the above relation are respectively given by 
\begin{eqnarray*}
\frac{1}{2} [ e^{-2(t+s)} + e^{-4(t+s)}], \quad 
\frac{1}{4} e^{-2t} e^{-4s},
\end{eqnarray*}
which are different. 
\end{enumerate}
\end{Rem}

\subsubsection{An extension of the orthogonality relation for matrix-variate Hermite polynomials}
It is well-known that for jointly standardized Gaussian random variables $X,Y$ such that $\E{XY}=\rho$, the 
univariate Hermite polynomials on the real line satisfy the orthogonality relation (see e.g. \cite[Proposition 2.2.1]{NP12})
\begin{gather}\label{H1}
\E{H_k(X)H_l(Y)}=\ind{k=l} \times k! \rho^{k}.\end{gather}
 Exploiting the action of the operator $\mathcal{O}_{t;A}^{(\ell,n)}$   on matrix-variate Hermite polynomials derived in Theorem \ref{Action} allows us to establish the matrix-counterpart of the orthogonality relation \eqref{H1}
in the setting where the correlation of the Gaussian matrix entries $X$ and $Y$ is reflected in a  matrix $R$.
This is the content of the following Theorem. 
\begin{Thm}\label{OrtA} 
Let $X, X_0\sim \mathscr{N}_{\ell \times n}(0, \Id_{\ell} \otimes \Id_n)$ be independent and $R$ be a deterministic matrix of dimension $n\times n$. Let $Y \eqLaw XR+X_0(\Id_n-R^2)^{1/2}$ in distribution. Then, for every integers $k,l \geq0$ and every partitions $\kappa \vdash k, \sigma\vdash l$, we have
\begin{gather}\label{RelOrtA}
\E{H_{\kappa}^{(\ell,n)}(X) H_{\sigma}^{(\ell,n)}(Y)}
= \ind{\kappa=\sigma}\times 4^{-k}\bigg(\frac{n}{2}\bigg)_{\kappa}^{-1} k ! C_{\kappa}(R^2)\frac{C_{\kappa}(\Id_{\ell})}{C_{\kappa}(\Id_n)} .
\end{gather}
\end{Thm}
Some remarks concerning Theorem \ref{OrtA} are in order.
\begin{Rem}\label{Rem:Ortho}
\begin{enumerate}[label=\textbf{(\alph*)}]
\item By independence of $X$ and $X_0$ and the distributional identity $Y \eqLaw XR+X_0(\Id_n-R^2)^{1/2}$, we have that for every $i,i'\in [\ell], j,j'\in [n]$, 
\begin{eqnarray*}
\E{X_{ij}Y_{i'j'}} = \sum_{k=1}^n\E{X_{ij}X_{i'k}}R_{kj'} = \sum_{k=1}^n \ind{i=i',j=k}R_{kj'}
= \ind{i=i'}R_{jj'},
\end{eqnarray*}
where we used that $X\sim \mathscr{N}_{\ell \times n}(0, \Id_{\ell} \otimes \Id_n)$, yielding that, for every $j,j'=1,\ldots,n, |R_{jj'}|\leq 1$ by   virtue of the Cauchy-Schwarz inequality. The above observation implies that $R$ is necessarily symmetric and positive-semidefinite as a covariance matrix, and therefore has non-negative eigenvalues $r_1,\ldots,r_n$. Note that if $R=\Delta=\mathrm{diag}(r_1,\ldots,r_n)$ is diagonal, we therefore necessarily have $|r_j|\leq 1$ for every $j=1,\ldots,n,$ so that $(\Id_n-R^2)^{1/2}$ is well-defined. Our arguments to prove Theorem \ref{OrtA} are   based on the following general reduction argument: for $f,g\in \Pi(\ell,n)$, writing $R=O\Delta O^T$ with $O\in O(n)$ and $\Delta=\mathrm{diag}(r_1,\ldots,r_n)$,
\begin{gather*}
\E{f(X)g(XO\Delta O^T+X_0O(\Id_n-R^2)^{1/2}O^T)}
= \E{f(XO)g(XO\Delta +X_0O(\Id_n-\Delta^2)^{1/2})}\\
=\E{f(X)g(X\Delta +X_0(\Id_n-\Delta^2)^{1/2})},
\end{gather*}  
where we used the fact that $f(X)=f(XO)$ and $g(XO^T)=g(X)$ since $f,g\in \Pi(\ell,n)$ as well as the fact that $(XO,X_0O)\eqLaw(X,X_0)$, showing in particular that  
 $|r_j|\leq 1$ for every $j=1,\ldots,n$.
\item  We point out two particular cases of Theorem \ref{OrtA}: (1) if  $R=\Id_n$, relation \eqref{RelOrtA} reduces to the orthogonality of Hermite polynomials stated in \eqref{HerOrtho} and 
(2)  when $R= \mathrm{diag}(\rho,\ldots,\rho)$, \eqref{RelOrtA} gives
\begin{gather}
\E{H_{\kappa}^{(\ell,n)}(X) H_{\sigma}^{(\ell,n)}(\rho X+\sqrt{1-\rho^2}X_0)}
= \ind{\kappa=\sigma}\times 4^{-k}\bigg(\frac{n}{2}\bigg)_{\kappa}^{-1} k ! C_{\kappa}(\rho^2 \Id_n)\frac{C_{\kappa}(\Id_{\ell})}{C_{\kappa}(\Id_n)}  \notag \\
= \ind{\kappa=\sigma}\times 4^{-k}\bigg(\frac{n}{2}\bigg)_{\kappa}^{-1} k ! \rho^{2k}C_{\kappa}(\Id_{\ell}), \label{Eq:Rigid}
\end{gather}
where we used homogeneity of zonal polynomials.  For   completeness, in Example \ref{Ex:Rigid} we present three explicit examples of \eqref{Eq:Rigid} for the Hermite polynomials in \eqref{HerPol} by relying on product formulae for univariate Hermite polynomials. 
\item Writing $\mathbf{X} = \mathrm{Vec}(X)$ and $ \mathbf{Y} = \mathrm{Vec}(Y)$, we know 
by Theorem \ref{ThmWC} that $H_{\kappa}^{(\ell,n)}(X) \in \mathbf{C}_{2k}^{\mathbf{X}}$
and $H_{\kappa}^{(\ell,n)}(Y) \in \mathbf{C}_{2k}^{\mathbf{Y}}$. In particular by orthogonality of Wiener chaoses, it is clear that $H_{\kappa}^{(\ell,n)}(X)$ and $H_{\sigma}^{(\ell,n)}(Y)$ are orthogonal in $L^2(\Prob)$ when $k\neq l$. Remarkably relation \eqref{RelOrtA} yields a stronger orthogonality in the sense that, even if $k=l$, the elements $H_{\kappa}^{(\ell,n)}(X)$ and $H_{\sigma}^{(\ell,n)}(Y)$, belonging both to the Wiener chaos of order $2k$,  are orthogonal as soon as $\kappa \neq \sigma$. 
\item 
Combining   relation \eqref{RelHerLag} with \eqref{RelOrtA}, we deduce an extended orthogonality relation for generalized matrix-variate Laguerre polynomials
\begin{gather*}
\E{L_{\kappa}^{(\frac{n-\ell-1}{2})}(2^{-1}XX^T)
L_{\sigma}^{(\frac{n-\ell-1}{2})}(2^{-1}YY^T)}
= \gamma_{\kappa}^{-1}
\gamma_{\sigma}^{-1}
\E{H_{\kappa}^{(\ell,n)}(X)H_{\sigma}^{(\ell,n)}(Y)}\\
= \ind{\kappa=\sigma}\times  \bigg(\frac{n}{2}\bigg)_{\kappa}^{-1} k ! C_{\kappa}(R^2)\frac{C_{\kappa}(\Id_{\ell})}{C_{\kappa}(\Id_n)} ,
\end{gather*}
where we used that 
$\gamma_{\kappa}:=(-2)^{-k} (\frac{n}{2})_{\kappa}^{-1}$, thus extending the orthogonality relation in \eqref{LagOrt} obtained for $R=\Id_n$. 
\end{enumerate}
\end{Rem}

\begin{Ex}\label{Ex:Rigid}
In this example, we explicitly compute the covariance
\begin{eqnarray*}
\E{H_{\kappa}^{(\ell,n)}(X)H_{\sigma}^{(\ell,n)}(Y)}, \quad Y:=\rho X+\sqrt{1-\rho^2}X_0
\end{eqnarray*}
  in the three examples (i) $\kappa=\sigma=(1)$, (ii) $\kappa=(2),\sigma=(1,1)$  and (iii) 
$\kappa=\sigma=(1,1)$
by relying on the explicit expansions of the corresponding matrix-Hermite polynomials in terms of univariate Hermite polynomials in \eqref{HerPol} and moment formulae for products of the latter. Our computations developed below are consistent with \eqref{Eq:Rigid}.
In view of the covariance structure between $X$ and $Y$, we have that $\E{X_{ij}Y_{i'j'}} = \ind{i=i',j=j'} \rho$.
 We start with (i).  Using the  covariance structure together with the  expression for $H_{(1)}^{(\ell,n)}$ in \eqref{HerPol}  yields
\begin{gather*}
\E{H_{(1)}^{(\ell,n)}(X)H_{(1)}^{(\ell,n)}(Y)}
= \frac{1}{4n^2} \sum_{i_1,i_2\in [\ell]}\sum_{j_1,j_2\in [n]} \E{H_2(X_{i_1j_1})H_2(Y_{i_2j_2})}\\
= \frac{\rho^2}{2n^2} \sum_{i_1,i_2\in [\ell]}\sum_{j_1,j_2\in [n]} \ind{i_1=i_2,j_1=j_2} 
= \frac{\rho^2}{2n^2}\ell n = \rho^2 \frac{\ell}{2n},
\end{gather*} 
where we used \eqref{H1}. 
On the other hand, using that $(n/2)_{(1)}=n/2$ and $C_{(1)}(\Id_{\ell})=\Tr(\Id_{\ell})=\ell$ yields from \eqref{Eq:Rigid}
\begin{eqnarray*}
\E{H_{(1)}^{(\ell,n)}(X)H_{(1)}^{(\ell,n)}(Y)}
= 4^{-1} \frac{2}{n} \rho^2 \ell = 
\rho^2 \frac{\ell}{2n},
\end{eqnarray*}
which coincides with the above. Let us now treat (ii). In view of \eqref{HerPol}, we can write 
\begin{eqnarray*}
H_{(2)}^{(\ell,n)}(X) := \frac{1}{12n(n+2)}\sum_{i=1}^5A_i(X), \quad 
H_{(1,1)}^{(\ell,n)}(Y) := 
\frac{1}{6n(n-1)}
(B_1(Y)+B_2(Y)),
\end{eqnarray*}
where $A_1(X),\ldots, A_5(X)$ and $B_1(Y),B_2(Y)$ are the double summations appearing in the respective definitions of $H_{(2)}^{(\ell,n)}(X)$ and  $H_{(1,1)}^{(\ell,n)}(Y)$ (including their multiplicative coefficient). We can thus compute 
\begin{eqnarray}\label{Eq:AB}
\E{H_{(2)}^{(\ell,n)}(X) H_{(1,1)}^{(\ell,n)}(Y)} = \frac{1}{12n(n+2)} \frac{1}{6n(n-1)}
\sum_{i=1}^5\sum_{j=1}^2 \E{A_i(X)B_j(Y)},
\end{eqnarray}
which is a sum of ten terms. First we recall the following relations for jointly standard Gaussian random variables $N_1,N_2,Z_1,Z_2$ such that $\E{N_1N_2}=\E{Z_1Z_2}=0$,
\begin{eqnarray*}
\E{H_4(N_1)H_2(Z_1)H_2(Z_2)} &=&
24\E{N_1Z_1}^2\E{N_1Z_2}^2,\\
\E{H_2(N_1)Z_1Z_2} &=& 2\E{N_1Z_1}\E{N_1Z_2},\\
\E{N_1N_2Z_1Z_2} &=& 
\E{N_1Z_1}\E{N_2Z_2}+
\E{N_1Z_2}\E{N_2Z_1}. 
\end{eqnarray*}
Combining these relations with the covariance structure between $X$ and $Y$, one verifies that 
\begin{eqnarray*}
\E{A_i(X)B_j(Y)} = 0, \quad  \forall (i,j)\notin \{(5,2),(4,1)\} 
\end{eqnarray*}
and 
\begin{eqnarray*}
\E{A_4(X)B_1(Y)} = -\E{A_5(X)B_2(Y)}
= 8\ell(\ell-1)n(n-1)\rho^4,
\end{eqnarray*}
implying in particular that 
$\E{H_{(2)}^{(\ell,n)}(X) H_{(1,1)}^{(\ell,n)}(Y)} =0$ in view of \eqref{Eq:AB}. Proceeding similarly for example (iii), we write
\begin{eqnarray*}
\E{H_{(1,1)}^{(\ell,n)}(X)H_{(1,1)}^{(\ell,n)}(Y)}
= \frac{1}{36n^2(n-1)^2}\sum_{i,j=1}^2\E{B_i(X)B_j(Y)}
\end{eqnarray*} 
where $B_1$ and $B_2$ are as above, for which we compute 
\begin{eqnarray*}
&&\E{B_1(X)B_1(Y)} = \E{A_4(X)B_1(X)} = 8\ell (\ell-1)n(n-1)\rho^4, \\ 
&&\E{B_1(X)B_2(Y)} = \E{B_2(X)B_1(Y)} = \E{A_4(X)B_2(Y)} = 0,\\
&&\E{B_2(X)B_2(Y)} = 
-\frac{1}{2}\E{A_5(X)B_2(Y)}
=4 \ell (\ell-1)n(n-1)\rho^4, 
\end{eqnarray*}
where $A_4$ and $A_5$ are the terms appearing in $H_{(2)}^{(\ell,n)}$. Summing these terms yields
\begin{eqnarray}\label{Eq:11}
\E{H_{(1,1)}^{(\ell,n)}(X)H_{(1,1)}^{(\ell,n)}(Y)}
= \frac{1}{36n^2(n-1)^2}\sum_{i,j=1}^2\E{B_i(X)B_j(Y)}
= \frac{1}{3n(n-1)}\ell(\ell-1)\rho^4.
\end{eqnarray}
On the other hand, computing 
$(n/2)_{(1,1)} = n(n-1)/4$ and $ C_{(1,1)}(\Id_{\ell}) = \frac{2}{3}\ell(\ell-1)$ yields from \eqref{Eq:Rigid}
\begin{eqnarray*}
\E{H_{(1,1)}^{(\ell,n)}(X)H_{(1,1)}^{(\ell,n)}(Y)}
=4^{-2}\bigg(\frac{n}{2}\bigg)_{(1,1)}^{-1} 2 ! \rho^{4}C_{(1,1)}(\Id_{\ell}) = 
 \frac{1}{3n(n-1)}\ell(\ell-1)\rho^4,
\end{eqnarray*}
which is consistent with \eqref{Eq:11}.
\end{Ex}

\subsection{Applications to geometric functionals of Gaussian random fields}\label{SecApp}
In this section, we apply our main results of Sections \ref{SecWC}, \ref{SectGeom} and \ref{SecOrtho} to the study of geometric functionals of multidimensional Gaussian fields. 

\medskip
In Section \ref{TotalVar}, we consider random variables admitting an integral representation in terms of Jacobian determinants associated with multi-dimensional Gaussian fields. We argue that such a definition can be interpreted as the \textit{total variation} of   vector-valued   
functions, generalizing the classical definition of total variation of multi-variate functions. More specifically, in the setting of a certain matrix correlation structure between two Jacobian matrices, appearing notably in the study of Gaussian Laplace eigenfunctions, we exploit the findings of Theorem \ref{OrtA} to obtain a precise expression for the variance of the total variation in terms of integrals of zonal polynomials.

\medskip
In Section \ref{SecARW}, we apply the general framework of Section \ref{TotalVar} to vectors of independent arithmetic random waves with the same eigenvalue on the three-dimensional torus, and prove a CLT in the high-energy regime for their generalized total variation on the full torus.

\medskip
In Section \ref{Nodal}, we consider the nodal volumes associated with vectors of independent arithmetic random waves on the three torus. In particular, we provide its Wiener-It\^{o} chaos expansions in terms of both, multivariate and matrix-variate Hermite polynomials, and provide some insight for variance estimates of its chaotic components.

\subsubsection{Generalized total variation of vector-valued functions}\label{TotalVar}

Let $n\geq 1$ be an integer and consider a centred smooth Gaussian field $\mathfrak{f}=\{\mathfrak{f}(z): z\in \R^n\}$ on $\R^n$. For $1\leq \ell \leq n$, we consider $\ell$ i.i.d copies 
$\mathfrak{f}^{(1)}, \ldots, \mathfrak{f}^{(\ell)}$
of $\mathfrak{f}$  and are interested in the $\ell$-dimensional Gaussian field
\begin{gather*}
\mathfrak{f}_{\ell}=\left\{\mathfrak{f}_{\ell}(z)=(\mathfrak{f}^{(1)}(z),\ldots,\mathfrak{f}^{(\ell)}(z)): z \in \R^n\right\}.
\end{gather*}
We denote by $\mathfrak{f}_{\ell}'(z) \in \R^{\ell \times n}$  the Jacobian matrix of $\mathfrak{f}_{\ell}$ evaluated at $z \in \R^n$. Moreover, we assume that (i) for every $z\in \R^n$, the distribution of $\mathfrak{f}_{\ell}(z)$ is non-degenerate and (ii) for every $z\in \R^n, \ \mathfrak{f}_{\ell}'(z) \sim \mathscr{N}_{\ell \times n}(0,\Id_{\ell}\otimes \Id_n)$. 
We define the following random variable.
\begin{Def}
For a compact domain $U \subset \R^n$, we define
\begin{eqnarray}\label{TOT}
\mathbf{V}_{\ell,n}(\mathfrak{f}_{\ell};U) := \int_{U} \Phi(\mathfrak{f}_{\ell}'(z)) dz, 
\end{eqnarray}
where $\Phi(M):=\det(MM^T)^{1/2}$ for $M \in \R^{\ell \times n}$. 
\end{Def} 
 We note that the above integral is well-defined since $U$ is compact and $\det(\mathfrak{f}_{\ell}'(z)) $ is a multivariate polynomial in the entries of $\mathfrak{f}_{\ell}'(z)$.
We remark that the random variable $\mathbf{V}_{\ell,n}(\mathfrak{f}_{\ell};U)$ can be seen as a generalization of the total variation of vector-valued functions. Indeed, for $\ell=1$, \eqref{TOT} coincides with the definition of the total variation for functions $\R^n \to \R$. For $\ell=n$, \cite{Fonseca,Guido} consider a \textit{relaxed total variation} of the Jacobian given by the Area formula (see e.g. \cite[Proposition 6.1]{AW09})
\begin{gather*}
\mathbf{V}_{n,n}(\mathfrak{f}_n;U) = \int_{U} |\det(\mathfrak{f}_n'(z))| dz = \int_{\R^n} N_y(\mathfrak{f}_n;U) dy ,
\end{gather*}
where $N_y(\mathfrak{f}_n;U)=\mathrm{card}(\{z \in U: \mathfrak{f}_n(z)=y\})$.  Using the Co-area formula (\cite[Proposition 6.13]{AW09}) in \eqref{TOT} shows that 
\begin{gather*}
\mathbf{V}_{\ell,n}(\mathfrak{f}_{\ell};U) = \int_{\R^{\ell}} \sigma_y(\mathfrak{f}_{\ell};U) dy,
\end{gather*}
where $\sigma_y(\mathfrak{f}_{\ell};U) $ denotes the $(n-\ell)$-dimensional Hausdorff measure of the level set $\{z\in U: \mathfrak{f}_{\ell}(z) = y\}$: Thus, the definition \eqref{TOT} generalizes the above setting to functions $\R^n \to \R^{\ell}$ with $\ell<n$.

\medskip
From now on, $1\leq\ell\leq n$ are fixed and we write $\mathbf{V}(\mathfrak{f}_{\ell};U)=\mathbf{V}_{\ell,n}(\mathfrak{f}_{\ell};U)$. The fact that, for every $z\in \R^n$,  $\Phi(\mathfrak{f}_{\ell}'(z))$ is an element of $L^2(\mu_{\mathfrak{f}_{\ell}'(z)})$ implies that $\mathbf{V}(\mathfrak{f}_{\ell};U)$ can be expanded in matrix-variate Hermite polynomials by means of Corollary \ref{CoeffId}, yielding its Wiener chaos expansion
\begin{gather}\label{chaos}
\mathbf{V}(\mathfrak{f}_{\ell};U) = \sum_{k \geq 0} \mathbf{V}(\mathfrak{f}_{\ell};U)[2k] \ , 
\quad \mathbf{V}(\mathfrak{f}_{\ell};U)[2k] = \sum_{\kappa \vdash k}  \widehat{\Phi}(\kappa) \int_U H_{\kappa}^{(\ell,n)}(\mathfrak{f}_{\ell}'(z))dz,
\end{gather} 
where $\widehat{\Phi}(\kappa)$ is as in \eqref{ForId} and $\mathbf{V}(\mathfrak{f}_{\ell};U)[2k]$ denotes the projection of $\mathbf{V}(\mathfrak{f}_{\ell};U)$ onto the Wiener chaos of order $2k$ associated with $\mathfrak{f}_{\ell}$.
In the following proposition, we compute the variance of the total variation of $\mathfrak{f}_{\ell}$ on $U$ in the specific framework, where the matrices $\mathfrak{f}_{\ell}'(z)$ and $\mathfrak{f}_{\ell}'(z')$ satisfy a certain matrix correlation structure for every $z,z'\in \R^n$ (see \eqref{CovMatrix} below).
\begin{Prop}\label{VarTV}
Let the above notation prevail. Assume furthermore that for every $z,z'\in \R^n$, 
\begin{gather}\label{CovMatrix}
\mathfrak{f}_{\ell}'(z') \eqLaw\mathfrak{f}_{\ell}'(z)R(z,z')+X_0(\Id_n-R(z,z')^2)^{1/2}, 
\end{gather}
in distribution, where $X_0=X_0(z,z')$ is an independent copy of $\mathfrak{f}_{\ell}'(z)$ and $R(z,z') $ is a deterministic   matrix. Then, 
\begin{gather}\label{VarV}
\V{\mathbf{V}(\mathfrak{f}_{\ell};U)}
= \sum_{k\geq1}\sum_{\kappa\vdash k} \widehat{\Phi}(\kappa)^2
4^{-k}\bigg(\frac{n}{2}\bigg)_{\kappa}^{-1} k !
\frac{C_{\kappa}(\Id_{\ell})}{C_{\kappa}(\Id_n)} \int_{U\times U} C_{\kappa}(R(z,z')^2)  dz dz' ,
\end{gather}
where $\widehat{\Phi}(\kappa)$ is as in \eqref{ForId}.
\end{Prop}

\subsubsection{Applications to Arithmetic Random Waves on the three-torus}\label{SecARW}

The study of local and non-local features in the high-energy regime associated with zero and non-zero level sets of Gaussian Laplace eigenfunctions on manifolds has gained great importance in past years, where different models have been taken into consideration. Celebrated models include Berry's monochromatic random waves (see e.g. \cite{MPRW16,NPR17,PV20,DEL19} and \cite{Ber77,Ber02} for seminal contributions), spherical harmonics (see e.g. \cite{Wig10,MP11,CMW16,CMW16b,MRW17}) and arithmetic random waves (ARW) on the torus (see e.g. \cite{ORW08,RW08,KKW13,PR16,Cam17,BM17,DNPR16,Not20}). 

\medskip
In this section, we apply the general framework presented in Section \ref{TotalVar} to the setting of vectors of independent arithmetic random waves on the three-torus, $\T$. 

\medskip
\noindent\underline{\textit{Arithmetic random waves on the $d$-torus.}}
Let $\tor{d}=\R^d/\Z^d=[0,1]/_{\sim}, \ d\geq2$ denote the torus of dimension $d$. Arithmetic random waves   on $\tor{d}$, first introduced in \cite{RW08,ORW08} are Gaussian Laplace eigenfunctions satisfying
\begin{gather*}
\Delta T_n(z) + E_n T_n(z) = 0, \quad E_n=4\pi^2 n, \ z \in \tor{d}
\end{gather*}
where 
\begin{gather*}
n \in S_d:=\left\{ m\geq 1: \exists (m_1,\ldots,m_d)\in \Z^d: m_1^2+\ldots+m_d^2=m\right\},
\end{gather*}
that is, $n$ is an integer expressible as a sum of $d$ integer squares. The set of \textit{frequencies} associated with $n\in S_d$ is  
\begin{gather*}
\Lambda_n:=\left\{ \lambda=(\lambda_1,\ldots,\lambda_d)\in \Z^d: \lambda_1^2+\ldots+\lambda_d^2=n\right\}, 
\end{gather*}
and we write $|\Lambda_n|=:\Nn$ for its cardinality, that is, $\Nn$ is the number of ways in which $n$ is represented as a sum of squares. An $L^2(\tor{d})$-basis of eigenfunctions is given by complex exponentials of the form $\{e_{\lambda}(\cdot):=\exp(2\pi i \scal{\lambda}{\cdot}):\lambda\in \Lambda_n\}$, where $\scal{\cdot}{\cdot}$ denotes the standard Euclidean inner product on $\R^d$. For $n \in S_d$, the ARW with eigenvalue $E_n$ is defined as random a linear combination of complex exponentials
\begin{gather}\label{Tn}
T_n(z) = \frac{1}{\sqrt{\Nn}} \sum_{\lambda \in \Lambda_n} a_{\lambda}e_{\lambda}(z),
\end{gather}
where the coefficients $\{a_{\lambda}:\lambda\in \Lambda_n\}$ are independent standard complex Gaussian random variables save for the relation $a_{-\lambda}=\overline{a_{\lambda}}$, which makes $T_n$ real-valued. Alternatively, ARWs are defined as the Gaussian process $\{T_n(z):z \in \tor{d}\}$ on $\tor{d}$ with covariance function
\begin{gather}\label{rn}
r^{(n)}(z,z'):=\E{T_n(z)T_n(z')} = \frac{1}{\Nn}\sum_{\lambda\in \Lambda_n} e_{\lambda}(z-z')=:r^{(n)}(z-z'), \quad z,z' \in \tor{d}.
\end{gather}
Note that $r^{(n)}$ only depends on the difference $z-z'$, meaning that the random field $\{T_n(z): z \in \tor{d}\}$ is stationary. Moreover, the fact that $r_n(0)=1$, implies that for every $z\in \tor{d}$, $T_n(z)$ has variance one.

\medskip
\noindent\underline{\textit{Total variation of vectors of ARW on $\T$.}}
For an integer $1\leq \ell \leq 3$ and $n\in S_3$, we consider i.i.d copies $T_n^{(1)},\ldots, T_n^{(\ell)}$ of $T_n$ in \eqref{Tn} and consider the associated $\ell$-dimensional Gaussian field
\begin{gather}\label{MultiARW}
\bT_n^{(\ell)}:=\left\{\bT_n^{(\ell)}(z)=(T_n^{(1)}(z),\ldots,T_n^{(\ell)}(z)):z \in \T\right\}.
\end{gather}
Our specific goal is to study the high-energy behaviour of the total variation $\mathbf{V}(\bT_n^{(\ell)}; \T)$ (as defined in  \eqref{TOT}) of $\bT_n^{(\ell)}$ on the full torus, that is when $\Nn\to \infty$. 
Since for every $z \in \T$, we have 
\begin{gather*}
\V{\frac{\partial}{\partial {z_j}} T_n^{(i)}(z)} = \frac{E_n}{3},  \quad i\in [\ell], \ j \in [n],
\end{gather*}
we introduce the normalised partial derivatives
\begin{gather}\label{dernor}
\tilde{\partial}_{j}T_n^{(i)}(z) := \left(\frac{E_n}{3}\right)^{-1/2} \frac{\partial}{\partial {z_j}}T_n^{(i)}(z), 
\end{gather} 
with unit variance and write $\dot{\bT}_n^{(\ell)}(z) \in \R^{\ell \times n}$ for the normalised Jacobian matrix of $\bT_n^{(\ell)}$.
According to \eqref{TOT}, we use the homogeneity of the determinant in order to rewrite the total variation as  
\begin{gather}\label{totdef}
\mathbf{V}(\bT_n^{(\ell)};\T) 
= \left(\frac{E_n}{3}\right)^{\ell/2}\int_{\T} \Phi( \dot{\bT}_n^{(\ell)}(z)) dz,
\end{gather}
where $\Phi(M)=\sqrt{\det(MM^T)}$. Differentiating \eqref{Tn} and using the fact that $r^{(n)}(0)=1$ implies that $\dot{\bT}_n^{(\ell)}(z) \sim \mathscr{N}_{\ell \times n}(0, \Id_{\ell}\otimes \Id_n)$ and is stochastically independent of $\bT_n^{(\ell)}(z)$ for every $z \in \T$.
We furthermore adopt the notation 
\begin{gather}\label{rnnor}
{r}^{(n)}_{j,j'}(z) := \frac{\partial^2}{\partial z_j \partial z_j' } r^{(n)}(z), \quad 
\tilde{r}^{(n)}_{j,j'}(z):=\left(\frac{E_n}{3}\right)^{-1} r^{(n)}_{j,j'}(z).
\end{gather}
The statement of our result is divided into three parts: (i) gives the expected total variation of vector-valued ARWs on the full torus, (ii) is an exact variance asymptotic and (iii) is a Central Limit Theorem in the high-energy regime for the normalised total variation  
\begin{gather}\label{nor}
\widehat{\mathbf{V}}(\bT_n^{(\ell)};\T):=\frac{ \mathbf{V}(\bT_n^{(\ell)};\T)-\E{\mathbf{V}(\bT_n^{(\ell)};\T)}}{\V{\mathbf{V}(\bT_n^{(\ell)};\T)}^{1/2}}.
\end{gather}
\begin{Thm}\label{ARW}
Let the above notation prevail. 
\begin{enumerate}[label=\rm{(\roman*)}]
\item (Expected total variation)
For every $n \in S_3$, we have 
\begin{gather}\label{mean}
\E{\mathbf{V}(\bT_n^{(\ell)};\T)} = 
\left(\frac{E_n}{3}\right)^{\ell/2}   2^{\ell/2}\frac{\Gamma_{\ell}(2)}{\Gamma_{\ell}(\frac{3}{2})}
\end{gather}
\item (Asymptotic variance) 
As $n\to \infty, \notcon{n}{0,4,7}{8}$, 
\begin{gather}\label{Var}
\V{\mathbf{V}(\bT_n^{(\ell)};\T)} =
\left(\frac{E_n}{3}\right)^{\ell} 2^{\ell}  \frac{\Gamma_{\ell}(2)^2}{\Gamma_{\ell}(\frac{3}{2})^2}
\frac{\ell}{2\Nn}\left(1  + O(n^{-1/28+o(1)})\right)
\end{gather}
\item (CLT) 
As $n\to \infty, \notcon{n}{0,4,7}{8}$, 
\begin{gather}\label{CLT}
\widehat{\mathbf{V}}(\bT_n^{(\ell)};\T) \Law   \mathscr{N}(0,1),
\end{gather}
where $\Law$ denotes convergence in distribution.
\end{enumerate}
\end{Thm}

We remark that \eqref{mean} and \eqref{Var} imply that the normalised total variation $\mathbf{V}(\bT_n^{(\ell)};\T)/E_n^{\ell}$ converges in probability to $\left(\frac{2}{3}\right)^{\ell/2}\frac{\Gamma_{\ell}(2)}{\Gamma_{\ell}(3/2)}$ as $n\to \infty,\notcon{n}{0,4,7}{8}$.

Our proof of Theorem \ref{ARW} is based on expanding the total variation in \eqref{totdef} into matrix-variate Hermite polynomials by means of Corollary \ref{CoeffId} (see also \eqref{chaos}). As we will prove, 
the high-energy distributional behaviour of the normalised total variation is entirely characterized by its projection on the second Wiener chaos, which explains the underlying Gaussian fluctuations. 
 In order to prove the negligibility of higer-order Wiener chaoses with respect to the second one, we rely on fine estimates for the second and sixth integral moments of $r^{(n)}$ derived in \cite{BM17}.

\subsubsection{Digression: Comparison with \cite{Not20}}\label{Nodal}
 In this section, we compare our findings with \cite{Not20}, where we study nodal volumes 
$L_n^{(\ell)}:=\mathcal{H}_{3-\ell}(\mathcal{Z}(\bT_n^{(\ell)}))$ 
of the nodal sets $\mathcal{Z}(\bT_n^{(\ell)})$ associated with $\bT_n^{(\ell)}$ in \eqref{MultiARW} (here, $\mathcal{H}_k$ denotes the $k$-dimensional Hausdorff measure, with $\mathcal{H}_0$ indicating counting measure). Such a work is in particular based on the asymptotic study of the fourth chaotic projection associated with  the Wiener-It\^{o} chaos expansion of $L_n^{(\ell)}$. Recall that, in view of the Co-Area formula, the nodal volume $L_n^{(\ell)}$ is defined $\Prob$-almost surely and in $L^2(\Prob)$ as 
\begin{gather}\label{VolZ}
L_{n}^{(\ell)}= \left(\frac{E_n}{3}\right)^{\ell/2}\int_{\T}
\delta_{(0,\ldots,0)}(\bT_n^{(\ell)}(z)) 
\times \Phi(\dot{\bT}_n^{(\ell)}(z)) dz ,
\end{gather}
where $\Phi$ is as in \eqref{totdef}, $\dot{\bT}_n^{(\ell)}$ is the normalised Jacobian matrix of $\bT_n^{(\ell)}$, and $\delta_{(0,\ldots,0)}(x):=\delta_0(x_1) \cdots \delta_0(x_{\ell}), x=(x_1,\ldots, x_{\ell})$ denotes the multiple Dirac mass at the origin. Using matrix-Hermite polynomials studied in the present article, the  chaotic projection of $L_n^{(\ell)}$ on the Wiener chaos of order $2q$ is obtained by means of Corollary \ref{CoeffId} as
\begin{eqnarray}
L_{n}^{(\ell)}[2q] &=& \left(\frac{E_n}{3}\right)^{\ell/2}\sum_{q_1+2q_2=2q} \int_{\T}
\sum_{|\alpha|=q_1} \frac{\tilde{\beta}_{\alpha}}{\alpha!} H_{\alpha}^{\otimes \ell}(\bT_n^{(\ell)}(z)) 
\times \sum_{\kappa \vdash q_2} \widehat{\Phi}(\kappa) H_{\kappa}^{(\ell,3)}(\dot{\bT}_n^{(\ell)}(z)) dz \notag\\
&=& \left(\frac{E_n}{3}\right)^{\ell/2}\sum_{q_1+2q_2=2q} 
\sum_{|\alpha|=q_1} \sum_{\kappa \vdash q_2}
\frac{\tilde{\beta}_{\alpha}}{\alpha!}\widehat{\Phi}(\kappa)  \int_{\T} H_{\alpha}^{\otimes \ell}(\bT_n^{(\ell)}(z)) 
H_{\kappa}^{(\ell,3)}(\dot{\bT}_n^{(\ell)}(z)) dz \label{c2q}
\end{eqnarray}
where for a multi-index $\alpha\in \N^{\ell}, \widetilde{\beta}_{\alpha}$ denote the projection coefficients of the Dirac mass and $\widehat{\Phi}(\kappa)$  can be computed from \eqref{ForId} (applied with $n=3$) 
\begin{gather}\label{PhiKappa}
\widehat{\Phi}(\kappa) 
= (-2)^{k}   \bigg(\frac{3}{2}\bigg)_{\kappa} 2^{\ell/2}   \cdot \frac{1}{k!}\frac{\Gamma_{\ell}(2)}{\Gamma_{\ell}(\frac{3}{2})} \sum_{s=0}^{k} \sum_{\sigma \vdash s}
{\kappa \choose \sigma} (-1)^s\frac{(2)_{\sigma}}{(\frac{3}{2})_{\sigma}}.
\end{gather}
Writing out the explicit values of the projection coefficients in \eqref{PhiKappa} for partitions $\kappa\in \{(2),(1,1)\}$ and using the expressions for matrix-variate Hermite polynomials in \eqref{HerPol}, we eventually recover the projection coefficients associated with $\Phi$ appearing in \cite[Proposition B.5]{Not20}. 
We remark that, unlike the Wiener chaos expansion of the generalized total variation in \eqref{chaos}, the presence of the multiple Dirac mass leads to an expression containing both, 
multivariate and matrix-variate Hermite polynomials. Specifying \eqref{c2q} to $q=2$ yields that the projection of $L_n^{(\ell)}$ on the fourth Wiener chaos can be written compactly as the sum of five terms 
\begin{eqnarray*}
L_n^{(\ell)}[4]=
\left(\frac{E_n}{3}\right)^{\ell/2} \left[S_1^{(\ell)}(n)+\ldots+S_5^{(\ell)}(n)\right], 
\end{eqnarray*}
where 
\begin{eqnarray*}
S_1^{(\ell)}(n)&:=& \frac{\beta_4}{4!}\widehat{\Phi}((0))\sum_{i\in[\ell]}  \int_{\T} 
H_4(T_n^{(i)}(z)) dz\\
S_2^{(\ell)}(n) &:=&\left(\frac{\beta_2 }{2! }\right)^2\widehat{\Phi}((0))\sum_{i< j\in [\ell]}  \int_{\T} H_2(T_n^{(i)}(z))H_2(T_n^{(j)}(z))dz\\
S_3^{(\ell)}(n) &:=&\frac{\beta_2}{2!}\widehat{\Phi}((1))\sum_{i\in [\ell]} \int_{\T}  H_2(T_n^{(i)}(z)) H^{(\ell,3)}_{(1)}(\dot{\bT}_n^{(\ell)}(z)) dz\\
S_4^{(\ell)}(n) &:=&\tilde{\beta}_0\widehat{\Phi}((2)) \int_{\T} H_{(2)}^{(\ell,3)}(\dot{\bT}_n^{(\ell)}(z)) dz\\
S_5^{(\ell)}(n) &:=&\tilde{\beta}_0\widehat{\Phi}((1,1)) \int_{\T} H_{(1,1)}^{(\ell,3)}(\dot{\bT}_n^{(\ell)}(z))dz.
\end{eqnarray*}
Such an expression should be compared with \cite[Eq.(3.23)]{Not20}.  
We remark that in the case $\ell=1$, the terms $S_2^{(1)}(n)$ and $S_5^{(1)}(n)$ disappear, 
since in this case, only matrix-Hermite polynomials associated with partitions $\kappa$ of length one contribute to the chaotic expansion of $L_n^{(\ell)}$. The rich combinatorial structure of matrix-Hermite polynomials allows us to deduce a number of interesting observations about variance estimates:
Exploiting Theorem \ref{OrtA} for the term $S_4^{(\ell)}(n)$ (and similarly for $S_5^{(\ell)}(n)$) yields
\begin{eqnarray*}
\V{S_4^{(\ell)}(n)} &=& \tilde{\beta}_0^2\widehat{\Phi}((2))^2 4^{-2}\bigg(\frac{3}{2}\bigg)_{(2)}^{-1} 2 ! \frac{C_{(2)}(\Id_{\ell})}{C_{(2)}(\Id_3)} 
\int_{\T\times \T}C_{(2)}(R_n(z-z')^2)dzdz'\\
&=&\tilde{\beta}_0^2\widehat{\Phi}((2))^2 4^{-2}\bigg(\frac{3}{2}\bigg)_{(2)}^{-1} 2 ! \frac{C_{(2)}(\Id_{\ell})}{C_{(2)}(\Id_3)} 
\int_{\T }C_{(2)}(R_n(z)^2)dz ,
\end{eqnarray*}
where the last identity follows by stationarity. Moreover since the terms $S_4^{(\ell)}(n)$ and $S_5^{(\ell)}(n)$ involve different partitions of the integer $2$, Theorem \ref{OrtA} implies that
the random variables $S_4(n)$ and $S_5(n)$ are orthogonal in $L^2(\Prob)$.
It should be remarked that $S_4^{(\ell)}(n)$ and $S_5^{(\ell)}(n)$ are however not orthogonal in $L^2(\Prob)$ to the remaining terms $S_p^{(\ell)}(n), p=1,2,3$, as can be seen for instance from
\begin{gather*}
\E{S_4^{(\ell)}(n) S_1^{(\ell)}(n)} = 
\left(\frac{\beta_4}{4!}\right)^2   \widehat{\Phi}((2))^2 \times\sum_{i\in[\ell]}  \int_{\T\times \T} \E{  
H_4(T_n^{(i)}(z)) H_{(2)}^{(\ell,3)}(\dot{\bT}_n^{(\ell)}(z'))  } dzdz',
\end{gather*}
involving covariances between univariate and matrix-variate Hermite polynomials. 
In order to deal with such expressions, one can 
expand matrix-Hermite polynomials into univariate Hermite polynomials and rely on the classical diagram formulae for the latter.

\section{Proofs of main results}\label{SecProofs}
\subsection{Proofs of Section \ref{SecWC}}
\subsubsection*{Proof of Theorem \ref{ThmWC}}
Since $F \in L^2(\mu_X) \subset L^2(\phi^{(\ell,n)})$, we can expand it in the two orthonormal systems $\mathbb{H}_{[\ell n]}$ and $\mathbb{H}_{[\ell \times n]}$ defined in \eqref{Hm} and \eqref{HerONB} respectively, yielding
\begin{gather}\label{WCEq}
F=\sum_{k \geq 0} \proj(F|\mathbf{C}_k^{\mathbf{X}}) = \sum_{k \geq 0} \sum_{\kappa \vdash k}
\widehat{F}(\kappa) 
H_{\kappa}^{(\ell,n)}.
\end{gather}
Using the representation of zonal polynomials in \eqref{HOM}, we write $C_{\kappa}(XX^T)$ as a homogeneous polynomial of degree $2k$ in the entries of $X=(X_{ij})$, that is 
\begin{gather*}
C_{\kappa}(XX^T) = \sum_{|\alpha|=2k} z_{\alpha}^{\kappa} \prod_{i=1}^{\ell} \prod_{j=1}^n X_{ij}^{\alpha_{ij}} ,
\end{gather*}
where $\alpha\in \N^{\ell \times n}$ is a multi-index such that $|\alpha|=2k$ and $z_{\alpha}^{\kappa}$ is an explicit constant depending on $\alpha$ and $\kappa$.
Using the above representation of zonal polynomials in the generalized Rodrigues formula \eqref{Rod}, it follows that 
\begin{eqnarray*}
H_{\kappa}^{(\ell,n)}(X) &=& 4^{-k} \left(\frac{n}{2}\right)_{\kappa}^{-1} [\phi^{(\ell,n)}(X)]^{-1} C_{\kappa}(\partial X\partial X^T) \phi^{(\ell,n)}(X) \\
&=& 4^{-k} \left(\frac{n}{2}\right)_{\kappa}^{-1} [\phi^{(\ell,n)}(X)]^{-1} \sum_{|\alpha|=2k} z_{\alpha}^{\kappa} \prod_{i=1}^{\ell} \prod_{j=1}^n \frac{\partial^{\alpha_{ij}}}{\partial X_{ij}^{\alpha_{ij}}} \phi^{(\ell,n)}(X) \\
&=& 4^{-k} \left(\frac{n}{2}\right)_{\kappa}^{-1}  \sum_{|\alpha|=2k} z_{\alpha}^{\kappa} \prod_{i=1}^{\ell} \prod_{j=1}^n [\phi(X_{ij})]^{-1} \frac{\partial^{\alpha_{ij}}}{\partial X_{ij}^{\alpha_{ij}}} \phi(X_{ij}).
\end{eqnarray*}
Then, using the classical Rodrigues formula for Hermite polynomials on the real line \eqref{UniRod} for every $i\in [\ell], j\in [n]$, we infer that 
\begin{gather*}
[\phi(X_{ij})]^{-1} \frac{\partial^{\alpha_{ij}}}{\partial X_{ij}^{\alpha_{ij}}} \phi(X_{ij}) = (-1)^{\alpha_{ij}} H_{\alpha_{ij}}(X_{ij}), 
\end{gather*} 
so that, using the fact that $|\alpha|=2k$,
\begin{eqnarray*}
H_{\kappa}^{(\ell,n)}(X) &=& 
4^{-k} \left(\frac{n}{2}\right)_{\kappa}^{-1}  \sum_{|\alpha|=2k} z_{\alpha}^{\kappa} \prod_{i=1}^{\ell} \prod_{j=1}^n (-1)^{\alpha_{ij}}H_{\alpha_{ij}}(X_{ij}) \\
&=& 4^{-k} \left(\frac{n}{2}\right)_{\kappa}^{-1}  \sum_{|\alpha|=2k} z_{\alpha}^{\kappa}  H_{\alpha}^{\otimes \ell n}(X_{11},\ldots, X_{\ell n}).
\end{eqnarray*}
The above expression yields the   expansion of $H_{\kappa}^{(\ell,n)}(X)$ into multivariate Hermite polynomials and implies in particular that $H_{\kappa}^{(\ell,n)}(X)$ is an element of the Wiener chaos of order $2k$ associated with the vector $\mathbf{X}=\mathrm{Vec}(X)$. The formula for the projection of $F$ onto $\mathbf{C}_{2k}^{\mathbf{X}}$ in \eqref{proj2k} then follows summing over all partitions of $k$. The fact that the projection of $F$ onto Wiener chaos of odd order is zero follows from the fact that the RHS of \eqref{WCEq} does not involve any multivariate Hermite polynomials of cumulative odd order since $H_{\kappa}^{(\ell,n)}(X) \in \mathbf{C}_{2k}^{\mathbf{X}}$. 

In order to prove formula \eqref{RAD}, we use the identity \eqref{RelHerLag} and subsequently apply the polar decomposition $X=R^{1/2}U$ according to \eqref{Pol}, yielding $
(dX) = \frac{\pi^{n\ell/2}}{\Gamma_{\ell}(\frac{n}{2})} \det(R)^{\frac{n-\ell-1}{2}}\nu(dR)\tilde{\mu}(dU)$, (see e.g.  \cite[Theorem 1.5.2]{Ch2012}).
Therefore, we have from \eqref{Coeff}  
\begin{eqnarray*}
&&\hat{F}(\kappa)= c(\kappa)^{-1}\int_{\R^{\ell\times n}} F(X) H_{\kappa}^{(\ell,n)}(X) \phi^{(\ell,n)}(X)(dX) \\
&=& c(\kappa)^{-1} \gamma_{\kappa} (2\pi)^{-n\ell/2} \int_{\R^{\ell\times n}} f_0(XX^T) L_{\kappa}^{(\frac{n-\ell-1}{2})}(2^{-1}XX^T) \etr{-2^{-1}XX^T}(dX)\\
&=& c(\kappa)^{-1} \gamma_{\kappa}
(2\pi)^{-n\ell/2}  \int_{O(n,\ell)} 
\int_{\mathcal{P}_{\ell}(\R)} f_0(R) L_{\kappa}^{(\frac{n-\ell-1}{2})}(2^{-1}R) \etr{-2^{-1}R} \frac{\pi^{n\ell/2}}{\Gamma_{\ell}(\frac{n}{2})} \det(R)^{\frac{n-\ell-1}{2}}\nu(dR)\tilde{\mu}(dU)\\
&=&c(\kappa)^{-1} \gamma_{\kappa}(2\pi)^{-n\ell/2}  
\int_{\mathcal{P}_{\ell}(\R)} f_0(R) L_{\kappa}^{(\frac{n-\ell-1}{2})}(2^{-1}R) \etr{-2^{-1}R} \frac{\pi^{n\ell/2}}{\Gamma_{\ell}(\frac{n}{2})} \det(R)^{\frac{n-\ell-1}{2}}\nu(dR) \\
&=& \frac{1}{2^{n\ell/2}\Gamma_{\ell}(\frac{n}{2})}\frac{(-2)^{k}}{k!C_{\kappa}(\Id_{\ell})}\int_{\mathcal{P}_{\ell}(\R)} f_0(R) L_{\kappa}^{(\frac{n-\ell-1}{2})}(2^{-1}R) \etr{-2^{-1}R}  \det(R)^{\frac{n-\ell-1}{2}}\nu(dR),
\end{eqnarray*}
where we used that $\tilde{\mu}$ is a probability measure on $O(n,\ell)$ and the definitions of $c(\kappa)$ and $\gamma_{\kappa}$ in \eqref{HerONB} and \eqref{RelHerLag}, respectively.
 This finishes the proof.

\subsubsection*{Proof of Proposition \ref{VarEx}}
By Theorem \ref{ThmWC}, the Wiener-It\^{o} chaos expansion of $F(X)$ is given by
\begin{gather}\label{Exp}
F(X) = \sum_{k\geq 0} \sum_{\kappa\vdash k} \widehat{F}(\kappa) H_{\kappa}^{(\ell,n)}(X), 
\end{gather}
where $\widehat{F}(\kappa)$ is as in \eqref{Coeff}.
Computing the $L^2(\Prob)$-norm on both sides of \eqref{Exp} and using the orthogonality relation \eqref{HerOrtho} then yields
\begin{gather*}
\E{F(X)^2} 
= \sum_{k \geq 0} \sum_{l \geq 0}\sum_{\kappa \vdash k}
\sum_{\sigma \vdash l}
\widehat{F}(\kappa)
\widehat{F}(\sigma)
\E{H_{\kappa}^{(\ell,n)}(X)
H_{\sigma}^{(\ell,n)}(X)} \\
= \sum_{k \geq 0}\sum_{\kappa \vdash k} 
\widehat{F}(\kappa)^2
4^{-k}\left(\frac{n}{2}\right)_{\kappa}^{-1} k !C_{\kappa}(\Id_{\ell})   
= \widehat{F}((0))^2
+ \sum_{k\geq 1} \sum_{\kappa\vdash k} 
\widehat{F}(\kappa)^2
4^{-k}\left(\frac{n}{2}\right)_{\kappa}^{-1} k!C_{\kappa}(\Id_{\ell})  .
\end{gather*}
Since $\widehat{F}((0))=\E{F(X)}$, we obtain the expansion for the variance of $F(X)$, 
\begin{eqnarray*}
\V{F(X)} &=& \sum_{k\geq1}\sum_{\kappa\vdash k} 
\widehat{F}(\kappa)^2
4^{-k}\left(\frac{n}{2}\right)_{\kappa}^{-1} k!C_{\kappa}(\Id_{\ell}) \\
&=& \sum_{k\geq 1}\sum_{\kappa \vdash k} 
\left[4^{-k}\left(\frac{n}{2}\right)_{\kappa}^{-1} k!C_{\kappa}(\Id_{\ell}) \right]^{-1} \E{F(X) H_{\kappa}^{(\ell,n)}(X)}^2 \\
&=& \sum_{k\geq 1}\sum_{\kappa \vdash k} 
\frac{4^{k}(\frac{n}{2})_{\kappa}}{k!C_{\kappa}(\Id_{\ell}) } 
\E{F(X) H_{\kappa}^{(\ell,n)}(X)}^2,
\end{eqnarray*}
where we used  \eqref{HerONB}.

\subsection{Proofs of Section \ref{SectGeom}} 
\subsubsection*{Polar decomposition of Gaussian rectangular matrices}
Let us assume that $X$ has the $\mathscr{N}_{\ell \times n}(0,\Id_{\ell} \otimes \Sigma)$ distribution with density function
$\phi_{\Sigma}^{(\ell,n)}(X)$ defined in \eqref{denSigma}, 
and write $X=R^{1/2}U$ for its polar decomposition according to \eqref{Pol}. 
In the following lemma, we compute the joint probability density function of the pair $(R,U)$.
\begin{Lem}
If $X \sim \mathscr{N}_{\ell \times n}(0,\Id_{\ell} \otimes \Sigma)$, the joint probability density of the pair $(R,U)$ is given by 
\begin{eqnarray}\label{denMU}
f_{(R,U)}(R,U) 
= \frac{1}{\Gamma_{\ell}(\frac{n}{2})}
\frac{1}{2^{n\ell/2}} \det(\Sigma)^{-\ell/2} \det(R)^{\frac{n-\ell-1}{2}}
\etr{ -2^{-1}U\Sigma^{-1}U^TR} . 
\end{eqnarray} 
\end{Lem}
\begin{proof}
Applying the polar change of variable $X=R^{1/2}U$ gives $(dX) = \frac{\pi^{n\ell/2}}{\Gamma_{\ell}(\frac{n}{2})} 
\det(R)^{\frac{n-\ell-1}{2}} \nu(dR) \tilde{\mu}(dU)$ (see e.g. Theorem 1.5.2 \cite{Ch2012}),
so that 
\begin{eqnarray*}
\phi_{\Sigma}(X)(dX) &=&  \phi_{\Sigma}(R^{1/2}U) \frac{\pi^{n\ell/2}}{\Gamma_{\ell}(\frac{n}{2})} 
\det(R)^{\frac{n-\ell-1}{2}} \nu(dR) \tilde{\mu}(dU)\\
&=& \frac{1}{\Gamma_{\ell}(\frac{n}{2})}
\frac{1}{2^{n\ell/2}} \det(\Sigma)^{-\ell/2} \det(R)^{\frac{n-\ell-1}{2}}
\etr{ -2^{-1}U\Sigma^{-1}U^TR}\nu(dR) \tilde{\mu}(dU), 
\end{eqnarray*}
where we used that $\etr{-2^{-1}R^{1/2}U\Sigma^{-1}U^{T}R^{1/2}}=\etr{ -2^{-1}U\Sigma^{-1}U^TR}$.
\end{proof}
The following lemma (see \cite[Theorem 2.4.2]{Ch2012}) gives the marginal density functions of $R$ and $U$, respectively. These are obtained when integrating the joint density $f_{(R,U)}(R,U)$ with respect to $U$ and $R$, respectively.
\begin{Lem}[Theorem 2.4.2, \cite{Ch2012}]
Assume that $X \sim \mathscr{N}_{\ell \times n}(0,\Id_{\ell} \otimes \Sigma)$ and write $X = R^{1/2} \cdot U$. Then, the marginal density functions of $R$ and $U$ are respectively given by 
\begin{gather}\label{denM}
f_R(R)= \frac{1}{2^{n\ell/2} \Gamma_{\ell}(\frac{n}{2}) \det(\Sigma)^{\ell/2}}  \Hyp{0}{}{0}{}{-2^{-1}\Sigma^{-1},R} \det(R)^{\frac{n-\ell-1}{2}},
\end{gather}
where 
\begin{gather*}
\Hyp{p}{a_1,\ldots,a_p}{q}{b_1,\ldots,b_q}{S,T}= 
\sum_{k\geq0}\sum_{\kappa \vdash k} \frac{(a_1)_{\kappa}\cdots (a_p)_{\kappa}}{(b_1)_{\kappa}\cdots (b_q)_{\kappa}} \frac{C_{\kappa}(S)C_{\kappa}(T)}{k!C_{\kappa}(\Id_{\ell})}
\end{gather*}
denotes the hypergeometric function with two matrix arguments (see e.g. \cite[Appendix A.6]{Ch2012})
and
\begin{gather}\label{denU}
f_U(U)= \det(\Sigma)^{-\ell/2}\det(U\Sigma^{-1}U^T)^{-n/2} .
\end{gather}
\end{Lem}
The density function of $U$ in \eqref{denU} is referred to as the \textit{matrix angular central distribution with parameter} $\Sigma$ on $O(n,\ell)$. We also point out that, when $\Sigma=\Id_n$, the matrix $R$ follows the \textit{Wishart distribution} with density function  
\begin{gather*}
\frac{1}{2^{n\ell/2} \Gamma_{\ell}(\frac{n}{2})} \Hyp{0}{}{0}{}{-2^{-1}\Id_{\ell},R} \det(R)^{\frac{n-\ell-1}{2}}  
= \frac{1}{2^{n\ell/2} \Gamma_{\ell}(\frac{n}{2})} 
\etr{-2^{-1}R}\det(R)^{\frac{n-\ell-1}{2}}
\end{gather*}
and the matrix angular central distribution of $U$ reduces to the uniform distribution on $O(n,\ell)$. Moreover, it follows from \eqref{denMU}, that in this case, $R$ and $U$ are independent.

\medskip
\noindent Combining \eqref{denMU} with \eqref{denU}, we obtain the conditional probability density  of $R$ given $U$:
\begin{gather}\label{MgivenU}
\frac{f_{(R,U)}(R,U)}{f_U(U)}
= \frac{1}{\Gamma_{\ell}(\frac{n}{2})}
\frac{1}{2^{n\ell/2}}  \det(R)^{\frac{n-\ell-1}{2}}
\etr{ -2^{-1}U\Sigma^{-1}U^TR}  \det(U\Sigma^{-1}U^T)^{ n/2}  .
\end{gather} 
In the forthcoming sections, whenever $Z$ is a random variable, we often write 
$\EX{Z}{\cdot} $ to indicate mathematical expectation with respect to the law of $Z$.

\subsubsection*{Proof of Proposition \ref{DefSigma}}
We observe that the following relation holds
\begin{gather}\label{RelPhi}
\phi^{(\ell,n)}_{\Sigma}(X) = \det(\Sigma)^{-\ell/2} \phi^{(\ell,n)}(X\Sigma^{-1/2}), 
\end{gather}
where $\phi^{(\ell,n)}$ denotes the standard Gaussian density on $\R^{\ell \times n}$.
From the definition \eqref{HerSigma} and the relation \eqref{RelPhi}, it hence follows that 
\begin{gather*}
\int_{\R^{\ell \times n}} H^{(\ell,n)}_{\kappa}(X;\Sigma)
H^{(\ell,n)}_{\sigma}(X;\Sigma)\phi^{(\ell,n)}_{\Sigma}(X)(dX) \\
= \det(\Sigma)^{\ell k+\ell l-\ell/2}\int_{\R^{\ell\times n} }
H_{\kappa}^{(\ell,n)}(X\Sigma^{-1/2})
H_{\sigma}^{(\ell,n)}(X\Sigma^{-1/2}) \phi^{(\ell,n)}(X\Sigma^{-1/2}) (dX) .
\end{gather*}
Applying the change of variables $Y=X\Sigma^{-1/2}$, we have $(dY) =  \det(\Sigma^{-1/2})^{\ell}(dX) = \det(\Sigma)^{-\ell/2} (dX)$
(see e.g.  \cite[Theorem 2.1.5]{Muir}),
i.e. $(dX) = \det(\Sigma)^{\ell/2} (dY)$, so that the integral above becomes 
\begin{eqnarray*}
&&\det(\Sigma)^{\ell k+\ell l-\ell/2}\int_{\R^{\ell \times n}} H_{\kappa}^{(\ell,n)}(X\Sigma^{-1/2})
H_{\sigma}^{(\ell,n)}(X\Sigma^{-1/2}) \phi^{(\ell,n)}(X\Sigma^{-1/2}) (dX) \\
&=&  \det(\Sigma)^{\ell k+\ell l }\int_{\R^{\ell \times n}} H_{\kappa}^{(\ell,n)}(Y)
H_{\sigma}^{(\ell,n)}(Y) \phi^{(\ell,n)}(Y) (dY) \\
&=& \ind{\kappa=\sigma} \det(\Sigma)^{2\ell k}4^{-k}\bigg(\frac{n}{2}\bigg)_{\kappa}^{-1} k !C_{\kappa}(\Id_{\ell}) ,
\end{eqnarray*}
where we used \eqref{HerOrtho}. This proves the statement.

\subsubsection*{Proof of Theorem \ref{CoeffSigma}}
The proof of Theorem \ref{CoeffSigma} is based on the following key identity. 
\begin{Lem}
Let $A\in \C^{\ell \times \ell}$ be a  complex symmetric matrix with positive real part, $B\in \C^{\ell \times \ell}$  a  complex symmetric matrix and $t \in \C$  such that $\Re(t)>(\ell-1)/2$. Then, we have 
\begin{gather}\label{LaplaceLag} 
\int_{\mathcal{P}_{\ell}(\R)}
\etr{-AR}\det(R)^{t-\frac{\ell+1}{2}}L^{(\gamma)}_{\kappa}(RB) \nu(dR) \\
= \left(\gamma+\frac{\ell+1}{2}\right)_{\kappa} C_{\kappa}(\Id_{\ell}) 
\Gamma_{\ell}(t) \det(A)^{-t}
\sum_{s=0}^{k} \sum_{ \sigma \vdash s}
{\kappa \choose \sigma} \frac{(-1)^s}{(\gamma+\frac{\ell+1}{2})_{\sigma}}\frac{1}{C_{\sigma}(\Id_{\ell})} 
(t)_{\sigma} C_{\sigma}(BA^{-1}).  \notag
\end{gather}
\end{Lem}
\begin{proof}
This identity follows directly from the definition of Laguerre polynomials in \eqref{LagZon}: indeed by linearity, it suffices to apply
 relation \eqref{intzonal} on each zonal polynomial $C_{\sigma}$ appearing in the expansion of $L_{\kappa}^{(\gamma)}$.
\end{proof}

We are now in position to prove Theorem \ref{CoeffSigma}. 
\begin{proof}[Proof of Theorem \ref{CoeffSigma}]
The fact that the random variable $F(X)= \det(XX^T)^{1/2}$ is an element of $L^2( \mu_X)$ follows from the following observation: Denoting by $s_1,\ldots, s_{\ell}$ the eigenvalues of $XX^T$, we have that 
\begin{gather*}
\det(XX^T) = \prod_{i=1}^{\ell} s_i  
= \frac{1}{\ell!} \sum_{i_1\neq \ldots \neq i_{\ell}\in[\ell]} s_{i_1} \cdots s_{i_{\ell}} 
=  \int_{\R}\cdots \int_{\R}
 \frac{1}{\ell!} \ind{t_i \neq t_j, \ \forall i\neq j \in [\ell] }
 \mu_X(dt_1)\ldots\mu_X(dt_{\ell}).
\end{gather*}
This justifies the decomposition into matrix-variate Hermite polynomials of $F$.  
We now prove formula \eqref{For1}. Using the definition of the polynomials $H_{\kappa}^{(\ell,n)}(X;\Sigma)$ in \eqref{HerSigma} and the relation \eqref{RelHerLag},
we obtain from \eqref{Fourier}
\begin{eqnarray*}
\widehat{F}(\kappa; \Sigma) &=& 
c(\kappa;\Sigma)^{-1}
\EX{X}{F(X) H^{(\ell,n)}_{\kappa}(X;\Sigma)} \\
&=& c(\kappa;\Sigma)^{-1} \det(\Sigma)^{\ell k} \gamma_{\kappa}
\EX{X}{F(X) L_{\kappa}^{(\frac{n-\ell-1}{2})}(2^{-1}X\Sigma^{-1}X^T) }, \quad 
X \sim \mathscr{N}_{\ell\times n}(0,\Id_{\ell}\otimes \Sigma).
\end{eqnarray*}
Applying the polar decomposition $X= R^{1/2}U$ and noting that $F(R^{1/2}U)=\det(R)^{1/2}$, we have 
\begin{eqnarray*}
\widehat{F}(\kappa; \Sigma)
&=& c(\kappa;\Sigma)^{-1} \det(\Sigma)^{\ell k} \gamma_{\kappa}
\EX{(R,U)}{ \det(R)^{1/2} L_{\kappa}^{(\frac{n-\ell-1}{2})}(2^{-1}R^{1/2}U\Sigma^{-1}U^TR^{1/2}) } \\
&=& c(\kappa;\Sigma)^{-1} \det(\Sigma)^{\ell k} \gamma_{\kappa} 
\EX{(R,U)}{ \det(R)^{1/2} L_{\kappa}^{(\frac{n-\ell-1}{2})}(2^{-1}U\Sigma^{-1}U^TR) } 
\end{eqnarray*}
where in the last line we used the fact that 
$L_{\kappa}^{(\frac{n-\ell-1}{2})}(2^{-1}R^{1/2}U\Sigma^{-1}U^TR^{1/2}) 
=L_{\kappa}^{(\frac{n-\ell-1}{2})}(2^{-1}U\Sigma^{-1}U^TR)$
in view of the permutation invariance property \eqref{Inv} of zonal polynomials appearing in the definition of matrix-variate Laguerre polynomials \eqref{LagZon}.
By conditioning on $U$, we can rewrite the above expectation as
\begin{eqnarray*}
&&\EX{(R,U)}{ \det(R)^{1/2} L_{\kappa}^{(\frac{n-\ell-1}{2})}(2^{-1}U\Sigma^{-1}U^TR) } \\
&=&\EX{U}{\EX{R|U}{\det(R)^{1/2} L_{\kappa}^{(\frac{n-\ell-1}{2})}(2^{-1}U\Sigma^{-1}U^TR)}}  
=: \EX{U}{Z_{\kappa}(U;\Sigma)},
\end{eqnarray*}
where 
\begin{gather*}
Z_{\kappa}(U;\Sigma):=\EX{R|U}{\det(R)^{1/2} L_{\kappa}^{(\frac{n-\ell-1}{2})}(2^{-1}U\Sigma^{-1}U^TR)},
\end{gather*}
so that 
\begin{gather}\label{FSigma}
\widehat{F}(\kappa; \Sigma)
= c(\kappa;\Sigma)^{-1} \det(\Sigma)^{\ell k} \gamma_{\kappa}
\EX{U}{Z_{\kappa}(U;\Sigma)}.
\end{gather}
We start by computing $Z_{\kappa}(U;\Sigma)$. Using the conditional probability density of $R$ given $U$ in \eqref{MgivenU}, we have 
\begin{eqnarray}
 Z_{\kappa}(U;\Sigma)&=& \int_{\mathcal{P_{\ell}(\R)}} \det(R)^{1/2} L_{\kappa}^{(\frac{n-\ell-1}{2})}(2^{-1}U\Sigma^{-1}U^TR) 
\frac{f_{(R,U)}(R,U)}{f_U(U)} \nu(dR) \notag \\
&=& \int_{\mathcal{P_{\ell}(\R)}} \det(R)^{1/2} L_{\kappa}^{(\frac{n-\ell-1}{2})}(2^{-1}U\Sigma^{-1}U^TR) 
\frac{1}{\Gamma_{\ell}(\frac{n}{2})}
\frac{1}{2^{n\ell/2}}  \det(R)^{\frac{n-\ell-1}{2}}
\etr{ -2^{-1}U\Sigma^{-1}U^TR}  \notag\\
&& \hspace{8cm}\times \det(U\Sigma^{-1}U^T)^{ n/2}   \nu(dR) \notag\\
&=&\int_{\mathcal{P_{\ell}(\R)}} \det(R)^{\frac{n-\ell}{2}} L_{\kappa}^{(\frac{n-\ell-1}{2})}(2^{-1}U\Sigma^{-1}U^TR) 
\etr{ -2^{-1}U\Sigma^{-1}U^TR}  \nu(dR) \notag\\
&&\hspace{8cm}\times \frac{1}{\Gamma_{\ell}(\frac{n}{2})}
\frac{1}{2^{n\ell/2}} \det(U\Sigma^{-1}U^T)^{ n/2} \notag \\
&=:&\frac{1}{\Gamma_{\ell}(\frac{n}{2})}
\frac{1}{2^{n\ell/2}} \det(U\Sigma^{-1}U^T)^{ n/2} \cdot I_{\kappa}(U;\Sigma), \label{ZUexp}
\end{eqnarray}
where 
\begin{gather}\label{I}
I_{\kappa}(U;\Sigma)=\int_{\mathcal{P_{\ell}(\R)}} \det(R)^{\frac{n-\ell}{2}} L_{\kappa}^{(\frac{n-\ell-1}{2})}(2^{-1}U\Sigma^{-1}U^TR) 
\etr{ -2^{-1}U\Sigma^{-1}U^TR}  \nu(dR).
\end{gather}
Exploiting identity \eqref{LaplaceLag} with $\gamma=(n-\ell-1)/2, t=(n+1)/2$ and $A=B=2^{-1}U\Sigma^{-1}U^T$ yields  
\begin{eqnarray*}
I_{\kappa}(U;\Sigma)
&=&\bigg(\frac{n}{2}\bigg)_{\kappa} C_{\kappa}(\Id_{\ell})
\sum_{s=0}^{k} \sum_{\sigma \vdash s}
{\kappa \choose \sigma} \frac{(-1)^s}{(\frac{n}{2})_{\sigma}}\frac{1}{C_{\sigma}(\Id_{\ell})}
\bigg(\frac{n+1}{2}\bigg)_{\sigma} \Gamma_{\ell}\bigg(\frac{n+1}{2}\bigg) \notag \\
&&\hspace{5cm}\times \det(2^{-1}U\Sigma^{-1}U^T)^{-(n+1)/2}C_{\sigma}(\Id_{\ell}) \\
&=& \det(2^{-1}U\Sigma^{-1}U^T)^{-(n+1)/2} 
\bigg(\frac{n}{2}\bigg)_{\kappa} C_{\kappa}(\Id_{\ell})
\Gamma_{\ell}\bigg(\frac{n+1}{2}\bigg)
\sum_{s=0}^{k} \sum_{\sigma \vdash s}
{\kappa \choose \sigma} (-1)^s\frac{(\frac{n+1}{2})_{\sigma}}{(\frac{n}{2})_{\sigma}}  \\
&=&\det(U\Sigma^{-1}U^T)^{-(n+1)/2} 
2^{\ell(n+1)/2}\bigg(\frac{n}{2}\bigg)_{\kappa} C_{\kappa}(\Id_{\ell})
\Gamma_{\ell}\bigg(\frac{n+1}{2}\bigg)
\sum_{s=0}^{k} \sum_{\sigma \vdash s}
{\kappa \choose \sigma} (-1)^s\frac{(\frac{n+1}{2})_{\sigma}}{(\frac{n}{2})_{\sigma}} \\
&=& \det(U\Sigma^{-1}U^T)^{-(n+1)/2} \cdot d_{\kappa},
\end{eqnarray*}
where 
\begin{gather}\label{dk}
d_{\kappa}:=
2^{\ell(n+1)/2}\bigg(\frac{n}{2}\bigg)_{\kappa} C_{\kappa}(\Id_{\ell})
\Gamma_{\ell}\bigg(\frac{n+1}{2}\bigg)
\sum_{s=0}^{k} \sum_{\sigma \vdash s}
{\kappa \choose \sigma} (-1)^s\frac{(\frac{n+1}{2})_{\sigma}}{(\frac{n}{2})_{\sigma}}.
\end{gather}
Replacing this expression into the RHS of \eqref{ZUexp} eventually gives
\begin{eqnarray*}
Z_{\kappa}(U;\Sigma) 
&=&\frac{1}{\Gamma_{\ell}(\frac{n}{2})}
\frac{1}{2^{n\ell/2}} \det(U\Sigma^{-1}U^T)^{ n/2} \det(U\Sigma^{-1}U^T)^{-(n+1)/2} \cdot d_{\kappa} \\
&=&d_{\kappa} \frac{1}{\Gamma_{\ell}(\frac{n}{2})}
\frac{1}{2^{n\ell/2}} \det(U\Sigma^{-1}U^T)^{-1/2}. 
\end{eqnarray*}
Taking expectations with respect to $U$ gives from \eqref{FSigma} 
\begin{eqnarray*}
\widehat{F}(\kappa;\Sigma)
&=& c(\kappa;\Sigma)^{-1} \det(\Sigma)^{\ell k} \gamma_{\kappa}
\EX{U}{Z_{\kappa}(U;\Sigma)} \\
&=&  c(\kappa;\Sigma)^{-1}
\det(\Sigma)^{\ell k} \gamma_{\kappa} d_{\kappa} \frac{1}{\Gamma_{\ell}(\frac{n}{2})}
\frac{1}{2^{n\ell/2}}  \EX{U}{\det(U\Sigma^{-1}U^T)^{-1/2}}.
\end{eqnarray*}
The expectation with respect to $U$ is computed using \eqref{denU}, 
\begin{eqnarray*}
\EX{U}{\det(U\Sigma^{-1}U^T)^{-1/2}}
&=& \int_{O(\ell,n)} \det(U\Sigma^{-1}U^T)^{-1/2} f_U(U) \tilde{\mu}(dU) \\
&=& \det(\Sigma)^{-\ell/2} \int_{O(\ell,n)}   \det(U\Sigma^{-1}U^T)^{-(n+1)/2}  \tilde{\mu}(dU) .
\end{eqnarray*}
Replacing  this expression into the previous relation, we conclude that
\begin{gather*}
\widehat{F}(\kappa; \Sigma)
= c(\kappa;\Sigma)^{-1}\det(\Sigma)^{\ell k-\ell/2} \gamma_{\kappa} d_{\kappa} \frac{1}{\Gamma_{\ell}(\frac{n}{2})}
\frac{1}{2^{n\ell/2}}   \int_{O(\ell,n)}   \det(U\Sigma^{-1}U^T)^{-(n+1)/2}  \tilde{\mu}(dU) \\
=  \det(\Sigma)^{-\ell k}4^{k}
\bigg(\frac{n}{2}\bigg)_{\kappa} \frac{1}{ k !C_{\kappa}(\Id_{\ell})} 
\det(\Sigma)^{-\ell/2} \gamma_{\kappa} d_{\kappa} 
\frac{1}{\Gamma_{\ell}(\frac{n}{2})}
\frac{1}{2^{n\ell/2}}  \times\int_{O(n,\ell)}   \det(U\Sigma^{-1}U^T)^{-(n+1)/2}  \tilde{\mu}(dU),
\end{gather*}
where we used the definition of $c(\kappa;\Sigma)$ in \eqref{cSigma}. Combining this expression with the definitions of $\gamma_{\kappa}$ in \eqref{RelHerLag} and $d_{\kappa}$ in \eqref{dk}, yields after simplications
\begin{eqnarray*}
\widehat{F}(\kappa; \Sigma) 
&=& \frac{ (-2)^{k}  }{\det(\Sigma)^{\ell k} k!}\bigg(\frac{n}{2}\bigg)_{\kappa} 
\sum_{s=0}^{k} \sum_{\sigma\vdash s}
{\kappa \choose \sigma} (-1)^s\frac{(\frac{n+1}{2})_{\sigma}}{(\frac{n}{2})_{\sigma}} \\
&&\times
\det(\Sigma)^{-\ell/2}   2^{\ell/2}  \frac{\Gamma_{\ell}(\frac{n+1}{2})}{\Gamma_{\ell}(\frac{n}{2})}  \int_{O(n,\ell)}   \det(U\Sigma^{-1}U^T)^{-(n+1)/2} \tilde{\mu}(dU),
\end{eqnarray*}
which finishes the proof.
\end{proof}

\subsubsection*{Proof of Theorem \ref{InterVol}}
In order to prove \eqref{For2}, it is sufficient to prove the relation  
\begin{gather}
\det(\Sigma)^{-\ell/2}   2^{\ell/2}  \frac{\Gamma_{\ell}(\frac{n+1}{2})}{\Gamma_{\ell}(\frac{n}{2})}
\int_{O(n,\ell)}   \det(U\Sigma^{-1}U^T)^{-(n+1)/2}  \tilde{\mu}(dU)  
= \frac{(n)_{\ell} }{(2\pi)^{\ell/2} \kappa_{n-\ell}} V(\mathcal{E}_{\Sigma}[\ell],\mathbb{B}_n[n-\ell]) , \label{Identity}
\end{gather}
since then \eqref{For2}  directly follows after combining \eqref{Identity} with \eqref{For1}. Let us now prove \eqref{Identity}.
A direct computation shows that
\begin{gather} \label{Const}  
2^{\ell/2}  \frac{\Gamma_{\ell}(\frac{n+1}{2})}{\Gamma_{\ell}(\frac{n}{2})}
= \frac{(n)_{\ell} }{(2\pi)^{\ell/2} } \frac{\kappa_{n}}{\kappa_{n-\ell}}, \quad \kappa_{n}= \frac{\pi^{n/2}}{\Gamma(1+n/2)}.
\end{gather}
Since the mixed volume on the RHS of \eqref{Identity} only involves the convex bodies $\mathcal{E}_{\Sigma}$ and $\mathbb{B}_n$, we can use \eqref{IntMix} to represent it as an intrinsic volume, 
\begin{gather}\label{RelMI}
V(\mathcal{E}_{\Sigma}[\ell],\mathbb{B}_n[n-\ell])  = \frac{ \kappa_{n-\ell} }{{n \choose \ell}} V_{\ell}(\mathcal{E}_{\Sigma}).
\end{gather}
Using the integral representation \eqref{Kubota} for the $\ell$-th intrinsic volume yields
\begin{gather*}
V_{\ell}(\mathcal{E}_{\Sigma}) = {n\choose \ell}\frac{\kappa_n}{\kappa_{\ell}\kappa_{n-\ell}} 
\int_{G(n,\ell)} \mathrm{vol}_{\ell}(\mathcal{E}_{\Sigma}|\mathscr{U} ) \nu_{n,\ell}(d\mathscr{U} ), 
\end{gather*}
where $\nu_{n,\ell}$ is the Haar probability measure on the Grassmannian $G(n,\ell)$.
Combining this with \eqref{Const} shows that the identity in \eqref{Identity} is equivalent to 
\begin{gather}\label{Identity2}
\det(\Sigma)^{-\ell/2}   \int_{O(n,\ell)}   \det(U\Sigma^{-1}U^T)^{-(n+1)/2}  \tilde{\mu}(dU)
= \frac{1}{\kappa_{\ell}} \int_{G(n,\ell)} \mathrm{vol}_{\ell}(\mathcal{E}_{\Sigma}|\mathscr{U} ) \nu_{n,\ell}(d\mathscr{U} ).
\end{gather}
Therefore it remains to prove \eqref{Identity2}.
We   rewrite the LHS of \eqref{Identity2} as follows
\begin{gather*}
\int_{O(n,\ell)}   \det(U\Sigma^{-1}U^T)^{-1/2} \Pi_{n,\ell}(dU)
=  \int_{O(n,\ell)}   \det([U\Sigma^{-1}U^T]^{-1})^{1/2} \Pi_{n,\ell}(dU),
\end{gather*}
where $\Pi_{n,\ell}(dU)= \det(\Sigma)^{-\ell/2} \det(U\Sigma^{-1}U^T)^{-n/2}\tilde{\mu}(dU) $ is a probability measure on $O(n,\ell)$ by virtue of \eqref{denU}. 
We now argue that 
\begin{gather*}
\int_{O(n,\ell)}   \det([U\Sigma^{-1}U^T]^{-1})^{1/2} \Pi_{n,\ell}(dU) 
= \int_{O(n,\ell)}   \det(U \Sigma U^T)^{1/2} \Pi_{n,\ell}(dU).
\end{gather*}
In order to see this, let us write $\Sigma = O \Lambda O^T$ for $O \in O(n), \Lambda=\mathrm{diag}(\lambda_1,\ldots, \lambda_n)$. Then, we have $\det(U\Sigma^{-1}U^T) = \det(W \Lambda^{-1} W^T) $ with $W=UO \in O(n,\ell)$ since $WW^T = UO(UO)^T= \Id_{\ell}$. 
Therefore, it suffices to consider the case where $\Sigma=\Lambda$ is diagonal. Moreover, since $W \in O(n,\ell)$ we have  for every $Q\in O(\ell)$ that  $(QW)(QW)^T=QWW^TQ^T=\Id_{\ell}$, that is $QW \in O(n,\ell)$. This implies that, up to rotating the matrix $W=UO$, we can assume that the rows of $W$ coincide with the $\ell$ first canonical basis vectors $e_1,\ldots, e_{\ell}$ in $\R^n$. Then, we compute
\begin{eqnarray*}
\det( [W \Lambda^{-1} W^T]^{-1} )
= \det( W\Lambda^{-1} W^T )^{-1} = \left(\prod_{i=1}^{\ell}\lambda_i^{-1}\right)^{-1}
= \prod_{i=1}^{\ell}\lambda_i 
= \det(W\Lambda W^T).
\end{eqnarray*}
Therefore, integrating on $O(n,\ell)$ and noting that 
$\Pi_{n,\ell}(d(QU))=\Pi_{n,\ell}(dU)$ for every $Q \in O(\ell)$ yields the claim. Now, since  $\Pi_{n,\ell}$ is left-invariant by orthogonal transformations, it can be viewed as a probability measure on $O(n,\ell)/O(\ell) \simeq G(n,\ell)$, where  two elements $U_1,U_2$ in $O(n,\ell)$ are equivalent if and only if there exists $Q\in O(\ell)$ such that $U_1=QU_2$. Thus, since $\nu_{n,\ell}$ is the unique left and right-invariant Haar probability measure on $G(n,\ell)$, it must coincide with $\Pi_{n,\ell}$.
 Writing $\mathscr{U}$ for the $\ell$-dimensional linear subspace generated by the rows of $U$, we have that the matrix $U\Sigma U^T$ represents the ellipsoid $\mathcal{E}_{\Sigma}|\mathscr{U}$ of volume $\mathrm{vol}_{\ell}(\mathcal{E}_{\Sigma}|\mathscr{U})=\kappa_{\ell}\det(U\Sigma U^T)^{1/2}$, implying in turn
\begin{gather*}
\int_{O(n,\ell)}   \det(U \Sigma U^T)^{1/2} \Pi_{n,\ell}(dU)
= \int_{G(n,\ell)} \frac{1}{\kappa_{\ell}} \mathrm{vol}_{\ell}(\mathcal{E}_{\Sigma}|\mathscr{U}) \nu_{n,\ell}(d\mathscr{U}).
\end{gather*}
This proves \eqref{Identity2} and thus \eqref{Identity}. Formula \eqref{For3} follows from \eqref{For2} and relation \eqref{RelMI}. Formula \eqref{Kab} is obtained when setting $\kappa=(0)$ in \eqref{For2} and \eqref{For3}, respectively, and using the fact that $\widehat{F}((0);\Sigma)=\EX{X}{F(X)}$.

\subsubsection{An attempt at generalizing to distinct covariance matrices}\label{General}
In this section, we try to generalize the results of Theorem \ref{CoeffSigma} and Theorem \ref{InterVol} to the more general setting where the rows of $X$ are independent Gaussian vectors with distinct covariance matrices. 

\medskip
Let $\{\Sigma_i \in \R^{n\times n} : i\in [\ell]\}$ be  positive-definite symmetric matrices and $\{X^{(i)}=(X^{(i)}_1,\ldots,X^{(i)}_n):i\in [\ell]\}$  a collection of $\ell$ independent Gaussian vectors with respective covariance matrices  $\Sigma_1,\ldots,\Sigma_{\ell}$. 
We write $X$ for the $\ell \times n$ matrix whose $i$-th row is $X^{(i)}$.
Then, the vector $\mathrm{Vec}(X^T)$ has the multivariate normal distribution $\mathcal{N}_{\ell n}(0,\Omega)$, where 
\begin{gather*}
\Omega= \sum_{i=1}^{\ell} \left( e_ie_i^T \otimes \Sigma_i\right) = \mathrm{diag}(\Sigma_1,\ldots,\Sigma_{\ell}) 
= \Sigma_1 \oplus \ldots \oplus \Sigma_{\ell},
\end{gather*}
with $e_i \in \R^{\ell}$ denoting the $i$-th canonical basis vector.  The density function of $X$ is given by  
\begin{gather*}
\phi_{\Omega}(\mathrm{Vec}(X^T)) = 
(2\pi)^{-n\ell/2}  \det(\Omega)^{-n\ell/2} \etr{ -\frac{1}{2} \Omega^{-1}\mathrm{Vec}(X^T)\mathrm{Vec}(X^T)^T}.
\end{gather*}
If $X $ is distributed as above, a computation shows that the $\ell \times n$ matrix
\begin{gather*}
Y_X:= (\Id_{\ell} \otimes \mathrm{Vec}(X^T)^T \Omega^{-1/2})
(\mathrm{Vec}(\Id_{\ell} ) \otimes \Id_n) 
\end{gather*}
has the standard matrix normal distribution. Therefore, we consider the matrix-variate polynomials
\begin{gather}\label{HerOmega}
H^{(\ell,n)}_{\kappa}(X;\Omega) = \det(\Omega)^{k} H_{\kappa}^{(\ell,n)}(Y_X) , \quad \kappa \vdash k
\end{gather}
satisfying the orthogonality relation (similar as in the proof of Proposition \ref{DefSigma})
\begin{gather*}
\int_{\R^{\ell \times n}}
H^{(\ell,n)}_{\kappa}(X;\Omega) 
H^{(\ell,n)}_{\sigma}(X;\Omega) \phi_{\Omega}(\mathrm{Vec}(X^T)) (dX) 
= \ind{\kappa=\sigma} \times \det(\Omega)^{2k }  
4^{-k}\bigg(\frac{n}{2}\bigg)_{\kappa}^{-1} k !C_{\kappa}(\Id_{\ell}), 
\end{gather*}
and thus the family 
\begin{gather*}
\mathbb{H}_{\Omega}:= \left\{ c(\kappa;\Omega)^{-1/2} H_{\kappa}^{(\ell,n)}(\cdot;\Omega): \kappa \vdash k \geq0\right\} , \quad 
c(\kappa;\Omega):=\det(\Omega)^{2 k}4^{-k}\bigg(\frac{n}{2}\bigg)_{\kappa}^{-1} k!C_{\kappa}(\Id_{\ell})
\end{gather*}
forms an orthonormal system of $L^2(\mu_{X})$, where, as usual, $\mu_{X}$ indicates the spectral measure associated with $XX^T$. Expanding the function $F(X)=\det(XX^T)^{1/2} \in L^2(\mu_{X})$ in the basis $\mathbb{H}_{\Omega}$, using the relation 
\begin{gather*}
\phi_{\Omega}(\mathrm{Vec}(X^T)) = \det(\Omega)^{-1/2} \phi^{(\ell,n)}(Y_X) 
\end{gather*}
and the definition of $H_{\kappa}^{(\ell,n)}(\cdot; \Omega)$, we have that the associated projection coefficients are 
\begin{eqnarray*}
 \widehat{F}(\kappa;\Omega) &=& c(\kappa;\Omega)^{-1}  
\int_{\R^{\ell \times n}} F(X) H^{(\ell,n)}_{\kappa}(X;\Omega) \phi_{\Omega}(\mathrm{Vec}(X^T)) (dX) \\
&=& \det(\Omega)^{k-1/2}  \int_{\R^{\ell \times n}} F(X) H_{\kappa}^{(\ell,n)}(Y_X)  \phi^{(\ell,n)}(Y_X) (dX)\\
&=& \det(\Omega)^{k-1/2} \gamma_{\kappa}
\int_{\R^{\ell \times n}} F(X)  L_{\kappa}^{(\frac{n-\ell-1}{2})}(2^{-1}Y_XY_X^T)  (2\pi)^{-n\ell/2}\etr{-2^{-1}Y_XY_X^T} (dX).
\end{eqnarray*}
The idea is now to perform the polar change of variables $X= R^{1/2}U$. In order to do so, we compute $\Tr(Y_XY_X^T)$:
\begin{gather*}
\Tr(Y_XY_X^T) = \Tr(\Omega^{-1} \mathrm{Vec}(X^T)\mathrm{Vec}(X^T)^T) 
= \mathrm{Vec}(X^T)^T \Omega^{-1} \mathrm{Vec}(X^T) \\
= \mathrm{Vec}(X^T)^T \sum_{i=1}^{\ell} \left( e_ie_i^T \otimes \Sigma_i^{-1}\right) \mathrm{Vec}(X^T) 
= \sum_{i=1}^{\ell}\mathrm{Vec}(X^T)^T \left( e_ie_i^T \otimes \Sigma_i^{-1}\right) \mathrm{Vec}(X^T).
\end{gather*}
Then, using the relation 
$\mathrm{Vec}(S)^T (BD\otimes E) \mathrm{Vec}(S)=\Tr(DS^TESB)$ (see e.g. \cite[Theorem 1.2.22]{Gupta2018}), 
we obtain 
\begin{eqnarray*}
\Tr(Y_XY_X^T) 
&=& \sum_{i=1}^{\ell}\mathrm{Vec}(X^T)^T \left( e_ie_i^T \otimes \Sigma_i^{-1}\right) \mathrm{Vec}(X^T)  = \sum_{i=1}^{\ell} \Tr\left( e_i^T R^{1/2} U \Sigma^{-1}_i U^T R^{1/2} e_i
\right)  \\
&=& \sum_{i=1}^{\ell} \Tr\left( e_ie_i^T R^{1/2} U \Sigma^{-1}_i U^T R^{1/2}  \right) =  \Tr\left(\sum_{i=1}^{\ell}e_ie_i^T R^{1/2} U \Sigma^{-1}_i U^T R^{1/2}\right).
\end{eqnarray*}
The difficulty to proceed now is the following: the above computation suggests that we cannot write $\Tr(Y_XY_X^T)$ as $\Tr(AR)$ for some matrix $A$, due to the fact that one cannot exploit the permutation invariance of the trace in view of 
presence of the matrix $e_i e_i^T$. We remark that, when $\Sigma_i=\Sigma$ for every $i=1,\ldots,\ell$, the above formula gives $\Tr(Y_XY_X^T)= \Tr(\Id_{\ell}U\Sigma^{-1}U^T)=\Tr(U\Sigma^{-1}U^T) $, which coincides with our computations in the proof of Theorem \ref{CoeffSigma}. This observation makes it in particular difficult to directly apply the integration formula \eqref{intzonal}, and thus hints to the fact that the polynomials $H_{\kappa}^{(\ell,n)}(\cdot;\Omega)$ are not easily amenable to matrix calculus. 

\subsection{Proofs of Section \ref{SecOrtho}}

\subsubsection*{Proofs of Theorem \ref{Action} and Theorem \ref{OrtA}}
Our proofs of Theorem \ref{Action} and Theorem \ref{OrtA} involve auxiliary polynomials introduced in \cite{Hay}. 
For $X \in \R^{\ell \times n}$ and $ A \in \R^{n\times n}$ symmetric, we consider the polynomials $P_{\kappa}(X,A), \ \kappa \vdash k$  defined by (see \cite[Eq.(34)]{Hay}) 
\begin{gather*}
\etr{-XX^T}P_{\kappa}(X,A)=\frac{(-1)^k}{\pi^{n\ell/2}} \int_{\R^{\ell\times n}} \etr{-2iXU^T}\etr{-UU^T}C_{\kappa}(UAU^T) (dU). 
\end{gather*}
These polynomials have the following properties (see e.g. \cite[p.229]{BFZP} and  \cite[Section 6]{Hay}).
\begin{Lem}
For every $k\geq 0, \kappa \vdash k$ and symmetric
$A\in \R^{n\times n}$, we have that   
\begin{gather}
P_{\kappa}(X,\Id_n) = 2^k\left(\frac{n}{2}\right)_{\kappa} H_{\kappa}^{(\ell,n)}(\sqrt{2}X) \label{P1}\\
\int_{O(n)} P_{\kappa}(XH,A) \tilde{\mu}(dH) = \int_{O(n)} P_{\kappa}(X,HAH^T) \tilde{\mu}(dH) = \frac{C_{\kappa}(A)}{C_{\kappa}(\Id_n)} P_{\kappa}(X,\Id_n) \label{P2}\\
P_{\kappa}(X,A) = \EX{V}{C_{\kappa}((X+iV)A(X+iV)^T)} , \quad V \sim \mathscr{N}(0,\Id_{\ell}/2\otimes\Id_n) . \label{P3}
\end{gather}
\end{Lem}

In order to prove  Theorem \ref{Action}, we shall first show the following Lemma, linking the conditional expectation of $H_{\kappa}^{(\ell,n)}$ with the polynomial $P_{\kappa}$ introduced above. 
\begin{Lem}
Let $k\geq 0$ be an integer, $\kappa\vdash k$ a partition of $k$ and $\Delta=\mathrm{diag}(d_1,\ldots,d_n)$ a diagonal matrix with $|d_i|\leq 1$ for $i=1,\ldots,n$. Then, for $X_0 \sim \mathscr{N}_{\ell\times n}(0,\Id_{\ell}\otimes \Id_n)$, we  have for every $X \in \R^{\ell \times n}$,
\begin{gather}\label{Aux}
\EX{X_0}{H_{\kappa}^{(\ell,n)}( X\Delta+X_0(\Id_n-\Delta^2)^{1/2}) }
= 2^{-k}\left(\frac{n}{2}\right)_{\kappa}^{-1} P_{\kappa}\left(\frac{X}{\sqrt{2}},\Delta^2\right).
\end{gather}
\end{Lem}
\begin{proof}
For $W=(W_{lj}) \in \R^{\ell\times n}$ we use the implicit representation \eqref{HOM} of $C_{\kappa}(WW^T)$ as homogeneous polynomials of degree $2k$ in the entries of $W$, 
\begin{gather*}
C_{\kappa}(WW^T) = \sum_{|\alpha|=2k} z_{\alpha}^{\kappa} \prod_{l=1}^{\ell}\prod_{j=1}^{n} W_{lj}^{\alpha_{ij}}.
\end{gather*}
Then, using \eqref{P3} with $B=\mathrm{diag}(b_1,\ldots,b_n)$ such that $b_1,\ldots,b_n\geq0$, we can write 
\begin{eqnarray*}
P_{\kappa}(W,B) &=& \EX{V}{C_{\kappa}((W+iV)B(W+iV)^T)} \\
&=& \EX{V}{C_{\kappa}((WB^{1/2}+iVB^{1/2})(WB^{1/2}+iVB^{1/2})^T)} \\
&=& \sum_{|\alpha|=2k} z_{\alpha}^{\kappa} \prod_{l=1}^{\ell}\prod_{j=1}^{n} \EX{V_{lj}}{(W_{lj}\sqrt{b_{j}}+iV_{lj}\sqrt{b_{j}})^{\alpha_{lj}}} \\
&=&  \sum_{|\alpha|=2k} z_{\alpha}^{\kappa} \prod_{l=1}^{\ell}\prod_{j=1}^{n} b_{j}^{\alpha_{lj}/2}\EX{V_{lj}}{(W_{lj}+iV_{lj})^{\alpha_{lj}}} .
\end{eqnarray*}
Using the one-dimensional representation of Hermite polynomials as Gaussian expectation, 
\begin{gather*}
\EX{V_{lj}}{(W_{lj}+iV_{lj})^{\alpha_{lj}}} = 2^{-\alpha_{lj}/2}H_{\alpha_{lj}}(\sqrt{2}W_{lj}) 
\end{gather*}
leads to  
\begin{eqnarray}\label{PWB}
P_{\kappa}(W,B) &=& \sum_{|\alpha|=2k} z_{\alpha}^{\kappa} \prod_{l=1}^{\ell}\prod_{j=1}^{n} b_{j}^{\alpha_{lj}/2} 2^{-\alpha_{lj}/2}H_{\alpha_{lj}}(\sqrt{2}W_{lj}) \notag\\
&=& 2^{-k}\sum_{|\alpha|=2k} z_{\alpha}^{\kappa} \prod_{l=1}^{\ell}\prod_{j=1}^{n} b_{j}^{\alpha_{lj}/2} H_{2\alpha_{lj}}(\sqrt{2}W_{lj}).
\end{eqnarray}
Applying \eqref{PWB} with $W=X/\sqrt{2}$ and $B=\Delta^2$, yields 
\begin{gather}\label{Eq1}
P_{\kappa}\left(\frac{X}{\sqrt{2}},\Delta^2\right) = 2^{-k}\sum_{|\alpha|=2k} z_{\alpha}^{\kappa} \prod_{l=1}^{\ell}\prod_{j=1}^{n} d_j^{\alpha_{lj}} H_{\alpha_{lj}}(X_{lj}).
\end{gather}
On the other hand, applying \eqref{PWB} with $W=(X/\sqrt{2})\Delta+(X_0/\sqrt{2})(\Id_n-\Delta^2)^{1/2}$ and $B=\Id_n$, we have 
\begin{gather}\label{Eq2}
P_{\kappa}\left( \frac{X}{\sqrt{2}}\Delta+\frac{X_0}{\sqrt{2}}(\Id_n-\Delta^2)^{1/2},\Id_n\right)
= 2^{-k}\sum_{|\alpha|=2k} z_{\alpha}^{\kappa} \prod_{l=1}^{\ell}\prod_{j=1}^{n}  H_{\alpha_{lj}}\left(d_jX_{lj}+\sqrt{1-d_j^2}X_{0,lj}\right).
\end{gather}
Taking expectation with respect to $X_0$ in \eqref{Eq2}, we infer 
\begin{eqnarray*}
&&\EX{X_0}{H_{\kappa}^{(\ell,n)}( X\Delta+X_0(\Id_n-\Delta^2)^{1/2})}\\
&=& 2^{-k}\left(\frac{n}{2}\right)_{\kappa}^{-1} 
\EX{X_0}{P_{\kappa}\left( \frac{X}{\sqrt{2}}\Delta+\frac{X_0}{\sqrt{2}}(\Id_n-\Delta^2)^{1/2},\Id_n\right) } \qquad (\mathrm{by \ }  \eqref{P1})\\
&=&  2^{-k}\left(\frac{n}{2}\right)_{\kappa}^{-1} \EX{X_0}{2^{-k}\sum_{|\alpha|=2k} z_{\alpha}^{\kappa} \prod_{l=1}^{\ell}\prod_{j=1}^{n} H_{\alpha_{lj}}\left(d_jX_{lj}+\sqrt{1-d_j^2}X_{0,lj}\right)} \qquad (\mathrm{by \ }  \eqref{Eq2}) \\
&=& 2^{-k}\left(\frac{n}{2}\right)_{\kappa}^{-1} 2^{-k}\sum_{|\alpha|=2k} z_{\alpha}^{\kappa} \prod_{l=1}^{\ell}\prod_{j=1}^{n} \EX{X_{0,lj}}{ H_{\alpha_{lj}}\left(d_jX_{lj}+\sqrt{1-d_j^2}X_{0,lj}\right)} \qquad (\mathrm{by \ independence})\\
&=&2^{-k}\left(\frac{n}{2}\right)_{\kappa}^{-1} 2^{-k}\sum_{|\alpha|=2k} z_{\alpha}^{\kappa} \prod_{l=1}^{\ell}\prod_{j=1}^{n} d_j^{\alpha_{lj}}H_{\alpha_{lj}}(X_{lj})   \qquad (\mathrm{by \ } \eqref{ActPt}) \\
&=& 2^{-k}\left(\frac{n}{2}\right)_{\kappa}^{-1} P_{\kappa}\left( \frac{X}{\sqrt{2}},\Delta^2\right),     \qquad (\mathrm{by \ }  \eqref{Eq1})
\end{eqnarray*}
which proves relation \eqref{Aux}. 
\end{proof}
We are now in position to prove Theorem \ref{Action}.
\begin{proof}[Proof of Theorem \ref{Action}]
In order to prove \eqref{OURel}, we use Fubini and apply \eqref{Aux} with $\Delta=e^{-tA}$ and 
integrate both sides with respect to the Haar measure on $O(n)$ to obtain: 
\begin{eqnarray*}
\mathcal{O}_{t;A}^{(\ell,n)}H_{\kappa}^{(\ell,n)}(X) &=& \E{ \int_{O(n)} H_{\kappa}^{(\ell,n)}( XH e^{-tA}+X_0(\Id_n-e^{-2tA})^{1/2}) \tilde{\mu}(dH)\bigg|X}\\
&=& \int_{O(n)}\E{  H_{\kappa}^{(\ell,n)}( XH e^{-tA}+X_0(\Id_n-e^{-2tA})^{1/2}) \bigg|X}\tilde{\mu}(dH)\\
&=& 2^{-k}\left(\frac{n}{2}\right)_{\kappa}^{-1} \int_{O(n)} P_{\kappa}\left(\frac{XH}{\sqrt{2}},e^{-2tA}\right) \tilde{\mu}(dH)  \\
&=& 2^{-k}\left(\frac{n}{2}\right)_{\kappa}^{-1} \frac{C_{\kappa}(e^{-2tA})}{C_{\kappa}(\Id_n)} P_{\kappa}\left(\frac{X}{\sqrt{2}},\Id_n\right)  
= \frac{C_{\kappa}(e^{-2tA})}{C_{\kappa}(\Id_n)} H_{\kappa}^{(\ell,n)}(X),
\end{eqnarray*}
where we used \eqref{P1} and \eqref{P2}. This finishes the proof of the first part of the statement. Let us now prove the second part: Assume first that $A= \mathrm{diag}(a,\ldots,a)$ and let $f \in \Pi(\ell,n)$ (see \eqref{Piln}). Then, one has that 
\begin{eqnarray*}
\mathcal{O}_{t;A}^{(\ell,n)}f(X) 
&=&\int_{O(n)}\E{  f(e^{-at}XH+\sqrt{1-e^{-2at}}X_0H) \bigg|X}\tilde{\mu}(dH)\notag\\
&=&\int_{O(n)}\E{  f(e^{-at}X +\sqrt{1-e^{-2at}}X_0 ) \bigg|X}\tilde{\mu}(dH)
= P_{at}^{(\ell n)} f(X),\label{Eq:Coincide}
\end{eqnarray*}
where we used the facts that $X_0\eqLaw X_0H$ for $H\in O(n)$, $f$ is an element of $\Pi(\ell,n)$  and   $\tilde{\mu}$ is a probability measure on $O(n)$. Finally, if the $a_i$'s are not all equal, then arguing as in Remark \ref{RemOU} \textbf{(b)}, one can derive  a relation contradicting the semigroup property of $\mathcal{O}_{t;A}^{(\ell,n)}$.
\end{proof}
\begin{proof}[Proof of Theorem \ref{OrtA}]
We proceed in two steps. In view of Remark \ref{Rem:Ortho}, the matrix $R$ is necessarily symmetric and has non-negative eigenvalues. We start by showing that \eqref{RelOrtA} holds for diagonal matrices $R=\mathrm{diag}(r_1,\ldots,r_n)$. The statement for arbitrary  symmetric matrices will then follow from the diagonal case by a reduction argument.

\medskip
\noindent\underline{\textit{Step 1: $R$ is diagonal.}} Let us first assume that $r_1,\ldots,r_n>0$. Since   $X \eqLaw XH, H \in O(n)$ and using the fact that $H_{\kappa}^{(\ell,n)}(XH)=H_{\kappa}^{(\ell,n)}(X)$ for every $H \in O(n)$ (as can be seen e.g. from \eqref{HerZon} or \eqref{RelHerLag}), we have  
\begin{eqnarray}\label{comp}
&&\E{H_{\kappa}^{(\ell,n)}(X) H_{\sigma}^{(\ell,n)}(XR+X_0(\Id_n-R^2)^{1/2})} \notag\\
&=& \E{H_{\kappa}^{(\ell,n)}(X) \E{H_{\sigma}^{(\ell,n)}(XR+X_0(\Id_n-R^2)^{1/2}) \big|X}} \notag\\
&=& \E{\int_{O(n)}H_{\kappa}^{(\ell,n)}(XH) \E{H_{\sigma}^{(\ell,n)}(XHR+X_0(\Id_n-R^2)^{1/2}) \big|X}\tilde{\mu}(dH)}\notag \\
&=& \E{H_{\kappa}^{(\ell,n)}(X) \E{ \int_{O(n)}H_{\sigma}^{(\ell,n)}(XHR+X_0(\Id_n-R^2)^{1/2}) \tilde{\mu}(dH) \bigg|X}}\\
&=& \E{H_{\kappa}^{(\ell,n)}(X) \mathcal{O}_{1;R_*}^{(\ell,n)}H_{\sigma}^{(\ell,n)}(X)} ,\notag
\end{eqnarray}
where $R_*:=\mathrm{diag}(\ln(1/r_1),\ldots,\ln(1/r_n))$.  
Then, exploiting the action of $\mathcal{O}_{t;R_*}^{(\ell,n)}$ on matrix-variate Hermite polynomials given in \eqref{OURel} we infer 
\begin{gather*} 
\E{H_{\kappa}^{(\ell,n)}(X) H_{\sigma}^{(\ell,n)}(XR+X_0(\Id_n-R^2)^{1/2})} 
=  \frac{C_{\kappa}(e^{-2R_*})}{C_{\kappa}(\Id_n)}\E{H_{\kappa}^{(\ell,n)}(X)H_{\sigma}^{(\ell,n)}(X)}  \\
= \frac{C_{\kappa}(R^2)}{C_{\kappa}(\Id_n)}\E{H_{\kappa}^{(\ell,n)}(X)H_{\sigma}^{(\ell,n)}(X)} 
=\ind{\kappa=\sigma}\times 4^{-k}\bigg(\frac{n}{2}\bigg)_{\kappa}^{-1} k ! C_{\kappa}(R^2)\frac{C_{\kappa}(\Id_{\ell})}{C_{\kappa}(\Id_n)} ,
\end{gather*}
where we used that $e^{-2R_*}=R^2$ and the orthogonality relation for Hermite polynomials \eqref{HerOrtho}.
If some of $r_1,\ldots,r_n$ are equal to zero, the conclusion remains valid, as in this case, from \eqref{comp}, we can use \eqref{Aux} and \eqref{P2} yielding the same conclusion.

\medskip
\noindent\underline{\textit{Step 2: $R$ is symmetric.}}
Since $R$ is symmetric, there exists $O \in O(n)$ such that $R=O\Delta_R O^T$, where $\Delta_R$ is diagonal. Moreover, since
$R^2=O\Delta_R^2O^T$, we have 
\begin{gather*}
\Id_n-R^2 = \Id_n-O\Delta_R^2O^T = OO^T-O\Delta_R^2O^T = O(\Id_n-\Delta_R^2)O^T
\end{gather*}
yielding 
$(\Id_n-R^2)^{1/2} = O(\Id_n-\Delta_R^2)^{1/2}O^T$
as can be seen from the computation
\begin{eqnarray*}
[O(\Id_n-\Delta_R^2)^{1/2}O^T]^2 = [O(\Id_n-\Delta_R^2)^{1/2}O^T][O(\Id_n-\Delta_R^2)^{1/2}O^T] = O(\Id_n-\Delta_R^2)O^T = \Id_n-R^2.
\end{eqnarray*}
Exploiting once more the fact that $H_{\kappa}^{(\ell,n)}(XO) = H_{\kappa}^{(\ell,n)}(X)$ for every $O \in O(n)$, we have 
\begin{eqnarray*}
&&\E{H_{\kappa}^{(\ell,n)}(X) H_{\sigma}^{(\ell,n)}(XR+X_0(\Id_n-R^2)^{1/2})} \\
&=& \E{H_{\kappa}^{(\ell,n)}(X) H_{\sigma}^{(\ell,n)}(XO \Delta_RO^T+X_0O(\Id_n-\Delta_R^2)^{1/2}O^T)} \\
&=&\E{H_{\kappa}^{(\ell,n)}(X) H_{\sigma}^{(\ell,n)}( (XO \Delta_R +X_0O(\Id_n-\Delta_R^2)^{1/2})O^T)} \\
&=&\E{H_{\kappa}^{(\ell,n)}(XO) H_{\sigma}^{(\ell,n)}( XO \Delta_R +X_0O(\Id_n-\Delta_R^2)^{1/2})} \\
&=& \E{H_{\kappa}^{(\ell,n)}(X) H_{\sigma}^{(\ell,n)}( X \Delta_R +X_0 (\Id_n-\Delta_R^2)^{1/2})},
\end{eqnarray*}
where the last equality follows from the fact that the pair $(X,X_0)$ has the same distribution as the pair $(XO,X_0O)$. Since $\Delta_R$ is diagonal, we can apply the conclusion of Step 1 to infer 
\begin{gather*}
 \E{H_{\kappa}^{(\ell,n)}(X) H_{\sigma}^{(\ell,n)}( X \Delta_R +X_0 (\Id_n-\Delta_R^2)^{1/2})}  
= \ind{\kappa=\sigma}\times 4^{-k}\bigg(\frac{n}{2}\bigg)_{\kappa}^{-1} k ! C_{\kappa}(\Delta_R^2)\frac{C_{\kappa}(\Id_{\ell})}{C_{\kappa}(\Id_n)} \\
= \ind{\kappa=\sigma}\times 4^{-k}\bigg(\frac{n}{2}\bigg)_{\kappa}^{-1} k ! C_{\kappa}(R^2)\frac{C_{\kappa}(\Id_{\ell})}{C_{\kappa}(\Id_n)},
\end{gather*}
where in the last line, we used the fact that $C_{\kappa}(\Delta_R^2) = 
C_{\kappa}(O\Delta_R^2O^T)= C_{\kappa}(R^2)$. This finishes the proof.
\end{proof}

\subsection{Proofs of Section \ref{SecApp}}
\subsubsection*{Proof of Proposition \ref{VarTV}}
The variance of the total variation is obtained from \eqref{chaos}. Using the orthogonality of Wiener chaoses, the variance of $\mathbf{V}(\mathfrak{f}_{\ell};U)$ is computed to be
\begin{gather}\label{V}
\V{\mathbf{V}(\mathfrak{f}_{\ell};U)} = \V{\sum_{k \geq1} \mathbf{V}(\mathfrak{f}_{\ell};U)[2k]}  
= \sum_{k\geq1} \V{\mathbf{V}(\mathfrak{f}_{\ell};U)[2k]} 
\end{gather}
where 
\begin{gather*}
\V{\mathbf{V}(\mathfrak{f}_{\ell};U)[2k]}
= \sum_{\kappa \vdash k} \sum_{\sigma \vdash k} \widehat{\Phi}(\kappa)\widehat{\Phi}(\sigma) \int_{U^2} \E{H_{\kappa}^{(\ell,n)}(\mathfrak{f}_{\ell}'(z))
H_{\sigma}^{(\ell,n)}(\mathfrak{f}_{\ell}'(z')) }dzdz',
\end{gather*}
with $\widehat{\Phi}(\kappa)$ as in \eqref{ForId}.
Now, in view of \eqref{CovMatrix}, we can apply Theorem \ref{OrtA} with $R=R(z,z')$ to infer
\begin{gather*}
\E{H_{\kappa}^{(\ell,n)}(\mathfrak{f}_{\ell}'(z))
H_{\sigma}^{(\ell,n)}(\mathfrak{f}_{\ell}'(z'))} 
= \ind{\kappa=\sigma}\times 4^{-k}\bigg(\frac{n}{2}\bigg)_{\kappa}^{-1} k ! C_{\kappa}(R(z,z')^2)\frac{C_{\kappa}(\Id_{\ell})}{C_{\kappa}(\Id_n)},
\end{gather*}
yielding
\begin{eqnarray*}
\V{\mathbf{V}(\mathfrak{f}_{\ell};U)[2k]}
= \sum_{\kappa\vdash k} \widehat{\Phi}(\kappa)^2
4^{-k}\bigg(\frac{n}{2}\bigg)_{\kappa}^{-1} k !
\frac{C_{\kappa}(\Id_{\ell})}{C_{\kappa}(\Id_n)} \int_{U^2} C_{\kappa}(R(z,z')^2) dzdz'.
\end{eqnarray*}
The relation in \eqref{VarV} then follows from \eqref{V}.

\subsubsection*{Proof of Theorem \ref{ARW}}
The Wiener chaos expansion of $\mathbf{V}(\bT_n^{(\ell)};\T)$ is given by \eqref{chaos}: 
\begin{gather}\label{chaosARW}
\mathbf{V}(\bT_n^{(\ell)};\T) = \left(\frac{E_n}{3}\right)^{\ell/2}\sum_{k \geq 0} \mathbf{V}(\bT_n^{(\ell)};\T)[2k], 
\end{gather}
where for $k\geq0$,
\begin{gather*}
\mathbf{V}(\bT_n^{(\ell)};\T)[2k] = \sum_{\kappa \vdash k}  \widehat{\Phi}(\kappa) \int_{\T} H_{\kappa}^{(\ell,3)}(\dot{\bT}_n^{(\ell)}(z)) dz
\end{gather*}
and $\widehat{\Phi}(\kappa)$ is as in \eqref{ForId}.
In particular, for $k=0$, we have by \eqref{coeff0}, 
\begin{gather*}
\E{\mathbf{V}(\bT_n^{(\ell)};\T)} =  
\left(\frac{E_n}{3}\right)^{\ell/2} \widehat{\Phi}((0)) = 
\left(\frac{E_n}{3}\right)^{\ell/2} 2^{\ell/2}\frac{\Gamma_{\ell}(2)}{\Gamma_{\ell}(\frac{3}{2})},
\end{gather*}
which proves \eqref{mean}. 

\medskip
\noindent\underline{\textit{Second Wiener chaos component.}}
The second Wiener chaos of $\mathbf{V}(\bT_n^{(\ell)};\T)$ is given by 
\begin{gather}\label{WC2}
\mathbf{V}(\bT_n^{(\ell)};\T)[2] = 
\left(\frac{E_n}{3}\right)^{\ell/2} \widehat{\Phi}((1)) \int_{\T} H_{(1)}^{(\ell,3)}(\dot{\bT}_n^{(\ell)}(z)) dz . 
\end{gather}
In the following lemma, we establish the asymptotic variance of the second Wiener chaos in the high-energy regime:
\begin{Lem}\label{Var2}
As $n\to \infty, \notcon{n}{0,4,7}{8}$, we have
\begin{gather*}
\V{\mathbf{V}(\bT_n^{(\ell)};\T)[2]}=
\left(\frac{E_n}{3}\right)^{\ell} 2^{\ell}  \frac{\Gamma_{\ell}(2)^2}{\Gamma_{\ell}(\frac{3}{2})^2}
\frac{\ell}{2\Nn}\left(1  + O(n^{-1/28+o(1)})\right).
\end{gather*}
\end{Lem}
\begin{proof}
Since, for every $z,z' \in \T$, we have 
\begin{gather}\label{corr}
\E{ \tilde{\partial}_jT_n^{(i)}(z) \cdot \tilde{\partial}_{j'} T_n^{(i')}(z')} 
= \ind{i=i'} \times \left(\frac{E_n}{3}\right)^{-1} r^{(n)}_{j,j'}(z-z')
\end{gather}
we note that the matrices 
$\dot{\bT}_n^{(\ell)}(z)$ and $\dot{\bT}_n^{(\ell)}(z')$ 
are such that  
\begin{gather}\label{CovARW}
\dot{\bT}_n^{(\ell)}(z')\eqLaw\dot{\bT}_n^{(\ell)}(z)R_n(z-z')
+X_0(\Id_3-R_n(z-z')^2)^{1/2},
\end{gather}
where $X_0=X_0(z,z')$ is an independent copy of $\dot{\bT}_n^{(\ell)}(z)$ and the matrix $R_n(z-z')$ is given by
\begin{gather*}
R_n(z-z'):= \left(\tilde{r}^{(n)}_{j,j'}(z-z')\right)_{j,j'\in[3]}, \quad 
\tilde{r}^{(n)}_{j,j'}(z-z'):=\left(\frac{E_n}{3}\right)^{-1} \frac{\partial^2}{\partial z_j \partial z'_{j'}} r^{(n)}(z-z'). 
\end{gather*}
Indeed, from  \eqref{CovARW} it follows that (see also   Remark \ref{Rem:Ortho} part (a))   \begin{gather*}
\E{ \tilde{\partial}_jT_n^{(i)}(z) \cdot \tilde{\partial}_{j'} T_n^{(i')}(z') } 
= \ind{i=i'} \times  \tilde{r}^{(n)}_{j,j'}(z-z'),
\end{gather*}
which is \eqref{corr}.
In particular, the variance of the second Wiener chaos component is computed by Proposition \ref{VarTV},
\begin{eqnarray}\label{Var2TOT}
\V{\mathbf{V}(\bT_n^{(\ell)};\T)[2]}
&=&  \left(\frac{E_n}{3}\right)^{\ell} \widehat{\Phi}((1))^2
4^{-1}\bigg(\frac{3}{2}\bigg)_{(1)}^{-1} 
\frac{C_{(1)}(\Id_{\ell})}{C_{(1)}(\Id_3)} \int_{\T\times \T} C_{(1)}(R_n(z-z')^2) dzdz' \notag\\
&=& \left(\frac{E_n}{3}\right)^{\ell} \widehat{\Phi}((1))^2
4^{-1}\frac{2}{3} 
\frac{\ell}{3} \int_{\T} \Tr(R_n(z)^2) dz  \notag\\
&=& \left(\frac{E_n}{3}\right)^{\ell} \widehat{\Phi}((1))^2
\frac{\ell}{18}  \int_{\T} \Tr(R_n(z)^2) dz ,
\end{eqnarray}
where we used that $C_{(1)}(A)=\Tr(A)$ and stationarity of $\bT_n^{(\ell)}$ to reduce integrations on $\T\times \T$ to $\T$.
A direct computation gives 
\begin{gather*}
\Tr(R_n(z)^2) = \sum_{j,j'\in[3]} \left(\tilde{r}^{(n)}_{j,j'}(z)\right)^2.
\end{gather*}
Now, in view of \eqref{rn} and \eqref{rnnor}, we have
\begin{gather*}
\tilde{r}^{(n)}_{j,j'}(z) = \left(\frac{E_n}{3}\right)^{-1}
(-4\pi^2) \frac{1}{\Nn} \sum_{\lambda\in \Lambda_n} \lambda_j \lambda_{j'} e_{\lambda}(z).
\end{gather*} 
Integrating over $\T$ and using the orthogonality relation for complex exponentials on the torus
\begin{gather}\label{ortoexp}
\int_{\T} e_{\lambda}(z) dz = \ind{\lambda=0}
\end{gather}
then yields 
\begin{eqnarray*}
&& \int_{\T} \Tr(R_n(z)^2) dz \\
&=&
\int_{\T} \sum_{j,j'\in[3]} \left(\tilde{r}^{(n)}_{j,j'}(z)\right)^2 dz
= \left(\frac{E_n}{3}\right)^{-2}
\sum_{j,j'\in[3]}16\pi^4 \frac{1}{\Nn^2}\sum_{\lambda, \lambda'\in \Lambda_n} \lambda_j \lambda_{j'} \lambda'_j \lambda'_{j'}\int_{\T } e_{\lambda+\lambda'}(z)dz\\
&=& \left(\frac{E_n}{3}\right)^{-2}16\pi^4 \frac{1}{\Nn^2} \sum_{j,j'\in[3]}\sum_{\lambda \in \Lambda_n} \lambda_j^2 \lambda_{j'}^2 
= \frac{9}{n^2\Nn^2}\sum_{j,j'\in[3]}\sum_{\lambda \in \Lambda_n} \lambda_j^2 \lambda_{j'}^2 \\
&=& \frac{9}{\Nn}\frac{1}{n^2\Nn}\sum_{j,j'\in[3]}\sum_{\lambda \in \Lambda_n} \lambda_j^2 \lambda_{j'}^2 .
\end{eqnarray*}
Now, using the relation (see e.g. \cite[Appendix C]{Cam17})
\begin{gather*}
\frac{1}{n^2\Nn}\sum_{\lambda \in \Lambda_n} \lambda_j^2 \lambda_{j'}^2 = \frac{1}{5}\ind{j=j'} + \frac{1}{15}\ind{j\neq j'} + O(n^{-1/28+o(1)})
\end{gather*}
gives 
\begin{gather*}
\int_{\T} \Tr(R_n(z)^2) dz =
\frac{9}{\Nn}   \left(\frac{3}{5}+\frac{6}{15} +  O(n^{-1/28+o(1)})\right)=
\frac{9}{\Nn}\left(1+ O(n^{-1/28+o(1)})\right), 
\end{gather*}
so that, computing $\widehat{\Phi}((1))=2^{\ell/2}\frac{\Gamma_{\ell}(2)}{\Gamma_{\ell}(\frac{3}{2})}$ from \eqref{ForId} gives by \eqref{Var2TOT}
\begin{eqnarray*}
\V{\mathbf{V}(\bT_n^{(\ell)};\T)[2]}
&=& \left(\frac{E_n}{3}\right)^{\ell} \widehat{\Phi}((1))^2
\frac{\ell}{18}\frac{9}{\Nn}\left(1+ O(n^{-1/28+o(1)})\right) \\
&=&\left(\frac{E_n}{3}\right)^{\ell} 2^{\ell}  \frac{\Gamma_{\ell}(2)^2}{\Gamma_{\ell}(\frac{3}{2})^2}
\frac{\ell}{2\Nn}\left(1  + O(n^{-1/28+o(1)})\right),
\end{eqnarray*}
which finishes the proof.
\end{proof}

\noindent\underline{\textit{Higher-order chaotic components.}}
The goal of this part is to prove the following
statement, dealing with the variance of the tail of the Wiener chaos expansion of $\mathbf{V}(\bT_n^{(\ell)};\T)$. 
\begin{Prop}\label{Higher}
As $n\to \infty,\notcon{n}{0,4,7}{8}$, we have 
\begin{gather}\label{Tail}
\V{\sum_{k \geq 2} \mathbf{V} (\bT_n^{(\ell)}; \T)[2k]} 
= o\left(\V{\mathbf{V}(\bT_n^{(\ell)};\T)[2]}\right) .
\end{gather}
In particular, as $ n \to \infty, \notcon{n}{0,4,7}{8}$,  
\begin{gather}\label{dom2}
\mathbf{V}(\bT_n^{(\ell)};\T) = 
\mathbf{V}(\bT_n^{(\ell)};\T)[2]+o_{\Prob}(1),
\end{gather}
where $o_{\Prob}(1)$ denotes a sequence of random variables converging to zero in probability, that is, in the high-energy regime, the random variable $\mathbf{V}(\bT_n^{(\ell)};\T)$ is dominated in the $L^2(\Prob)$-sense by its projection on the second Wiener chaos.
\end{Prop}
The proof of Proposition \ref{Higher} is based on a suitable partition of the torus into \textit{singular} and \textit{non-singular} pairs of cubes, as introduced in \cite{ORW08} (see also e.g. \cite{DNPR16,PR16,Not20} for further references using this approach). 

\medskip
We now describe this partition. 
For every $n \in S_3$, we partition the torus into a disjoint union of cubes of length $1/M$, where $M=M_n\geq1$ is an integer proportional to $\sqrt{E_n}$, as follows: Let $Q_0 = [0,1/M)^3$; then we consider the partition of $\T$ obtained by translating $Q_0$ in the directions $k/M,k \in \Z^3$. Denote by $\mathcal{P}(M)$ the partition of $\T$ that is obtained in this way. By construction, we have that $\textrm{card}({\mathcal{P}(M)})=M^3$.
Let us now denote by
\begin{eqnarray*}
\mathbf{V}(\bT_n^{(\ell)};\T)[4^+]:= \sum_{k\geq 2} \mathbf{V}(\bT_n^{(\ell)};\T)[2k]
\end{eqnarray*}
the projection of $\mathbf{V}(\bT_n^{(\ell)};\T)$ onto chaoses of order at least $4$. 
By linearity, we can write
\begin{equation}\label{decomp}
\mathbf{V}(\bT_n^{(\ell)};\T)[4^+] = \sum_{Q \in \mathcal{P}(M)} \mathbf{V}(\bT_n^{(\ell)};Q)[4^+] \ , \quad \ell \in [3]
\end{equation}
where $ \mathbf{V}(\bT_n^{(\ell)};Q)$ denotes the total variation of $\bT_n^{(\ell)}$ in the cube $Q$. 
From now on, we fix a small number $0 < \eta < 10^{- 10}$. In the forthcoming definition, we  define singular pairs of points and cubes. Recall the notations 
\begin{eqnarray*}
r_i^{(n)}(z):= \frac{\partial}{\partial z_i} r^{(n)}(z), \quad 
r^{(n)}_{i,j}(z):= \frac{\partial^2}{\partial z_i \partial z_j}r^{(n)}(z), \quad (i,j)\in [3]\times [3].
\end{eqnarray*}
\begin{Def}[Singular pairs of points and cubes]\label{DefSing}
A pair of points $(z,z') \in \T \times \T $ is called a \textit{singular pair of points} if one of the following inequalities is satisfied:
\begin{equation*}
|r^{(n)}(z-z')| > \eta \ ,  \quad
| r^{(n)}_i(z-z')| > \eta \sqrt{E_n/3} \ ,  \quad
|  r^{(n)}_{i,j} (z-z')| > \eta E_n/3 \end{equation*}
for $(i, j) \in [3]\times[3]$.
A pair of cubes $(Q, Q') \in \mathcal{P}(M)^2 $ is called a \textit{singular pair of cubes} if the product $Q\times Q'$ contains a singular pair of points. We denote by $\mathcal{S}=\mathcal{S}(M) \subset \mathcal{P}(M)^2$ the set of singular pairs of cubes. 
A pair of cubes $(Q,Q') \in \mathcal{S}^c$ is called \textit{non-singular}. By construction, $\mathcal{P}(M)^2 = \mathcal{S} \cup \mathcal{S}^c$. 
\end{Def}
For fixed $Q \in \mathcal{P}(M)$, let us furthermore denote by $\mathcal{B}_Q$ the union  over all cubes $Q' \in \mathcal{P}(M)$ such
that $(Q,Q') \in \mathcal{S}$. Arguing
as in  \cite[Lemma 6.3]{DNPR16}, we have that
\begin{equation}\label{LebBQ}
\Leb(\mathcal{B}_Q) = O(\mathcal{R}_n(6)) ,
\end{equation} 
where $\mathcal{R}_n(6)=\int_{\T} [r^{(n)}(z)]^6 dz$.
In view of \eqref{decomp}, we can thus split the variance into its singular and non-singular contribution as follows
\begin{gather*}
\V{\mathbf{V}(\bT_n^{(\ell)};\T)[4^+]}
= \left\{\sum_{(Q,Q') \in \mathcal{S} }
+ \sum_{(Q,Q') \in \mathcal{S}^c }\right\}
\E{\mathbf{V}(\bT_n^{(\ell)};Q)[4^+] \cdot \mathbf{V}(\bT_n^{(\ell)};Q')[4^+]}
:= \Delta_{n,1}^{(\ell)}
+ \Delta_{n,2}^{(\ell)}.
\end{gather*}
The contributions to the variance of the terms  $\Delta_{n,j}^{(\ell)}, j=1,2$ are given in Lemma \ref{SingLemma} and \ref{NonSingLemma} below. The combination of both results proves Proposition \ref{Higher}.
\begin{Lem}[Singular part]\label{SingLemma}
As $n \to \infty,  \notcon{n}{0,4,7}{8}$, we have that
\begin{eqnarray*}
\left| \Delta_{n,1}^{(\ell)}\right| 
= o\left(\V{\mathbf{V}(\bT_n^{(\ell)};\T)[2]}\right).
\end{eqnarray*}
\end{Lem}
\begin{proof}
Using the triangle inequality, the Cauchy-Schwarz inequality, and \eqref{LebBQ}, we can write
\begin{eqnarray}\label{EstOne}
\left|\Delta_{n,1}^{(\ell)}\right|
&\leq&  \sum_{(Q,Q')\in \mathcal{S}}
\sqrt{\V{\mathbf{V}(\bT_n^{(\ell)};Q)[4^+]}}  
\sqrt{\V{\mathbf{V}(\bT_n^{(\ell)};Q')[4^+]}} \notag \\
&\ll& E_n^3 \mathcal{R}_n(6)
\cdot \V{\mathbf{V}(\bT_n^{(\ell)};Q_0)[4^+]}  , 
\end{eqnarray}
where we exploited stationarity of $\bT_n^{(\ell)}$ and where $Q_0$ denotes the cube around the origin. Now we notice that  
\begin{equation*}
\V{\mathbf{V}(\bT_n^{(\ell)};Q_0)[4^+]}  \leq \V{\mathbf{V}(\bT_n^{(\ell)};Q_0)} 
\leq \E{\mathbf{V}(\bT_n^{(\ell)};Q_0)^2} \ .
\end{equation*}
By definition of the total variation \eqref{totdef}, we can write 
\begin{eqnarray*}
\E{\mathbf{V}(\bT_n^{(\ell)};Q_0)^2}
= \left(\frac{E_n}{3}\right)^{\ell}
\int_{Q_0\times Q_0} \E{ \Phi( \dot{\bT}_n^{(\ell)}(z)  ) \Phi( \dot{\bT}_n^{(\ell)}(z')  ) 
} dzdz'.
\end{eqnarray*}
Now for every fixed $z,z'\in Q_0$, we have by the Cauchy-Schwarz inequality 
\begin{eqnarray*}
\E{ \Phi( \dot{\bT}_n^{(\ell)}(z)  ) \Phi( \dot{\bT}_n^{(\ell)}(z')  ) 
}
\leq \sqrt{\E{ \Phi( \dot{\bT}_n^{(\ell)}(z) )^2} \E{ \Phi( \dot{\bT}_n^{(\ell)}(z')  )^2
}} 
=\E{\Phi(\mathbf{N})^2}= O(1),
\end{eqnarray*}
where $\mathbf{N} \eqLaw \mathscr{N}_{\ell \times 3}(0, \Id_{\ell}\otimes \Id_3)$. Therefore, bearing in mind that $\mathrm{Leb}(Q_0)=M^{-3}=O(E_n^{-3/2})$,  it follows that 
\begin{eqnarray*}
\E{\mathbf{V}(\bT_n^{(\ell)};Q_0)^2}
= O\left(E_n^{\ell}M^{-6}\right)
= O(E_n^{\ell-3}).
\end{eqnarray*} 
Combining this with the estimate in \eqref{EstOne} yields 
$\left|\Delta_{n,1}^{(\ell)}\right| \ll E_n^3 \mathcal{R}_n(6) E_n^{\ell-3}  \ll E_n^{\ell} \mathcal{R}_n(6)$. By \cite[Eq.(1.18)]{BM17}, we have that $\mathcal{R}_n(6) \ll \Nn^{-7/3+o(1)}$, as $ n\to\infty ,\notcon{n}{0,4,7}{8}$. Combining this with the estimate in Lemma \ref{Var2} yields the desired conclusion. 
\end{proof}

\begin{Lem}[Non-singular part]\label{NonSingLemma}
As $n \to \infty,  \notcon{n}{0,4,7}{8}$, we have that
\begin{eqnarray*}
\left| \Delta_{n,2}^{(\ell)}\right| 
= o\left(\V{\mathbf{V}(\bT_n^{(\ell)};\T)[2]}\right).
\end{eqnarray*}
\end{Lem}
\begin{proof}
Using the expansion in matrix-Hermite polynomials and arguing as in the proof of Proposition \ref{VarTV}, we have that 
\begin{eqnarray*}
\left| \Delta_{n,2}^{(\ell)}\right|  
\leq 
E_n^{\ell}
\sum_{k\geq 2}
\sum_{\kappa \vdash k}
\widehat{\Phi}(\kappa)^2 4^{-k} 
\left(\frac{n}{2}\right)^{-1}_{\kappa}
k!
\frac{C_{\kappa}(\Id_{\ell})}{C_{\kappa}(\Id_3)}
\sum_{(Q,Q')\in \mathcal{S}^c}
\int_{Q\times Q'}
\left| C_{\kappa}(R_n(x-y)^2) \right|
dxdy.
\end{eqnarray*}
Now for a matrix $S \in \C^{m\times m}$, we denote by $\rho(S):=\max(|\lambda_1|,\ldots,|\lambda_m|)$, where $\lambda_i$ denote the eigenvalues of $S$. 
We now use the following two facts: (i) For every  partition $\kappa\vdash k$,   every matrix $S \in \C^{m\times m}$ and every $x$ such that $\rho(S) < x$, one has that $\left|C_{\kappa}(S)\right| \leq x^k C_{\kappa}(\Id_m)$ (see for instance \cite{BFZP}, p.197) and (ii) by  Gerschgorin's Theorem (see for instance \cite{GR:14}, p.1084), writing $S=(s_{ij})$,  
\begin{eqnarray*}
\rho(S) \leq \min\left(
\max_{i=1,\ldots,m} \sum_{j=1}^m |s_{ij}|, 
\max_{j=1,\ldots,m}
\sum_{i=1}^m |s_{ij}|
\right)=: \tilde{\rho}(S).
\end{eqnarray*}
Applying the  facts above with the symmetric matrix $S= R_n(x-y)^2$ and 
$x = 2\tilde{\rho}(S)$ yields 
\begin{eqnarray}\label{BoundB1}
\left| \Delta_{n,2}^{(\ell)}\right|  
\leq 
E_n^{\ell}
\sum_{k\geq 2}
\sum_{\kappa \vdash k}
\widehat{\Phi}(\kappa)^2 4^{-k} 
\left(\frac{n}{2}\right)^{-1}_{\kappa}
k!
C_{\kappa}(\Id_{\ell})
\sum_{(Q,Q')\in \mathcal{S}^c}
\int_{Q\times Q'}
\left(2 \tilde{\rho}(R_n(x-y)^2) \right)^{k} dxdy.
\end{eqnarray}
By definition of $\tilde{\rho}$ and the triangular inequality, we have that 
\begin{eqnarray}\label{tilderho}
2\tilde{\rho}(R_n(x-y)^2)
\leq 2\max_{i=1,2,3}
\sum_{j,l=1}^3
\left| \tilde{r}^{(n)}_{il}(x-y)\right| 
\left| \tilde{r}^{(n)}_{lj}(x-y)\right|, 
\end{eqnarray}
which is bounded by $18\eta^2<1$ on the non-singular regions. 
Combining this with the fact that we are summing over integers 
 $k\geq 2$ yields
\begin{eqnarray*}
&&\sum_{(Q,Q')\in \mathcal{S}^c}
\int_{Q\times Q'}
\left(2 \tilde{\rho}(R_n(x-y)^2) \right)^{k} dxdy\\
&=& \sum_{(Q,Q')\in \mathcal{S}^c}
\int_{Q\times Q'}
\left(2 \tilde{\rho}(R_n(x-y)^2) \right)^{k-2}
\left(2 \tilde{\rho}(R_n(x-y)^2) \right)^{2}
 dxdy \\
 &\ll&  \int_{\T}
\left( \tilde{\rho}(R_n(z)^2) \right)^{2}
 dz. \\
\end{eqnarray*}
Combining  \eqref{tilderho} with the Cauchy-Schwarz inequality and the estimate 
\begin{gather*}
\int_{\T} [\tilde{r}^{(n)}_{j,j'}(z)]^{2p} dz = O( \mathcal{R}_n(2p)), \quad p\geq1,
\end{gather*}
where the constant involved in the 'big-O' notation depends only on $p$ (see for instance \cite[Lemma E.2]{Not20}), we deduce that 
\begin{gather*}
\int_{\T}
\left( \tilde{\rho}(R_n(z)^2) \right)^{2}
 dz
 \ll \int_{\T}\max_{i,l=1,2,3} \left| \tilde{r}^{(n)}_{il}(z)\right|^2  
 \max_{j,l=1,2,3} \left| \tilde{r}^{(n)}_{jl}(z)\right|^2 dz 
 = \int_{\T} \max_{j,l=1,2,3} \left| \tilde{r}^{(n)}_{jl}(z)\right|^4 dz \ll \mathcal{R}_n(4).
\end{gather*}
Therefore, in view of the estimate \eqref{BoundB1}, 
and the fact that $\Phi \in L^2(\mu_X)$ for $X \sim\mathscr{N}_{\ell\times n}(0,\Id_{\ell}\otimes \Id_n)$, we conclude that 
\begin{gather*}
\left| \Delta_{n,2}^{(\ell)}\right|  
\ll  
E_n^{\ell} \mathcal{R}_n(4)
\sum_{k\geq 2}
\sum_{\kappa \vdash k}
\widehat{\Phi}(\kappa)^2 4^{-k} 
\left(\frac{n}{2}\right)^{-1}_{\kappa}
k!
C_{\kappa}(\Id_{\ell}) 
\leq E_n^{\ell} \mathcal{R}_n(4) \E{\Phi(X)^2} 
\ll E_n^{\ell} \mathcal{R}_n(4).
\end{gather*} 
Now, the Cauchy-Schwarz inequality implies that
\begin{gather*}
\mathcal{R}_n(4) = \int_{\T} [r^{(n)}(z)]^4 dz
\leq \left( \int_{\T} [r^{(n)}(z)]^2 dz 
\int_{\T} [r^{(n)}(z)]^6 dz 
\right)^{1/2} = \sqrt{\mathcal{R}_n(2) \mathcal{R}_n(6)}.
\end{gather*}
Using the estimates (see \cite[Eq.(1.16) and (1.18)]{BM17})
\begin{gather*}
\mathcal{R}_n(2) = \frac{1}{\Nn} , \quad 
\mathcal{R}_n(6) \ll \Nn^{-7/3+o(1)} , \quad n\to\infty ,\notcon{n}{0,4,7}{8}
\end{gather*}
implies that $\sqrt{\mathcal{R}_n(2)\mathcal{R}_n(6)} \ll \Nn^{-5/3+o(1)}$. The fact that $E_n^{\ell}\Nn^{-5/3+o(1)}=o\left(\V{\mathbf{V}(\bT_n^{(\ell)},\T)}\right)$ follows from the order of the variance of the second Wiener chaos in Lemma \ref{Var2}.
 \end{proof}

\medskip
\noindent\underline{\textit{Limiting distribution of the normalised total variation.}}
The next proposition establishes a CLT in the high-frequency regime for normalised version of the second chaotic component of the total variation
\begin{gather}\label{nor2}
\widehat{\mathbf{V}}(\bT_n^{(\ell)};\T)[2]:=\frac{\mathbf{V}(\bT_n^{(\ell)};\T)[2]}{\V{\mathbf{V}(\bT_n^{(\ell)};\T)[2]}^{1/2}}.
\end{gather}
\begin{Prop}\label{Lim2}
As $n\to \infty, \notcon{n}{0,4,7}{8}$, we have 
\begin{gather*}
\widehat{\mathbf{V}}(\bT_n^{(\ell)};\T)[2]  \Law \mathscr{N}(0,1).
\end{gather*}
\end{Prop}
\begin{proof}
The expression of $\mathbf{V}(\bT_n^{(\ell)};\T)[2]$ is given in \eqref{WC2}. 
Normalising that expression by the square root of the order of the variance in \eqref{Var} yields
\begin{gather*}
\widehat{\mathbf{V}}(\bT_n^{(\ell)};\T)[2]
= \frac{\sqrt{2}\sqrt{\Nn}}{\sqrt{\ell}}  \int_{\T} H_{(1)}^{(\ell,3)}(\dot{\bT}_n^{(\ell)}(z)) dz.
\end{gather*}
Using \eqref{HerPol}, we can rewrite
\begin{gather*}
H_{(1)}^{(\ell,3)}(\dot{\bT}_n^{(\ell)}(z)) =  \frac{1}{6}\sum_{k=1}^{\ell} \sum_{j=1}^{3} H_2(\tilde{\partial}_jT_n^{(k)}(z)).
\end{gather*}
Now, for every $k\in [\ell]$, we
  write $H_2(u)=u^2-1$ and
 exploit once more the orthogonality relations for complex exponentials on the torus \eqref{ortoexp} in order to write 
\begin{eqnarray*}
&&\sum_{j=1}^3\int_{\T} H_2(\tilde{\partial}_jT_n^{(k)}) dz= 
\sum_{j=1}^3\int_{\T}\left[\left(\tilde{\partial}_jT_n^{(k)}(z)\right)^2 -1\right] dz \\
&=&\sum_{j=1}^3\frac{3}{n\Nn} \sum_{\lambda\in\Lambda_n}  \lambda_j^2(|a_{k,\lambda}|^2-1)
= \frac{3}{\Nn}\sum_{\lambda \in \Lambda_n}(|a_{k,\lambda}|^2-1) , 
\end{eqnarray*}
where we used that $\lambda_1^2+\lambda_2^2+\lambda_3^2=n$, so that 
\begin{eqnarray} 
\widehat{\mathbf{V}}(\bT_n^{(\ell)};\T)[2]
&=& \frac{\sqrt{2}\sqrt{\Nn}}{\sqrt{\ell}}  \int_{\T} H_{(1)}^{(\ell,3)}(\dot{\bT}_n^{(\ell)}(z)) dz =
\frac{1}{\sqrt{\ell}} 
\frac{1}{\sqrt{2}}\sum_{k=1}^{\ell} \frac{1}{\sqrt{\Nn}}\sum_{\lambda \in \Lambda_n}(|a_{k,\lambda}|^2-1) \notag \\ 
&=& \frac{1}{\sqrt{\ell}} \sum_{k=1}^{\ell}
\frac{1}{\sqrt{\Nn/2}}\sum_{\lambda\in \Lambda_n/_{\sim}}(|a_{k,\lambda}|^2-1),\label{tildaV2}
\end{eqnarray}
where $\Lambda_n/_{\sim}$ stands the equivalence classes in $\Lambda_n$ obtained by identifying $\lambda$ with $-\lambda$, so that $|\Lambda_n/_{\sim}|=\Nn/2$. Note that the random variables $\{ |a_{k,\lambda}|^2-1: \lambda\in \Lambda_n/_{\sim}\}$ are i.i.d, centred and have unit variance. The classical CLT thus implies that, as $n \to \infty, \notcon{n}{0,4,7}{8}$, 
\begin{gather*}
\sum_{k=1}^{\ell}
\frac{1}{\sqrt{\Nn/2}}\sum_{\lambda\in \Lambda_n/_{\sim}}(|a_{k,\lambda}|^2-1) \Law \sum_{k=1}^{\ell}Z_k, 
\end{gather*}
where $(Z_1,\ldots,Z_{\ell})$ is a standard Gaussian vector. The statement then follows from \eqref{tildaV2}. 
\end{proof}

\noindent\underline{\textit{End of the proof of Theorem \ref{ARW}.}}
Relation \eqref{dom2} implies that the second chaotic component of the total variation dominates the Wiener chaos expansion in \eqref{chaosARW}. In particular,  \eqref{dom2} implies that, as $n \to \infty, \notcon{n}{0,4,7}{8}$
\begin{gather*}
\V{\mathbf{V}(\bT_n^{(\ell)};\T)} =
\V{\mathbf{V}(\bT_n^{(\ell)};\T)[2]}(1+o(1)),
\end{gather*}
and that the normalised sequences of random variables 
\begin{gather*}
\left\{\widehat{\mathbf{V}}(\bT_n^{(\ell)};\T) : n \in S_3\right\}  \ , \quad 
\left\{\widehat{\mathbf{V}}(\bT_n^{(\ell)};\T)[2] : n \in S_3\right\}
\end{gather*}
defined in  \eqref{nor} and \eqref{nor2} respectively, have the same limiting distribution.
Combining this with the asymptotic variance for the second chaos in Lemma \ref{Var2} proves the variance estimate in \eqref{Var}.
Finally, the CLT in \eqref{CLT} for the total variation follows when combining \eqref{dom2} with the content of Proposition \ref{Lim2}. This concludes the proof of Theorem \ref{ARW}.

\appendix 
\section{Proof of Proposition \ref{Prop:Complete}}\label{App:Complete}
We recall that $L^2(\mu_X)=L^2(\Omega,\sigma(\mu_X),\Prob)$, where $\sigma(\mu_X)$ is the $\sigma$-field generated by   random variables of the form 
\begin{eqnarray*}
 \int_{\R} f(x) \mu_X(dx)
 \end{eqnarray*}
where $f$ is a finite linear combination of trigonometric functions of the form $\cos(ax), \sin(bx)$ with $a,b\in \R$. Since $f(x)$ is equal to the  limit  of its Taylor expansion for every $x\in \R$, and since the support of $\mu_X$ consists of at most $\ell$ points, we deduce that  $\sigma(\mu_X)$ is generated by random variables as above where $f=p$ is a polynomial. 
Our goal is now to prove that, if $F \in L^2(\mu_X)$ is such that $\E{F H_{\kappa}^{(\ell,n)}(X)} = 0$ for every $\kappa\vdash k $ and every $k\geq 0$, then $F=0$, $\Prob$-almost everywhere. In order to obtain the desired conclusion, we will use the following three facts: (i) zonal polynomials can be expanded into a finite linear combination of matrix-Hermite polynomials (see e.g \cite[Eq.(4.12)]{Ch92}), (ii) the product of finitely many zonal polynomials is a finite linear combination of zonal polynomials (see \eqref{akt}) and (iii) every monomial of the form $t_k(X):=\Tr([XX^T]^k)=s_1^{k}+\ldots+s_{\ell}^k$ (where $s_1,\ldots,s_{\ell}$ denote the eigenvalues of $XX^T$) can be represented as a linear combination of  zonal polynomials (see \eqref{Eq:Pol}). Using these three facts shows that, whenever $F$ satisfies the assumption above, then $\E{F t_{j_0}(X)^{a_0}\ldots t_{j_M}(X)^{a_M}}=0$ for every finite $M\geq 1$ and every collection of integers $j_0,\ldots,j_M\geq 0$ and $a_0,\ldots,a_M\geq 0$. In particular, writing $p(x) = \sum_{j=0}^M c_j x^j$ for a polynomial of degree $M$, one has that 
\begin{eqnarray*}
&&\E{F \exp\left(i \int_{\R} p(x) \mu_X(dx)\right)}
= \E{F \exp\left(i \sum_{j=0}^M c_j(s_1^j + \ldots+ s_{\ell}^j)\right)}
=\E{F \exp\left(i \sum_{j=0}^{M} c_j t_j(X)\right)}\\
&&= \E{F  \prod_{j=0}^M \exp\left(i c_j t_j(X)\right)}
=  \sum_{a_0,\ldots,a_M\geq 0} \frac{(ic_0)^{a_0} \cdots (ic_M)^{a_M}}{a_0!\cdots a_M!}
\E{F t_{0}(X)^{a_0}\ldots t_M(X)^{a_M}} = 0
\end{eqnarray*}
by assumption. By a standard approximation argument, we therefore deduce that $\E{F|\sigma(\mu_X)} = 0$, yielding the desired conclusion. 
\bibliographystyle{alpha}
\bibliography{Refs.bib}
\end{document}